\documentclass[reqno]{amsart}

\usepackage{amssymb}
\usepackage{graphicx}
\usepackage{color}
\usepackage{subfigure}
\usepackage[all]{xy}
\usepackage{calc}
\usepackage{extarrows}
\usepackage{a4wide}
\usepackage{hyperref}
\hypersetup{colorlinks=true,linkcolor=blue}

\usepackage{mathtools}
\usepackage{soul}
\usepackage{marginnote}

\definecolor{gray}{gray}{0.5}

\theoremstyle{plain}
\newtheorem{thm}{Theorem}[section]
\newtheorem{lem}[thm]{Lemma}
\newtheorem{prop}[thm]{Proposition}
\newtheorem{cor}[thm]{Corollary}

\theoremstyle{definition}
\newtheorem{defn}[thm]{Definition}
\newtheorem{ex}[thm]{Example}

\theoremstyle{remark}
\newtheorem{rmk}[thm]{Remark}

\newcommand{\R}{\mathbb{R}}

\newcommand{\G}{\mathcal{G}}
\newcommand{\B}{\mathcal{B}}
\newcommand{\T}{\mathcal{T}}
\renewcommand{\L}{\mathcal{L}}
\newcommand{\K}{\mathcal{K}}

\newcommand{\SG}{\mathcal{SG}}
\newcommand{\SB}{\mathcal{SB}}
\newcommand{\ST}{\mathcal{ST}}
\newcommand{\SL}{\mathcal{SL}}
\newcommand{\SK}{\mathcal{SK}}
\newcommand{\SC}{\mathcal{SC}}

\renewcommand{\tilde}{\widetilde}
\renewcommand{\hat}{\widehat}
\renewcommand{\bar}{\overline}

\begin{document}
\title{Grid diagram for singular links}
\author[B.~H.~An]{Byung Hee An}
\address{Center for Geometry and Physics, Institute for Basic Science (IBS), Pohang 37673, Republic of Korea}
\email{anbyhee@ibs.re.kr}
\thanks{The first author was supported by IBS-R003-D1.}

\author[H.~J.~Lee]{Hwa Jeong Lee}
\address{School of Undergraduate Studies, College of Transdisciplinary Studies, DGIST, Daegu 42988, Republic of Korea}
\email{hjwith@dgist.ac.kr}
\thanks{The second author was supported by the National Research Foundation of Korea(NRF) grant funded by the Korea government(MSIP : Ministry of Science, ICT $\&$ Future Planning) (No. 2015R1C1A2A01054607)}


\date{\today}

\begin{abstract}
In this paper, we define the set of singular grid diagrams $\SG$
which provides a unified description for singular links, singular Legendrian links, singular transverse links, and singular braids.
We also classify the complete set of all equivalence relations on $\SG$ which induce the bijection onto each singular object. 
This is an extension of the known result of Ng-Thurston \cite{NT} for non-singular links and braids.
\end{abstract}
\maketitle
\tableofcontents

\section{Introduction}
\subsection{Grid diagrams and a unified description}
A {\em grid diagram} of size $n$ is an oriented link diagram which consists only of $n$ vertical and $n$ horizontal line segments in such a way that at each crossing the vertical line segment crosses over the horizontal line segment and no two line segments are colinear. 
In short, a grid diagram of size $n$ is an $n\times n$ matrix of $8$ kinds of the following symbols, called {\em grid tiles}, representing a link such that no more than two corners exist in any vertical and horizontal line.

 \begin{figure}[ht]
  \begin{center}
    \includegraphics[scale=1]{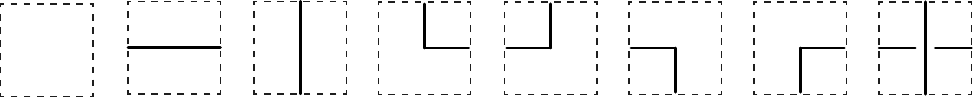}
  \end{center}
  \vspace{-3mm}
  \caption{Grid tiles of a grid diagram}\label{fig:local_grid}
 \end{figure}

Due to Cromwell~\cite{C}, a grid diagram is noted as an arc presentation of a link, which is defined as an embedding of the link in finitely many pages of the open-book decomposition so that the link meets each page in a single simple arc. In the paper, he described combinatorial transformations of grid diagrams which do not change topological knot type. These are {\em translations}, {\em commutations}, and {\em (de)stabilizations} and are called {\em elementary moves}. See Figures~\ref{fig:trans}, \ref{fig:commu} and \ref{fig:stabil}.
In~\cite{D}, Dynnikov proved that the decomposition problem of arc presentation is solvable by monotonic simplification using elementary moves. This  is a remarkable result in knot theory of which the most important problem is the classification of knots and links.
Grid diagrams became more popular in recent years due to a connection with Legendrian links~\cite{NT, OST} and giving a combinatorial description of knot Floer homology that is a knot invariant defined in terms of Heegaard Floer homology~\cite{MOS,MOST}.

\begin{figure}[ht]
\[
\xymatrix{
\vcenter{\hbox{\includegraphics{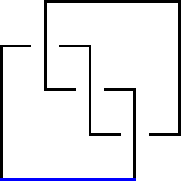}}} \ar@{<->}@[blue][r]^-{(a)}\quad&\quad
\vcenter{\hbox{\includegraphics{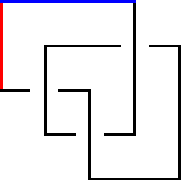}}} \ar@{<->}@[red][r]^-{(b)}\quad&\quad
\vcenter{\hbox{\includegraphics{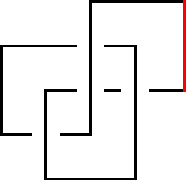}}}
}
\]
\caption{Translations mean cyclic permutations of $(a)$ horizontal or $(b)$ vertical edges, respectively}\label{fig:trans}
\end{figure}

\begin{figure}[ht]
\[
\xymatrix{
\vcenter{\hbox{\includegraphics{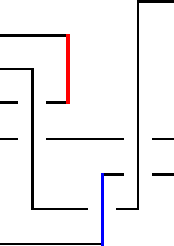}}} \ar@{<->}[r]\quad&\quad
\vcenter{\hbox{\includegraphics{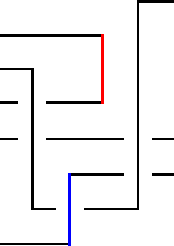}}} &
\vcenter{\hbox{\includegraphics{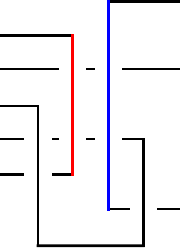}}} \ar@{<->}[r]\quad&\quad
\vcenter{\hbox{\includegraphics{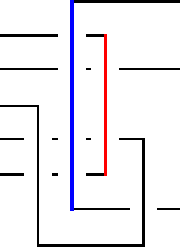}}}
}
\]
\caption{Commutation is interchanging adjacent two non-interleaved vertical edges or horizontal edges, respectively.}\label{fig:commu}
\end{figure}

\begin{figure}[ht]
\begin{center}
\[
\xymatrix@C=5pc@R=.5pc{
\vcenter{\hbox{\includegraphics{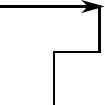}}} \quad
\ar@<-.5ex>@{<-}[rd]_*[@]!/^3pt/{\labelstyle(NW)}
\ar@<.5ex>[rd]^*[@]!/_7pt/{\labelstyle(NW)^{-1}}
& & \quad\vcenter{\hbox{\includegraphics{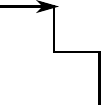}}}\\
 & \quad\vcenter{\hbox{\includegraphics{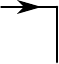}}}\quad
\ar@<.5ex>[ru]^*[@]!/_3pt/{\labelstyle(NE)}
\ar@<-.5ex>@{<-}[ru]_*[@]!/^7pt/{\labelstyle(NE)^{-1}}
\ar@<-.5ex>[rd]_*[@]!/^3pt/{\labelstyle(SE)}
\ar@<.5ex>@{<-}[rd]^*[@]!/_7pt/{\labelstyle(SE)^{-1}}
 & \\
\vcenter{\hbox{\includegraphics{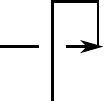}}} \quad
\ar@<.5ex>@{<-}[ru]^*[@]!/_3pt/{\labelstyle(SW)}
\ar@<-.5ex>[ru]_*[@]!/^7pt/{\labelstyle(SW)^{-1}}
& & \quad \vcenter{\hbox{\includegraphics{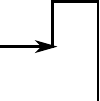}}}
}
\]
\end{center}
\caption{(De)stabilization is a local move which changes the grid size by one. }\label{fig:stabil}
\end{figure}

\medskip

We denote $\G, \B, \L, \T$, and $\K$ as follows:

\begin{itemize}
\item[-] $\G$ : the set of all grid diagrams.
\item[-] $\B$ : the set of all equivalent classes of braids on $D^2$ modulo conjugation and exchange move.
\item[-] $\L$ : the set of all equivalent classes of Legendrian links in $(\R^3, \xi_0)$, where $\xi_0=\ker(dz-ydx)$ is the {\em standard contact structure}.
\item[-] $\T$ : the set of all equivalent classes of {\em positively oriented} transverse links in $(\R^3, \xi_0)$.
\item[-] $\K$ : the set of all equivalent classes of smooth links in $\R^3$.
\end{itemize}

\medskip

This work is motivated by the construction in Ozsv\'ath-Szab\'o-Thurston~\cite{OST} and Ng-Thurston~\cite{NT} of maps between $\G, \B, \L, \T$, and $\K$. 

Figure~\ref{fig:gridsquareexample} shows how a grid diagram leads to a braid, a Legendrian link and a transverse link. In the figure, the braid $(a)$ is obtained by flipping horizontal segments of the grid diagram going from right to left. The front projection of Legendrian link $(b)$ is naturally transformed from the grid diagram by rotating the diagram of 45 degrees counterclockwise and then smoothing up and down corners and turning right and left corners into cusps.  Transverse links can be obtained by positive push-off of Legendrian links using the rule of Figure~\ref{fig:pushoff}. 
Using rotationally symmetric contact structure $\ker(dz -ydx + xdy)$ \cite{Ben}, closed braids can be transformed into transverse links. Then the unit circle in the $xy$-plane is the transverse knot. 

\begin{figure}[ht]
\[
\xymatrix{
\qquad\vcenter{\hbox{\includegraphics{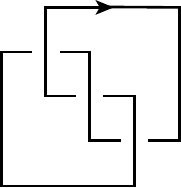}}} \ar[r]\ar[d] \qquad&\quad (b)\quad\vcenter{\hbox{\includegraphics{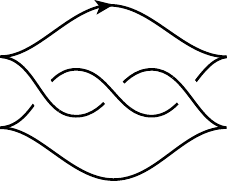}}}\quad \ar@<2.5ex>[d]\\
(a)\quad\vcenter{\hbox{\includegraphics{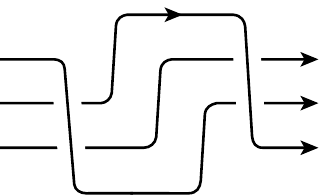}}} \ar[r] \quad&\quad (c)\quad\vcenter{\hbox{\includegraphics{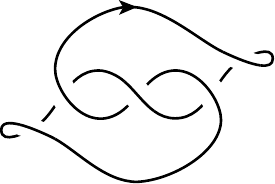}}}
}
\]
\caption{Grid diagrams are very closely related to $(a)$braids, $(b)$the front projection of Legendrian links and $(c)$transverse links}\label{fig:gridsquareexample}
\end{figure}

Khandhawit and Ng~\cite{KN} proved the commutativity of the maps as in Figure~\ref{fig:maps} and
Ozsv\'ath-Szab\'o-Thurston~\cite{OST} and Ng-Thurston~\cite{NT} showed that those maps induce bijections. See Proposition~\ref{prop:COSTNT}. It means that the maps can be understood in grid diagrams.
In the papers, they distinguished (de)stabilizations into four types, $(NW), (SW), (SE)$ and $(NE)$, which are exemplified in Figure~\ref{fig:stabil}. 

\begin{prop}\cite{C,NT,OST}\label{prop:COSTNT}
The grid diagram $\G$ gives a {\em unified description} for $\B, \L, \T$ and $\K$ in the following sense: 
There are bijections induced by the canonical maps
\begin{align*}
    \B & \longleftrightarrow \tilde{\G} / \{(NE),(SE)\}, \\
    \L & \longleftrightarrow \tilde{\G} / \{(NE),(SW)\}, \\
    \T & \longleftrightarrow \tilde{\G} / \{(NE),(SW),(SE)\}, \\
    \K & \longleftrightarrow \tilde{\G} / \{(NE),(SW),(SE),(NW)\},
\end{align*}
where $\tilde{\G}$ is the quotient of $\G$ by translations and commutations. 
\end{prop}

\subsection{Singular links and extended grid diagrams}
We extend scope of study of the maps in Figure~\ref{fig:maps} in terms of singular links as follows.
Let $\SK$ be the set of all equivalent classes of singular links in $S^3$.  We use naturally the following notations:
\begin{itemize}
\item[-] $\SB$ : the set of all equivalent classes of singular braids on $D^2$ modulo conjugation and exchange move.
\item[-] $\SL$ : the set of all equivalent classes of Legendrian singular links in $(S^3, \xi_{std})$.
\item[-] $\ST$ : the set of all equivalent classes of {\em positively oriented} transverse singular links in $(S^3, \xi_{std})$.
\end{itemize}
Then there are maps between $\SB, \SL, \ST$ and $\SK$, which correspond to solid lines in Figure~\ref{fig:e_maps} and extend maps between non-singular sets $\B, \L, \T$ and $\K$.
\begin{figure}[ht]
\subfigure[\label{fig:maps}]{\makebox[0.4\textwidth]{$
\xymatrix{
\G \ar[r]\ar[d]\ar[rd]\ar@/^4pc/[ddrr] & \L \ar[d]\ar[ddr]\\
\B\ar[r]\ar[rrd] & \T\ar[dr]\\
& & \K
}
$}}
\subfigure[\label{fig:e_maps}]{\makebox[0.4\textwidth]{$
\xymatrix{
\fbox{$\SG$} \ar@{-->}[r]\ar@{-->}[d]\ar@{-->}@/^4pc/[ddrr] \ar@{-->}[rd]
& \SL \ar[d]\ar[ddr]\\
\SB\ar[r]\ar[rrd] & \ST\ar[dr]\\
& & \SK
}
$}}
\caption{Commutative diagram (a) of $\G, \B, \L, \T$, and $\K$ and its extension (b) of $\SB, \SL, \ST$ and $\SK$}
\end{figure}
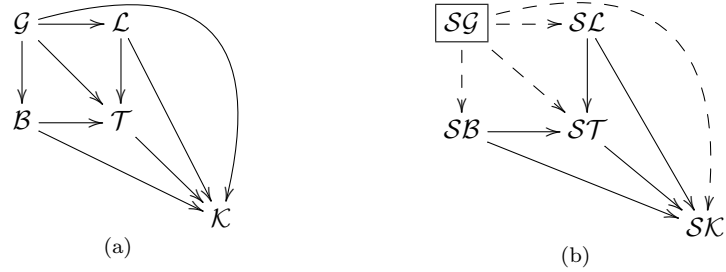

One of the important properties of the above maps  between singular sets $\SB$, $\SL$, $\ST$ and $\SK$ is that they commute with {\em resolutions}, which are the ways to resolve singular points.

In this paper, we want to extend grid diagrams $\G$ to {\em extended} grid diagrams $\SG$, which give us a {\em unified description} for $\SB$, $\SL$, $\ST$ and $\SK$ in the sense that not only the whole diagram in Figure~\ref{fig:e_maps} is commutative but also all maps in the figure commute with resolutions. 

Grid diagrams for singular links $\SK$ were already defined independentely by Audoux~\cite{Au}, Welji~\cite{Ws} and Harvey-O'Donnel~\cite{HO} in order to generalize link Floer homology to singular links or graphs. 
Their definitions are based on the {\em `$OX$' system} used in many literatures including \cite{MOS, MOST} and \cite{NT} to give a combinatorial description of link Floer homology.
Roughly speaking, the `$OX$' system on a rectangular board is a way indicating {\em corners} of the given link, whose orientation comes from the following rules: $O\to X$ in each row, and $X\to O$ in each column.

Audoux allows that a row or column may have two $O$'s and two $X$'s whose configuration is the same as $O\cdots O\cdots X\cdots X$, up to cyclic permutations, and interprets this configuration as two transversely intersecting arcs.
It is remarkable that a grid diagram of Audoux is a rectangle but not necessarily a square.

On the other hand, Welji considered singular points as special corners denoted by ${}^{X}{\!}_X$, which represents a singular point having two incoming horizontal arcs and two outgoing vertical arcs.
Harvey and O'Donnol use the {\em dual} convention, which is essentially the same as but opposite to Welji's. In other words, they introduced a special corner denoted by $O^*$ which represents a singular point with incoming vertical arcs and outgoing horizontal arcs. However, they generalized even more so that $O^*$ may have $k$ incoming vertical arcs and $\ell$ outgoing horizontal edges for any $k$ and $\ell\ge1$, and so $O^*$ becomes a vertex of spatial graph.

See Figure~\ref{fig:known models} for the comparison of the above models. One can see the similarity among these models, and regard Audoux's model as the one {\em between} other two.
Even though these models have their own meaning and are useful especially for the computation of the new invariant, there are still no canonical recipes for other variants of singular links, such as, singular Legendrian and transverse links and singular braids that we concern.

\begin{figure}[ht]
\subfigure[A pinched $3_1$]{\makebox[0.17\textwidth]{\includegraphics{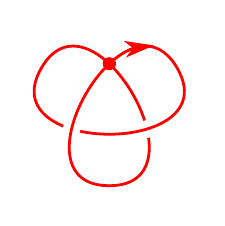}}}
\subfigure[Welji's model]{\makebox[0.24\textwidth]{\includegraphics{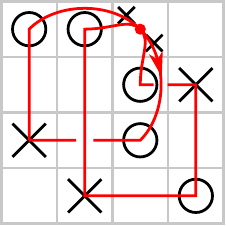}}}
\subfigure[Audoux's model]{\makebox[0.27\textwidth]{\includegraphics{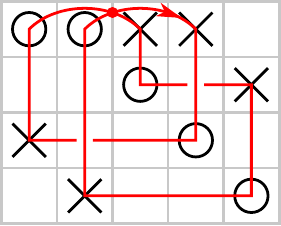}}}
\subfigure[Harvey-O'Donnol's model]{\makebox[0.26\textwidth]{\includegraphics{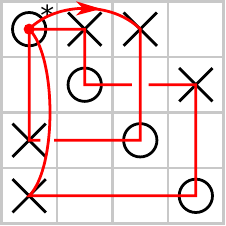}}}
\caption{Known models of grid diagrams for singular links in $\SK$}
\label{fig:known models}
\end{figure}

\subsection{Results}
We define grid diagrams for singular links using different shapes of tiles rather than symbols like $O$ or $X$. In addition to the tiles in Figure~\ref{fig:local_grid}, we need appropriate tiles representing singular points. Depending on the projection of a singular point we can naturally consider the following two tiles $t_\times$ and $t_\bullet$ with transverse intersection and non-transverse intersection near the singular point, respectively:

\[
t_\times=\vcenter{\hbox{\includegraphics[scale=0.7]{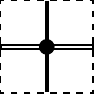}}}\quad\text{ and }\quad
t_\bullet=\vcenter{\hbox{\includegraphics[scale=0.7]{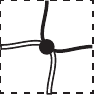}}}.
\]

Then we prove the following theorems:

\begin{thm}\label{thm:unified}
Let $\SG_\bullet$ be the set of grid diagrams of singular links extended by $t_\bullet$. Then $\SG_\bullet$ gives a unified description for $\SB,\SL,\ST$ and $\SK$.
\end{thm}

\begin{thm}\label{thm:notunified}
Let $\SG_\times$ be the set of grid diagrams of singular links extended by $t_\times$. 
Then $\SG_\times$ does not give a unified description whatever the maps onto $\SB$,$\SL$,$\ST$ and $\SK$ are defined.
\end{thm}

Therefore we may regard $\SG_\bullet$ as $\SG$ equipped with a unified description.
In $\SG$, we consider generalized elementary moves, $(\pm)$-rotations $(Rot_\pm)$, swirl $(Swirl)$ and flype $(Flype)$ which are described in Section~\ref{sec:proof_mainthm}. Then we prove that the diagram of Figure~\ref{fig:e_maps} induces the following bijections.

\begin{thm}[Main Theorem]\label{thm:maintheorem}
Let \,$\widetilde{\SG}$ denote the quotient set of $\SG$ by translations and commutations. The maps $\SG\to\SK$,\,$\SG \rightarrow \SB, \,\SG \rightarrow \SL$ and \,$\SG \rightarrow \ST$ induce bijections as follows:
 \begin{align*}
    \SB & \longleftrightarrow \tilde{\SG} /  \{(NE), (SE), (Flype), (Swirl), (Rot_\pm^*)\}, \\
    \SL & \longleftrightarrow \tilde{\SG} /  \{(NE), (SW), (Rot_\pm)\}, \\
    \ST & \longleftrightarrow \tilde{\SG} /  \{(NE), (SW), (SE), (Flype), (Rot_\pm)\}, \\
    \SK & \longleftrightarrow \tilde{\SG} /  \{(NE), (SW), (SE), (NW), (Flype), (Rot_\pm)\}. 
 \end{align*}
\end{thm}

The rest of the paper consists as follows. In Section~\ref{sec:singular links}, we briefly review the definition and equivalence relations on each singular variant, namely, $\SK, \SB, \SL$ and $\ST$. In Section~\ref{sec:extended_grid}, we define extended grid diagrams and prove Theorem~\ref{thm:unified} and Theorem~\ref{thm:notunified}. Finally, in Section~\ref{sec:proof_mainthm}, we give a proof of Theorem~\ref{thm:maintheorem}.

\section{Singular links and relatives}
\label{sec:singular links}
\subsection{Singular links}
A {\em singular link} $K$ of $n$ components is an immersion of a disjoint union $n S^1$ of $n$ oriented circles into $\R^3$ having only transverse double point singularities, called {\em singular points}. Then two singular links $K_1$ and $K_2$ are {\em equivalent} if and only if they are homotopic and the inverse images of double points vary continuously during the homotopy.
We denote the sets of all equivalent classes of nonsingular and singular links by $\K$ and $\SK$, respectively.

Let $\pi:\R^3\to\R^2_{xz}$ be a projection onto $xz$-plane. For a singular link $K\in\SK$, the projection $\pi(K)$ is {\em regular} if 
the composition $\pi\circ K$ is a smooth immersion in $\R^2$ without triple points.
Then it is easy to see that for any singular link $K\in\SK$, there exists $K'$ equivalent to $K$ such that $\pi(K')$ is regular.

A {\em regular diagram $D(K)$} is a regular projection $\pi(K)$ equipped with a crossing information at each double point. In other words, at each double point $c\in\R^2_{xz}$, one can determine which arc is lying over or two arcs make a singular point according to the $y$-coordinate.

Two singular links $K_1$ and $K_2$ having regular diagrams $D(K_1)$ and $D(K_2)$ are equivalent if and only if their diagrams differ by a finite sequence of the planar isotopy, and the classical Reidemeister moves $(RM1)\sim(RM3)$ \cite{R} and the singular Reidemeister moves $(RM4)$ and $(RM5)$ \cite{K} depicted in the Figure~\ref{fig:S_Rmoves} including their reflections about $x$ and $z$-axes.

\begin{figure}[ht]
\begin{align*}
\vcenter{\hbox{\includegraphics[scale=1]{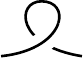}}}\quad\xleftrightarrow{(RM1)}\quad
\vcenter{\hbox{\includegraphics[scale=1]{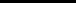}}}\quad\xleftrightarrow{(RM1)}\quad
\vcenter{\hbox{\includegraphics[scale=1]{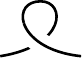}}}
\end{align*}
\begin{align*}
\vcenter{\hbox{\includegraphics[scale=1]{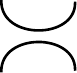}}}\quad&\xleftrightarrow{(RM2)}\quad
\vcenter{\hbox{\includegraphics[scale=1]{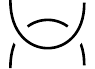}}}& 
\vcenter{\hbox{\includegraphics[scale=1]{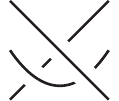}}}\quad&\xleftrightarrow{(RM3)}\quad
\vcenter{\hbox{\includegraphics[scale=1]{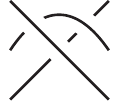}}}\\
\vcenter{\hbox{\includegraphics[scale=1]{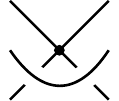}}}\quad&\xleftrightarrow{(RM4)}\quad
\vcenter{\hbox{\includegraphics[scale=1]{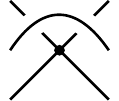}}}& 
\vcenter{\hbox{\includegraphics[scale=1]{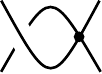}}}\quad&\xleftrightarrow{(RM5)}\quad
\vcenter{\hbox{\includegraphics[scale=1]{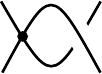}}}
\end{align*}
\caption{Classical and singular Reidemeister moves}\label{fig:S_Rmoves}
\end{figure}

\begin{defn}[Resolutions on $\SK$]
Let $K\in\SK$ and $p$ be a singular point of $K$. Then for each $\eta\in\{+,-,0\}$, we define the {\em $\eta$-resolution} $\mathcal{R}_\eta(K,p)$ at $p$ by the diagram replacement as follows.
\begin{align*}
\mathcal{R}_+\left(\vcenter{\hbox{\includegraphics[scale=0.7]{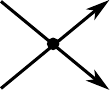}}},p\right)&:=
\vcenter{\hbox{\includegraphics[scale=0.7]{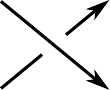}}}&
\mathcal{R}_-\left(\vcenter{\hbox{\includegraphics[scale=0.7]{K_singularcrossing_RR.pdf}}},p\right)&:=
\vcenter{\hbox{\includegraphics[scale=0.7]{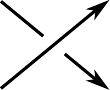}}}&
\mathcal{R}_0\left(\vcenter{\hbox{\includegraphics[scale=0.7]{K_singularcrossing_RR.pdf}}},p\right)&:=
\vcenter{\hbox{\includegraphics[scale=0.7]{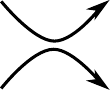}}}
\end{align*}
\end{defn}
We will omit the singular point $p$ when it is obvious. 

\subsection{Singular braids}\label{sec:braid}
For $n\ge 1$, we define $\mathbf{SB}_n$ by a monoid generated by $\sigma_1^{\pm1},\dots,\sigma_{n-1}^{\pm1}$, $\tau_1,\dots,\tau_{n-1}$ which satisfy the following relations.
\begin{align*}
\tag{$Br0$}\sigma_i\sigma_i^{-1}=\sigma_i^{-1}\sigma_i=e &\qquad i=1,\dots, n-1\\
\tag{$Br1$}\sigma_i\sigma_j=\sigma_j\sigma_i,\quad \sigma_i\tau_j=\tau_j\sigma_i,\quad \tau_i\tau_j=\tau_j\tau_i&\qquad|i-j|>1\\
\tag{$Br2$}\sigma_i\sigma_j\sigma_i=\sigma_j\sigma_i\sigma_j,\quad \sigma_i\sigma_j\tau_i=\tau_j\sigma_i\sigma_j&\qquad |i-j|=1\\
\tag{$BrF$}\sigma_i\tau_i = \tau_i\sigma_i&\qquad i=1,\dots,n-1
\end{align*}

Let $\xi_i:=\sigma_i\tau_i$ for $1\le i<n$ be auxiliary generators.
We may use $\xi_i$ instead of $\tau_i$ for generators, then due to an automorphism on $\mathbf{SB}_n$ defined as 
\[
\mathbf{SB}_n\to\mathbf{SB}_n=\begin{cases}
\sigma_i\mapsto \sigma_i,\\
\tau_i\mapsto \xi_i=\sigma_i\tau_i,
\end{cases}
\]
$\mathbf{SB}_n$ admits exactly the same monoid presentation as above except for using $\xi_i$ instead of $\tau_i$.

The generators can be viewed geometrically as depicted in Figure~\ref{fig:sbn}. We use a non-transverse intersection point for $\xi_i$ to distinguish with $\tau_i$.
Therefore for any $\beta\in\mathbf{SB}_n$, we can associate a map 
\[
\tilde\beta:(\sqcup_nI, \sqcup \partial I)\to(D^2\times I, D^2\times\partial I)
\]
by concatenating the pieces drawn in the figure properly according to the word $\beta$. We call $\tilde\beta$ a {\em geometric braid}.
It is obvious that the submonoid generated by $\sigma_i^{\pm1}$'s becomes the classical braid group $\mathbf{B}_n$.
We denote the unions $\sqcup_{n=1}^{\infty} \mathbf{B}_n$ and $\sqcup_{n=1}^{\infty} \mathbf{SB}_n$ by $\mathbf{B}$ and $\mathbf{SB}$.

\begin{figure}[ht]
\begin{align*}
\tilde\sigma_i&=\vcenter{\hbox{\includegraphics{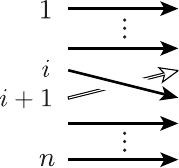}}}&
\tilde\sigma_i^{-1}&=\vcenter{\hbox{\includegraphics{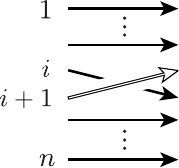}}}&
\tilde\tau_i&=\vcenter{\hbox{\includegraphics{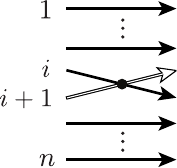}}}&
\tilde\xi_i&=\vcenter{\hbox{\includegraphics{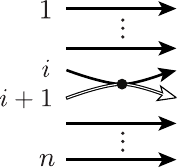}}}
\end{align*}
\caption{Geometric braids for generators of $\mathbf{SB}_n$}
\label{fig:sbn}
\end{figure}

\begin{defn}[Singular braid conjugation]Let $\beta_1, \beta_2\in\mathbf{SB}_n$. A {\em conjugation} $(BrC)$ of $\beta_1\beta_2$ by $\beta_1$ is defined by $\beta_2\beta_1$.
\[
\xymatrix@C=3pc@R=1pc{
\vcenter{\hbox{
\begingroup%
  \makeatletter%
  \providecommand\color[2][]{%
    \errmessage{(Inkscape) Color is used for the text in Inkscape, but the package 'color.sty' is not loaded}%
    \renewcommand\color[2][]{}%
  }%
  \providecommand\transparent[1]{%
    \errmessage{(Inkscape) Transparency is used (non-zero) for the text in Inkscape, but the package 'transparent.sty' is not loaded}%
    \renewcommand\transparent[1]{}%
  }%
  \providecommand\rotatebox[2]{#2}%
  \ifx\svgwidth\undefined%
    \setlength{\unitlength}{90.102bp}%
    \ifx\svgscale\undefined%
      \relax%
    \else%
      \setlength{\unitlength}{\unitlength * \real{\svgscale}}%
    \fi%
  \else%
    \setlength{\unitlength}{\svgwidth}%
  \fi%
  \global\let\svgwidth\undefined%
  \global\let\svgscale\undefined%
  \makeatother%
  \begin{picture}(1,0.35610752)%
    \put(0,0){\includegraphics[width=\unitlength,page=1]{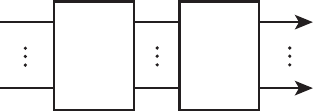}}%
    \put(0.247851,0.15761533){\color[rgb]{0,0,0}\makebox(0,0)[lb]{\smash{$\beta_1$}}}%
    \put(0.6504298,0.15761533){\color[rgb]{0,0,0}\makebox(0,0)[lb]{\smash{$\beta_2$}}}%
  \end{picture}%
\endgroup%
}}\quad
\ar@{<->}[r]^-{(BrC)}&
\quad\vcenter{\hbox{
\begingroup%
  \makeatletter%
  \providecommand\color[2][]{%
    \errmessage{(Inkscape) Color is used for the text in Inkscape, but the package 'color.sty' is not loaded}%
    \renewcommand\color[2][]{}%
  }%
  \providecommand\transparent[1]{%
    \errmessage{(Inkscape) Transparency is used (non-zero) for the text in Inkscape, but the package 'transparent.sty' is not loaded}%
    \renewcommand\transparent[1]{}%
  }%
  \providecommand\rotatebox[2]{#2}%
  \ifx\svgwidth\undefined%
    \setlength{\unitlength}{90.102bp}%
    \ifx\svgscale\undefined%
      \relax%
    \else%
      \setlength{\unitlength}{\unitlength * \real{\svgscale}}%
    \fi%
  \else%
    \setlength{\unitlength}{\svgwidth}%
  \fi%
  \global\let\svgwidth\undefined%
  \global\let\svgscale\undefined%
  \makeatother%
  \begin{picture}(1,0.35610752)%
    \put(0,0){\includegraphics[width=\unitlength,page=1]{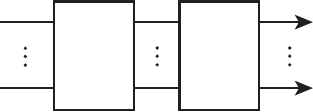}}%
    \put(0.247851,0.15761533){\color[rgb]{0,0,0}\makebox(0,0)[lb]{\smash{$\beta_2$}}}%
    \put(0.6504298,0.15761533){\color[rgb]{0,0,0}\makebox(0,0)[lb]{\smash{$\beta_1$}}}%
  \end{picture}%
\endgroup%
}}
}
\]
\end{defn}
\begin{defn}[Singular braid (de)stabilizations]Let $\beta\in\mathbf{SB}_n$. The positive and negative stabilizations $(BrS_\pm)$ of $\beta$ are defined by the $(n+1)$-braid $\beta\sigma_n^{\pm1}\in\mathbf{SB}_{n+1}$, respectively, and their inverse operations are called the {\em positive and negative destabilizations}, denoted by $(BrS_\pm)^{-1}$, respectively.
\[
\xymatrix@C=3pc@R=1pc{
\vcenter{\hbox{
\begingroup%
  \makeatletter%
  \providecommand\color[2][]{%
    \errmessage{(Inkscape) Color is used for the text in Inkscape, but the package 'color.sty' is not loaded}%
    \renewcommand\color[2][]{}%
  }%
  \providecommand\transparent[1]{%
    \errmessage{(Inkscape) Transparency is used (non-zero) for the text in Inkscape, but the package 'transparent.sty' is not loaded}%
    \renewcommand\transparent[1]{}%
  }%
  \providecommand\rotatebox[2]{#2}%
  \ifx\svgwidth\undefined%
    \setlength{\unitlength}{86.332bp}%
    \ifx\svgscale\undefined%
      \relax%
    \else%
      \setlength{\unitlength}{\unitlength * \real{\svgscale}}%
    \fi%
  \else%
    \setlength{\unitlength}{\svgwidth}%
  \fi%
  \global\let\svgwidth\undefined%
  \global\let\svgscale\undefined%
  \makeatother%
  \begin{picture}(1,0.47310125)%
    \put(0,0){\includegraphics[width=\unitlength,page=1]{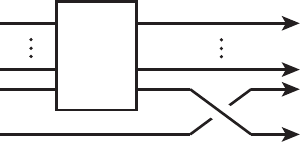}}%
    \put(0.2926837,0.23847346){\color[rgb]{0,0,0}\makebox(0,0)[lb]{\smash{$\beta$}}}%
  \end{picture}%
\endgroup%
}}\ 
\ar@<-.5ex>[r]_-{(BrS_+)^{-1}}&\ar@<-.5ex>[l]_-{(BrS_+)}
\ \vcenter{\hbox{
\begingroup%
  \makeatletter%
  \providecommand\color[2][]{%
    \errmessage{(Inkscape) Color is used for the text in Inkscape, but the package 'color.sty' is not loaded}%
    \renewcommand\color[2][]{}%
  }%
  \providecommand\transparent[1]{%
    \errmessage{(Inkscape) Transparency is used (non-zero) for the text in Inkscape, but the package 'transparent.sty' is not loaded}%
    \renewcommand\transparent[1]{}%
  }%
  \providecommand\rotatebox[2]{#2}%
  \ifx\svgwidth\undefined%
    \setlength{\unitlength}{63.31499634bp}%
    \ifx\svgscale\undefined%
      \relax%
    \else%
      \setlength{\unitlength}{\unitlength * \real{\svgscale}}%
    \fi%
  \else%
    \setlength{\unitlength}{\svgwidth}%
  \fi%
  \global\let\svgwidth\undefined%
  \global\let\svgscale\undefined%
  \makeatother%
  \begin{picture}(1,0.50673618)%
    \put(0,0){\includegraphics[width=\unitlength,page=1]{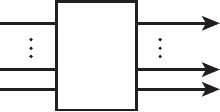}}%
    \put(0.39233038,0.17699114){\color[rgb]{0,0,0}\makebox(0,0)[lb]{\smash{$\beta$}}}%
  \end{picture}%
\endgroup%
}}\ 
\ar@<.5ex>[r]^-{(BrS_-)}&\ar@<.5ex>[l]^-{(BrS_-)^{-1}}
\ \vcenter{\hbox{
\begingroup%
  \makeatletter%
  \providecommand\color[2][]{%
    \errmessage{(Inkscape) Color is used for the text in Inkscape, but the package 'color.sty' is not loaded}%
    \renewcommand\color[2][]{}%
  }%
  \providecommand\transparent[1]{%
    \errmessage{(Inkscape) Transparency is used (non-zero) for the text in Inkscape, but the package 'transparent.sty' is not loaded}%
    \renewcommand\transparent[1]{}%
  }%
  \providecommand\rotatebox[2]{#2}%
  \ifx\svgwidth\undefined%
    \setlength{\unitlength}{86.332bp}%
    \ifx\svgscale\undefined%
      \relax%
    \else%
      \setlength{\unitlength}{\unitlength * \real{\svgscale}}%
    \fi%
  \else%
    \setlength{\unitlength}{\svgwidth}%
  \fi%
  \global\let\svgwidth\undefined%
  \global\let\svgscale\undefined%
  \makeatother%
  \begin{picture}(1,0.47310116)%
    \put(0,0){\includegraphics[width=\unitlength,page=1]{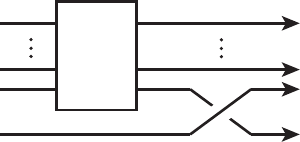}}%
    \put(0.2926837,0.23847346){\color[rgb]{0,0,0}\makebox(0,0)[lb]{\smash{$\beta$}}}%
  \end{picture}%
\endgroup%
}}
}
\]
\end{defn}
Here we identify $\mathbf{SB}_n$ with a submonoid of $\mathbf{SB}_{n+1}$ via the natural inclusion $\mathbf{SB}_n\to\mathbf{SB}_{n+1}$ sending $\sigma_i$ and $\tau_i$ to themselves, and call these moves, conjugations and (de)stabilizations, {\em Markov moves}.

We introduce one more important move as follows.
\begin{defn}[Exchange move]
Let $\beta_1,\beta_2\in\mathbf{SB}_n$. Then an {\em exchange move} $(BrE)$ changes $\beta_1\sigma_n\beta_2\sigma_n^{-1}$ into $\beta_1\sigma_n^{-1}\beta_2\sigma_n$ or {\it vice versa}.
\[
\xymatrix@C=3pc@R=1pc{
\vcenter{\hbox{
\begingroup%
  \makeatletter%
  \providecommand\color[2][]{%
    \errmessage{(Inkscape) Color is used for the text in Inkscape, but the package 'color.sty' is not loaded}%
    \renewcommand\color[2][]{}%
  }%
  \providecommand\transparent[1]{%
    \errmessage{(Inkscape) Transparency is used (non-zero) for the text in Inkscape, but the package 'transparent.sty' is not loaded}%
    \renewcommand\transparent[1]{}%
  }%
  \providecommand\rotatebox[2]{#2}%
  \ifx\svgwidth\undefined%
    \setlength{\unitlength}{140.539bp}%
    \ifx\svgscale\undefined%
      \relax%
    \else%
      \setlength{\unitlength}{\unitlength * \real{\svgscale}}%
    \fi%
  \else%
    \setlength{\unitlength}{\svgwidth}%
  \fi%
  \global\let\svgwidth\undefined%
  \global\let\svgscale\undefined%
  \makeatother%
  \begin{picture}(1,0.29062231)%
    \put(0,0){\includegraphics[width=\unitlength,page=1]{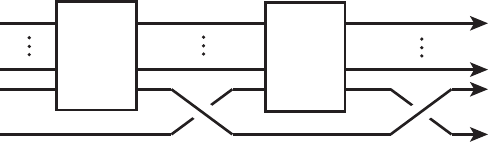}}%
    \put(0.16840855,0.14902169){\color[rgb]{0,0,0}\makebox(0,0)[lb]{\smash{$\beta_1$}}}%
    \put(0.5953363,0.14902169){\color[rgb]{0,0,0}\makebox(0,0)[lb]{\smash{$\beta_2$}}}%
  \end{picture}%
\endgroup%
}}
\ar@{<->}[r]^{(BrE)}&
\vcenter{\hbox{
\begingroup%
  \makeatletter%
  \providecommand\color[2][]{%
    \errmessage{(Inkscape) Color is used for the text in Inkscape, but the package 'color.sty' is not loaded}%
    \renewcommand\color[2][]{}%
  }%
  \providecommand\transparent[1]{%
    \errmessage{(Inkscape) Transparency is used (non-zero) for the text in Inkscape, but the package 'transparent.sty' is not loaded}%
    \renewcommand\transparent[1]{}%
  }%
  \providecommand\rotatebox[2]{#2}%
  \ifx\svgwidth\undefined%
    \setlength{\unitlength}{140.539bp}%
    \ifx\svgscale\undefined%
      \relax%
    \else%
      \setlength{\unitlength}{\unitlength * \real{\svgscale}}%
    \fi%
  \else%
    \setlength{\unitlength}{\svgwidth}%
  \fi%
  \global\let\svgwidth\undefined%
  \global\let\svgscale\undefined%
  \makeatother%
  \begin{picture}(1,0.29062231)%
    \put(0,0){\includegraphics[width=\unitlength,page=1]{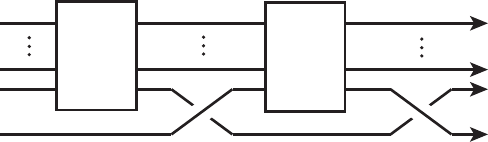}}%
    \put(0.16840854,0.14902169){\color[rgb]{0,0,0}\makebox(0,0)[lb]{\smash{$\beta_2$}}}%
    \put(0.59533623,0.14902169){\color[rgb]{0,0,0}\makebox(0,0)[lb]{\smash{$\beta_1$}}}%
  \end{picture}%
\endgroup%
}}
}
\]
\end{defn}

Then it is well-known that the exchange move $(BrE)$
can be generated by conjugations $(BrC)$ and $(+)$-(de)stabilizations $(BrS_+)^{\pm1}$.
Therefore it induces the equivalence in $\mathbf{SB}/\{(BrC), (BrS_+)\}$, and in $\SK$ as well via the closure map.
We denote the singular braid monoid $\mathbf{SB}$ modulo conjugation $(BrC)$ and exchange move $(BrE)$ by  $\SB$.
\[
\B := \mathbf{B}/\{(BrC), (BrE)\}\subset \SB:=\mathbf{SB}/\{(BrC),(BrE)\}.
\]

\begin{defn}[Resolutions on $\SB$]
For $\beta=\beta_1p\beta_2\in\SB$ with $p=\tau_i$ or $\xi_i$, the $\eta$-resolution $\mathcal{R}_\eta(\beta,p)$ of $\beta$ at $p$ is defined as
\[
\mathcal{R}_\eta(\beta,p):=\beta_1\mathcal{R}_\eta(p)\beta_2,
\]
where for each $\eta\in\{+,-,0\}$, $\mathcal{R}_\eta(p)$ is defined as 
\begin{align*}
\mathcal{R}_+(\tau_i)&:=\sigma_i, & 
\mathcal{R}_-(\tau_i)&:=\sigma_i^{-1}, &
\mathcal{R}_0(\tau_i)&:=e,\\
\mathcal{R}_+(\xi_i)&:=\sigma_i^2, & 
\mathcal{R}_-(\xi_i)&:=e, &
\mathcal{R}_0(\xi_i)&:=\sigma_i.
\end{align*}
\end{defn}

\subsubsection{Closures}
For any $\beta\in\mathbf{SB}$, a {\em closure $\hat\beta$ of $\beta$} is a link obtained by connecting the corresponding ends of a geometric braid $\tilde\beta$.
Then any singular link has a singular braid representation as follows.
\begin{thm}\cite{A, B}
For any singular link $K\in\SK$, there exists $\beta\in\mathbf{SB}$ such that the closure $\widehat{\beta}$ is equivalent to $K$.
\end{thm}

However, a singular braid representation is not unique in general. Indeed, two singular braids which differs by the following moves have the equivalent closures.

\begin{thm}\cite{B}\label{thm:SBtoSK}
If $\alpha,\beta\in\mathbf{SB}$ are two singular braids with $\widehat\alpha=\widehat\beta$ in $\SK$, then they differs by a sequence of Markov moves, that is, braid conjugations $(BrC)$ and $(\pm)$-(de)stabilizations $(BrS_\pm)^{\pm1}$.

In other words, the closure $\widehat{(\cdot)}$ induces bijections
\begin{align*}
\widehat{(\cdot)}&:\mathbf{B}/\{(BrC), (BrS_\pm)\}=\B/\{(BrS_\pm)\}\to\K,\\
\widehat{(\cdot)}&:\mathbf{SB}/\{(BrC), (BrS_\pm)\}=\SB/\{(BrS_\pm)\}\to\SK.
\end{align*}
%
\end{thm}

\subsubsection{Singular rectilinear braid diagrams}\label{sec:rectilinear}
We consider a singular version of a rectilinear braid diagram as follows. 
\begin{defn}\cite{MM,NT}
A {\em singular rectilinear braid diagram} is a singular braid diagram consisting of horizontal and vertical oriented line segments satisfying the following.
\begin{enumerate}
\item All horizontal segments are oriented from left to right, and
\item at each double point, either a vertical segment passes over a horizontal segment, or intersects with a horizontal segment and make a singular point. 
\end{enumerate}
We denote the set of all singular rectilinear braid diagrams by $\bar{\mathbf{SB}}$. 
\end{defn}
Then there is a canonical way to realize a singular rectilinear braid diagram $\bar\beta\in\bar{\mathbf{SB}}$ as a singular braid diagram of $\beta\in\mathbf{SB}$ by slanting all vertical segments slightly, where the map is denoted by $sl:\bar{\mathbf{SB}}\to\mathbf{SB}$.

\begin{figure}[ht]
\[
\xymatrix{
\vcenter{\hbox{\includegraphics[scale=0.7]{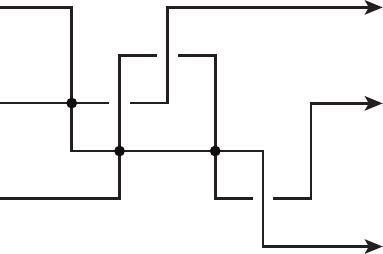}}}\quad\ar@{|->}[r]^{sl} & \quad\vcenter{\hbox{\includegraphics[scale=0.7]{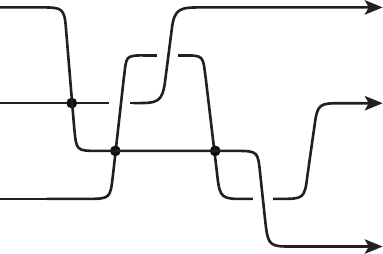}}}
}
\]
\caption{A singular rectilinear diagram $\bar\beta\in\bar{\mathbf{SB}}$ and corresponding singular braid $\beta\in\mathbf{SB}$}
\label{fig:rectilinearbraid}
\end{figure}

Indeed, this map is surjective as follows. Since for any braid $\beta$, one can construct a singular rectilinear braid diagram $\bar\beta$ with $sl(\bar\beta)=\beta$ by concatenating the diagrams below from the left according to the word representing $\beta$.
\begin{align*}
\bar\sigma_i&=\vcenter{\hbox{\includegraphics{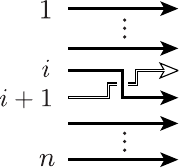}}}&
\bar\sigma_i^{-1}&=\vcenter{\hbox{\includegraphics{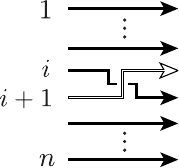}}}&
\bar\tau_i&=\vcenter{\hbox{\includegraphics{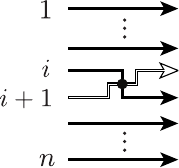}}}&
\bar\xi_i&=\vcenter{\hbox{\includegraphics{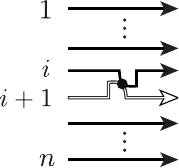}}}
\end{align*}
Here we interpret the non-transverse crossing in $\bar\xi_i$ as follows.
\[
\vcenter{\hbox{\includegraphics[scale=1.5]{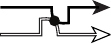}}}\qquad := \qquad
\vcenter{\hbox{\includegraphics[scale=1.5]{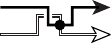}}}
\]

\subsection{Singular Legendrian links}

Let $\alpha_0=dz-ydx$ be a contact 1-form defined on $\R^3$ and $\xi_0=\ker \alpha_0$ be a plane distribution on $\R^3$, called the {\em standard contact structure}.
A {\em singular Legendrian link $L$} is a singular link which is {\em Legendrian} with respect to $\xi_0$. In other words, $L$ is tangent to the contact plane at each point of $L$. Two singular Legendrian links are said to be {\em equivalent} if and only if they are equivalent as singular links by preserving the Legendrian-ness during the homotopy.
We denote the sets of all equivalence classes of nonsingular and singular Legendrian links by $\L$ and $\SL$, respectively.

The projection $\pi_F:\R^3\to\R^2_{xz}$ is called the {\em front projection}, and for any $L\in\SL$, $\pi_F(L)$ consists of piecewise smooth closed curves in the $xz$-plane {\em without} a vertical tangency by the Legendrian-ness. Instead, it may have cusps. See Figure~\ref{fig:gridsquareexample}(b) for example.

Moreover at each double point in $\pi_F(L)$, the crossing information comes naturally since the $y$-coordinate, namely $L$ itself, can be recovered from $\pi_F(L)$ by using the Legendrian condition $dz-ydx=0$.
Hence at each nonsingular crossing in $\pi_F(L)$, the strand with a smaller slope is always lying over the strand with a larger slope.
Therefore we consider $\pi_F(L)$ as the front projection equipped with the crossing information.

The projection near each singular point looks like a picture $L_\bullet=\vcenter{\hbox{\includegraphics{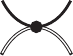}}}$ with non-transverse intersection because the same $y$-coordinates yield the same slopes $dz/dx$. For each singular point, we indicate a dot to avoid confusion with ordinary double points (crossings) in the front projections.
Then there are four cases of the projection $\pi_F(L_\bullet)$ according to the orientations on each arc as depicted in Figure~\ref{fig:singularpoints}.

\begin{figure}[ht]
\begin{align*}
\pi_F(L_{\bullet}^{N})&=\vcenter{\hbox{\includegraphics[scale=0.7]{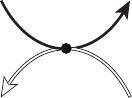}}}&
\pi_F(L_{\bullet}^{E})&=\vcenter{\hbox{\includegraphics[scale=0.7]{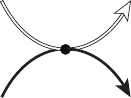}}}&
\pi_F(L_{\bullet}^{W})&=\vcenter{\hbox{\includegraphics[scale=0.7]{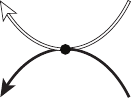}}}&
\pi_F(L_{\bullet}^{S})&=\vcenter{\hbox{\includegraphics[scale=0.7]{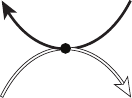}}}
\end{align*}
\caption{Four projections $\pi_F(L_\bullet)$ of $L_\bullet$ near a singular point}
\label{fig:singularpoints}
\end{figure}

We say that $\pi_F(L)$ is {\em regular} if 
\begin{enumerate}
\item it has no triple (or more) point, none of its double points is a cusp; and 
\item it is parameterized like one of front projections at every singular point depicted in Figure~\ref{fig:singularpoints}.
\end{enumerate}
Note that it is not hard to make the front projection $\pi_F(L)$ regular by perturbing $L$ slightly. Moreover, there is a combinatorial description for the equivalence between regular front projections in $\SL$ as follows.

\begin{prop}\label{prop:frontmove}\cite[Proposition~2.1]{ABK}
Let $L_1, L_2\in\SL$ be singular Legendrian links with $\pi_F(L_1)$ and $\pi_F(L_2)$ regular. Then $L_1$ and $L_2$ are equivalent in $\SL$ if and only if $\pi_F(L_1)$ and $\pi_F(L_2)$ are related by a planar isotopy and a finite sequence of moves $(LRM1)\sim (LRM6)$ including their reflections about the $x$ and $z$-axes, depicted in Figure~\ref{fig:frontmove}.
\end{prop}
\begin{rmk}
The move $(LRM4)$ is indeed a composition of $(LRM6)$ and a planar isotopy.
\end{rmk}

\begin{figure}[ht]
\begin{align*}
\vcenter{\hbox{\includegraphics{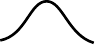}}}\quad&\xleftrightarrow{(LRM1)}\quad
\vcenter{\hbox{\includegraphics{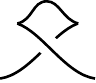}}}&
\vcenter{\hbox{\includegraphics{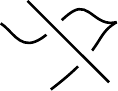}}}\quad&\xleftrightarrow{(LRM2)}\quad
\vcenter{\hbox{\includegraphics{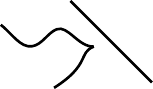}}}\\
\vcenter{\hbox{\includegraphics{KRM3_1.pdf}}}\quad&\xleftrightarrow{(LRM3)}\quad
\vcenter{\hbox{\includegraphics{KRM3_2.pdf}}}&
\vcenter{\hbox{\includegraphics{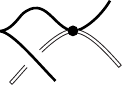}}}\quad&\xleftrightarrow{(LRM4)}\quad
\vcenter{\hbox{\includegraphics{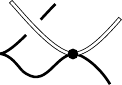}}}\\
\vcenter{\hbox{\includegraphics{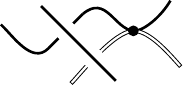}}}\quad&\xleftrightarrow{(LRM5)}\quad
\vcenter{\hbox{\includegraphics{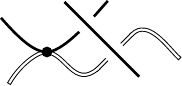}}}&
\vcenter{\hbox{\includegraphics{L_singularcrossing.pdf}}}\quad&\xleftrightarrow{(LRM6)}\quad
\vcenter{\hbox{\includegraphics{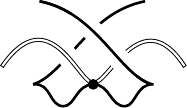}}}
\end{align*}
\caption{Reidemeister moves for $\SL$}
\label{fig:frontmove}
\end{figure}

\begin{cor}\label{cor:LegTr}
The following global moves in $\SL$ give equivalent pairs of singular Legendrian links.
\begin{align*}
\vcenter{\hbox{\includegraphics[scale=0.9]{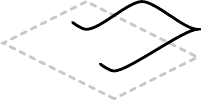}}}\quad&\xleftrightarrow{(LTr_H)}\quad
\vcenter{\hbox{\includegraphics[scale=0.9]{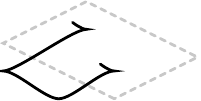}}}&
\vcenter{\hbox{\includegraphics[scale=0.9]{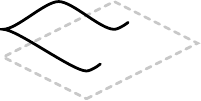}}}\quad&\xleftrightarrow{(LTr_V)}\quad
\vcenter{\hbox{\includegraphics[scale=0.9]{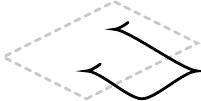}}}\\
\vcenter{\hbox{\includegraphics[scale=1]{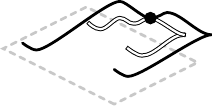}}}\quad&\xleftrightarrow{(LTr_H)}\quad
\vcenter{\hbox{\includegraphics[scale=1]{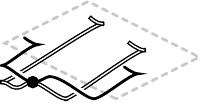}}}&
\vcenter{\hbox{\includegraphics[scale=1]{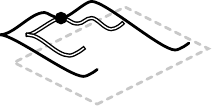}}}\quad&\xleftrightarrow{(LTr_V)}\quad
\vcenter{\hbox{\includegraphics[scale=1]{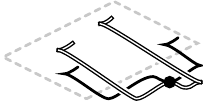}}}
\end{align*}
\end{cor}
We call these two moves {\em Legendrian horizontal and vertical translations}, denoted by $(LTr_H)$ and $(LTr_V)$, respectively.
\begin{proof}
The proof is not hard and left as an exercise for the reader.
\end{proof}

\begin{defn}[Resolutions on $\SL$]\label{defn:LegResolutions}
Let $L\in\SL$ be a singular Legendrian link and $p$ be a singular point of $L$. Then for each $\eta\in\{+,-,0\}$, we define the {\em $\eta$-resolution} $\mathcal{R}_\eta(L,p)$ at $p$ by the diagram replacement as follows including their reflections about $z$-axis.
\begin{align*}
\mathcal{R}_+\left(\vcenter{\hbox{\includegraphics[scale=0.7]{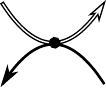}}},p\right)&:=
\vcenter{\hbox{\includegraphics[scale=0.7]{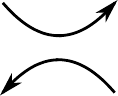}}} &
\mathcal{R}_-\left(\vcenter{\hbox{\includegraphics[scale=0.7]{L_singularcrossing_RL.pdf}}},p\right)&:=
\vcenter{\hbox{\includegraphics[scale=0.7]{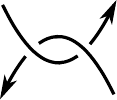}}}&
\mathcal{R}_0\left(\vcenter{\hbox{\includegraphics[scale=0.7]{L_singularcrossing_RL.pdf}}},p\right)&:=
\vcenter{\hbox{\includegraphics[scale=0.7]{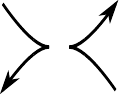}}}\\
\mathcal{R}_+\left(\vcenter{\hbox{\includegraphics[scale=0.7]{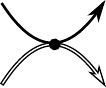}}},p\right)&:=
\vcenter{\hbox{\includegraphics[scale=0.7]{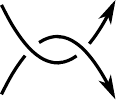}}} &
\mathcal{R}_-\left(\vcenter{\hbox{\includegraphics[scale=0.7]{L_singularcrossing_RR.pdf}}},p\right)&:=
\vcenter{\hbox{\includegraphics[scale=0.7]{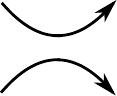}}}& 
\mathcal{R}_0\left(\vcenter{\hbox{\includegraphics[scale=0.7]{L_singularcrossing_RR.pdf}}},p\right)&:=
\vcenter{\hbox{\includegraphics[scale=0.7]{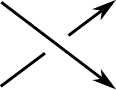}}}.
\end{align*}
\end{defn}

Recall the {\em Thurston-Bennequin number $tb(L)$} for a Legendrian link $L\in\L$, which is a classical invariant for $\L$ and measures how many times the contact planes rotate when we travel along $L$. More precisely, it is defined as
\begin{align*}
tb(L) =&   \#\left\{
\vcenter{\hbox{\includegraphics{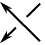}}},
\vcenter{\hbox{\includegraphics{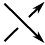}}}
\subset \pi_F(L)\right\}
- \#\left\{         
\vcenter{\hbox{\includegraphics{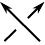}}},
\vcenter{\hbox{\includegraphics{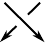}}},
\vcenter{\hbox{\includegraphics{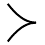}}}\subset \pi_F(L)
\right\}.
\end{align*}

In $\SL$, one can extend $tb$ as follows.
\begin{defn}\cite[\S2.4]{ABK}
Let $L\in\SL$. Then the Thurston-Bennequin number of $L$ is an integer defined as
\begin{align*}
tb(L) =&   \#\left\{
\vcenter{\hbox{\includegraphics{L_small_crossing_L.pdf}}},
\vcenter{\hbox{\includegraphics{L_small_crossing_R.pdf}}},
\vcenter{\hbox{\includegraphics{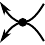}}}, 
\vcenter{\hbox{\includegraphics{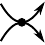}}} \subset \pi_F(L)
\right\}
- \#\left\{         
\vcenter{\hbox{\includegraphics{L_small_crossing_U.pdf}}},
\vcenter{\hbox{\includegraphics{L_small_crossing_D.pdf}}},
\vcenter{\hbox{\includegraphics{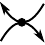}}}, 
\vcenter{\hbox{\includegraphics{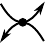}}},
\vcenter{\hbox{\includegraphics{L_small_cusp_R.pdf}}}\subset \pi_F(L)
\right\}.
\end{align*}
\end{defn}
It is not hard to see that $tb$ on $\SL$ is invariant under singular Legendrian Reidemeister moves, and therefore it is well-defined. 

\begin{rmk}
For $L\in\SK$, the invariant $tb(L)$ can be defined as the linking number between $L$ and its positive pushoff $L^+$ and so it does not depend on the front projection. See \cite[\S2.4]{ABK}.
\end{rmk}

Especially, on $\SL$, the Thurston-Bennequin number $tb$ behaves under resolutions as follows.

\begin{lem}\cite[\S2.4]{ABK}\label{lem:resolutionandtb} Let $L\in \SL$ and $p$ be a singular point of $L$. Then
\[
tb(\mathcal{R}_\pm(L,p)) = tb(L) \pm 1,\quad tb(\mathcal{R}_0(L,p))=tb(L).
\]
\end{lem}

\subsubsection{A map onto $\SK$}
We consider a canonical map $\|\cdot\|:\SL\to\SK$ defined by taking the underlying singular links for a given singular Legendrian link. 
Diagramatically, it may be defined as depicted in Figure~\ref{fig:SLSK}.

\begin{figure}[ht]
\[
\xymatrix@C=1pc@R=1pc{
\SL\ar[d]_{\|\cdot\|}&&
\vcenter{\hbox{\includegraphics{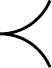}}}\ar@{|->}[d]&
\vcenter{\hbox{\includegraphics{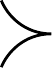}}}\ar@{|->}[d]&
\vcenter{\hbox{\includegraphics{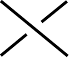}}}\ar@{|->}[d]&
\vcenter{\hbox{\includegraphics{L_singularcrossing.pdf}}}\ar@{|->}[d]\\
\SK& &
\vcenter{\hbox{\includegraphics{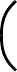}}}&
\vcenter{\hbox{\includegraphics{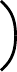}}}&
\vcenter{\hbox{\includegraphics{K_crossing.pdf}}}&
\vcenter{\hbox{\includegraphics{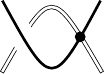}}}
}
\]
\caption{A pictorial definition of $\|\cdot\|:\SL\to\SK$}
\label{fig:SLSK}
\end{figure}

It is obvious that this map is (infinitely) many-to-one, and there are three moves which induce the equivalences not in $\SL$ but in $\SK$ via $\|\cdot\|$ as follows.

\begin{defn}[Legendrian (de)stabilizations]
Let $L\in\SL$. Then Legendrian $(\pm)$-stabilizations $(LS_\pm)$ on $L$ add $(\pm)$-{\em zigzags} as follows, and their inverse operations are called the {\em Legendrian $(\pm)$-destabilizations}, denoted by $(LS_\pm)^{-1}$.
\[
\xymatrix@C=3pc{
\vcenter{\hbox{\includegraphics{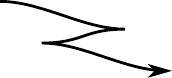}}}\quad
\ar@<-.5ex>[r]_-{(LS_+)^{-1}}&\ar@<-.5ex>[l]_-{(LS_+)}
\quad\vcenter{\hbox{\includegraphics{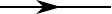}}}\quad
\ar@<.5ex>[r]^-{(LS_-)}&\ar@<.5ex>[l]^-{(LS_-)^{-1}}
\quad\vcenter{\hbox{\includegraphics{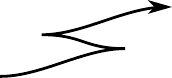}}}
}
\]
\end{defn}
\begin{defn}[Legendrian flype]A {\em Legendrian flype} $(LF)$ is a local move which changes the order of two contiguous Legendrian singular crossing and ordinary crossing in the front projection as follows.
\[
\xymatrix@C=3pc{
\vcenter{\hbox{\includegraphics{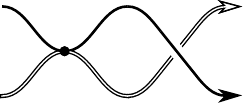}}}\quad
\ar@{<->}[r]^-{(LF)}&
\quad\vcenter{\hbox{\includegraphics{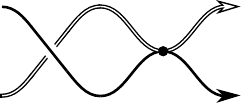}}}
}
\]
\end{defn}

Indeed, these three moves generate all singular Legendrian links in $\SL$ which are the same in $\SK$ as follows.

\begin{prop}\cite{ABK, FT}\label{prop:SLSK}
Let $L_1,L_2\in\SL$. Then $\|L_1\|$ and $\|L_2\|$ are equivalent in $\SK$ if and only if they differ by a sequence of $(\pm)$-(de)stabilizations $(LS_\pm)^{\pm1}$ and Legendrian flypes $(LF)$.
In other words, the map $\|\cdot\|$ induces bijections
\[
\|\cdot\|:\L/\{(LS_\pm)\}\to\K,\qquad
\|\cdot\|:\SL/\{(LS_\pm), (LF)\}\to\SK.
\] 
%
\end{prop}

\subsubsection{Sums of singular Legendrian tangles}
Let us consider singular Legendrian tangles. 
For example, four projections of $L_\bullet$ depicted in Figure~\ref{fig:singularpoints} can be considered as singular Legendrian tangles.

\begin{defn}
The {\em sum} $L_{T_1}\oplus L_{T_2}\in\SL$ of two singular Legendrian tangles $L_{T_1}$ and $L_{T_2}$ is a singular Legendrian link obtained by gluing two tangles as follows.
\[
L_{T_1}\oplus L_{T_2} = \vcenter{\hbox{
\begingroup%
  \makeatletter%
  \providecommand\color[2][]{%
    \errmessage{(Inkscape) Color is used for the text in Inkscape, but the package 'color.sty' is not loaded}%
    \renewcommand\color[2][]{}%
  }%
  \providecommand\transparent[1]{%
    \errmessage{(Inkscape) Transparency is used (non-zero) for the text in Inkscape, but the package 'transparent.sty' is not loaded}%
    \renewcommand\transparent[1]{}%
  }%
  \providecommand\rotatebox[2]{#2}%
  \ifx\svgwidth\undefined%
    \setlength{\unitlength}{112.69932583bp}%
    \ifx\svgscale\undefined%
      \relax%
    \else%
      \setlength{\unitlength}{\unitlength * \real{\svgscale}}%
    \fi%
  \else%
    \setlength{\unitlength}{\svgwidth}%
  \fi%
  \global\let\svgwidth\undefined%
  \global\let\svgscale\undefined%
  \makeatother%
  \begin{picture}(1,0.6460677)%
    \put(0,0){\includegraphics[width=\unitlength,page=1]{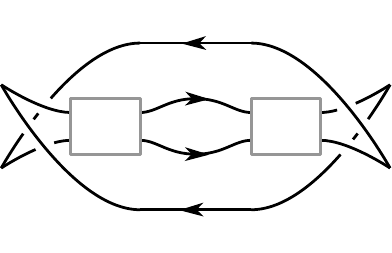}}%
    \put(0.20353363,0.29751633){\color[rgb]{0,0,0}\makebox(0,0)[lb]{\smash{$L_{T_1}$}}}%
    \put(0.66427523,0.29751633){\color[rgb]{0,0,0}\makebox(0,0)[lb]{\smash{$L_{T_2}$}}}%
  \end{picture}%
\endgroup%
}}\in\SL
\]

Especially, the sum $L_\bullet \oplus L_T$ of $L_\bullet$ and $L_T$ is  called the {\em Legendrian tangle closure of $L_T$} and denoted by $\hat{L_T}$.
\[
\widehat{\vcenter{\hbox{
\begingroup%
  \makeatletter%
  \providecommand\color[2][]{%
    \errmessage{(Inkscape) Color is used for the text in Inkscape, but the package 'color.sty' is not loaded}%
    \renewcommand\color[2][]{}%
  }%
  \providecommand\transparent[1]{%
    \errmessage{(Inkscape) Transparency is used (non-zero) for the text in Inkscape, but the package 'transparent.sty' is not loaded}%
    \renewcommand\transparent[1]{}%
  }%
  \providecommand\rotatebox[2]{#2}%
  \ifx\svgwidth\undefined%
    \setlength{\unitlength}{34.087636bp}%
    \ifx\svgscale\undefined%
      \relax%
    \else%
      \setlength{\unitlength}{\unitlength * \real{\svgscale}}%
    \fi%
  \else%
    \setlength{\unitlength}{\svgwidth}%
  \fi%
  \global\let\svgwidth\undefined%
  \global\let\svgscale\undefined%
  \makeatother%
  \begin{picture}(1,0.60964187)%
    \put(0,0){\includegraphics[width=\unitlength]{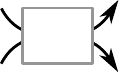}}%
    \put(0.33196332,0.23002214){\color[rgb]{0,0,0}\makebox(0,0)[lb]{\smash{$L_T$}}}%
  \end{picture}%
\endgroup%
}}}=\vcenter{\hbox{
\begingroup%
  \makeatletter%
  \providecommand\color[2][]{%
    \errmessage{(Inkscape) Color is used for the text in Inkscape, but the package 'color.sty' is not loaded}%
    \renewcommand\color[2][]{}%
  }%
  \providecommand\transparent[1]{%
    \errmessage{(Inkscape) Transparency is used (non-zero) for the text in Inkscape, but the package 'transparent.sty' is not loaded}%
    \renewcommand\transparent[1]{}%
  }%
  \providecommand\rotatebox[2]{#2}%
  \ifx\svgwidth\undefined%
    \setlength{\unitlength}{96.7bp}%
    \ifx\svgscale\undefined%
      \relax%
    \else%
      \setlength{\unitlength}{\unitlength * \real{\svgscale}}%
    \fi%
  \else%
    \setlength{\unitlength}{\svgwidth}%
  \fi%
  \global\let\svgwidth\undefined%
  \global\let\svgscale\undefined%
  \makeatother%
  \begin{picture}(1,0.75206825)%
    \put(0,0){\includegraphics[width=\unitlength]{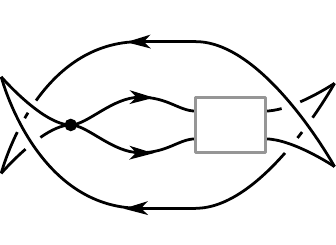}}%
    \put(0.62746539,0.32886105){\color[rgb]{0,0,0}\makebox(0,0)[lb]{\smash{$L_T$}}}%
  \end{picture}%
\endgroup%
}}\in\SL
\]
We also say that $L_T$ is a {\em tangle representative} for $L\in\SL$ if $L=\hat{L_T}$.
\end{defn}

\begin{rmk}
The sum $L_{T_1}\oplus L_{T_2}$ of two tangles is the same as the {\em singular connected sum} $\hat{L_{T_1}}\otimes \hat{L_{T_2}}$ of two Legendrian tangle closures, which is defined in \cite[Definition~4.1]{ABK}.
\end{rmk}

\begin{ex}\label{ex:L00}
Let $L_{\bigcirc\!\!\!\bigcirc}\in\SL$ be the closure of $L_\bullet$ as follows.
\[
L_{\bigcirc\!\!\!\bigcirc}=L_\bullet\oplus L_\bullet=
\vcenter{\hbox{\includegraphics{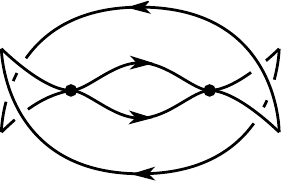}}}= 
\widehat{\vcenter{\hbox{\includegraphics[scale=1.3]{L_small_singularcrossing_RR.pdf}}}} = \hat{L_\bullet^{E}}\in\SL
\]
Then 
\begin{align*}
tb(L_{\bigcirc\!\!\!\bigcirc}) 
&= \#\left\{
\vcenter{\hbox{\includegraphics[scale=1]{L_small_crossing_L.pdf}}},
\vcenter{\hbox{\includegraphics[scale=1]{L_small_crossing_R.pdf}}},
\vcenter{\hbox{\includegraphics[scale=1]{L_small_singularcrossing_LL.pdf}}}, 
\vcenter{\hbox{\includegraphics[scale=1]{L_small_singularcrossing_RR.pdf}}} \subset \pi(L_{\bigcirc\!\!\!\bigcirc})
\right\} - \#\left\{         
\vcenter{\hbox{\includegraphics[scale=1]{L_small_crossing_U.pdf}}},
\vcenter{\hbox{\includegraphics[scale=1]{L_small_crossing_D.pdf}}},
\vcenter{\hbox{\includegraphics[scale=1]{L_small_singularcrossing_LR.pdf}}}, 
\vcenter{\hbox{\includegraphics[scale=1]{L_small_singularcrossing_RL.pdf}}},
\vcenter{\hbox{\includegraphics[scale=1]{L_small_cusp_R.pdf}}}\subset \pi(L_{\bigcirc\!\!\!\bigcirc})
\right\} \\
&=( 2 + 0 + 0 + 2 ) - ( 2 + 2 + 0 + 0 + 2 )  = -2.
\end{align*}

Moreover, its topological type is
\[
\|L_{\bigcirc\!\!\!\bigcirc}\| = \vcenter{\hbox{\includegraphics[scale=0.8]{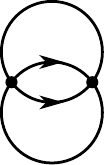}}}\in\SK,
\]
and whatever orientations are assigned on $L_{\bigcirc\!\!\!\bigcirc}$, they are all the same not only in $\SK$ but also in $\SL$ by the symmetry that $L_{\bigcirc\!\!\!\bigcirc}$ has.
\end{ex}

\begin{lem}\cite[Lemma~5.2]{ABK}\label{lem:tanglerep}
For any $L\in\SL$ with a singular point $p$, there exists a tangle representative $L_T$ for $L$ such that
$p$ corresponds to the singular point in $L_\bullet\subset L_\bullet\oplus L_T=\hat{L_T}$.
\end{lem}

\begin{cor}\label{cor:complement}
Let $L\in\SL$. If a tangle $L_T$ is contained in $L$, then there exists a tangle $L_T^c$ such that
\[
L_T\oplus L_T^c = L.
\]
\end{cor}
\begin{proof}
Let $L'$ be a singular Legendrian link obtained by replacing $L_T$ in $L$ with $L_\bullet$. We write this replacement as
\[
L':=(L\setminus L_T)\sqcup L_\bullet,
\]
and we denote the singular point in $L_\bullet$ by $p$.
By Lemma~\ref{lem:tanglerep}, there exists a tangle $L_T^c$ such that
\[
L'=\hat{L_T^c}=L_T^c\oplus L_\bullet,
\]
where a singular point in $L_\bullet$ corresponds to $p$. Then we have
\begin{align*}
L&=(L'\setminus L_\bullet)\sqcup L_T = \left((L_T^c\oplus L_\bullet)\setminus L_\bullet\right)\sqcup L_T=L_T^c\oplus L_T
\end{align*}
as desired.
\end{proof}

\begin{thm}\cite[Theorem~1.3]{ABK}\label{thm:deg2unit}
Suppose that $L_T$ is a tangle summand of $L_{\bigcirc\!\!\!\bigcirc}$. If $L_T$ contains only 1 singular point, then its closure $\hat{L_T}$ is the same as $L_{\bigcirc\!\!\!\bigcirc}$.
\end{thm}

\subsection{Singular transverse links}
A {\em singular transverse link $T$} in $(\R^3,\alpha_0)$ is a singular link which is {\em positively transverse} to the contact plane at each point in $T$. In other words, the pull-back of $\alpha_0=dz-ydx$ by $T$ is precisely a {\em positive} volume form on $n S^1$.
Similar to singular Legendrian links, two singular transverse links are {\em equivalent} if and only if they are equivalent as singular links by preserving the transversality during the homotopy.
We denote the sets of all equivalence classes of nonsingular and singular transverse links by $\T$ and $\ST$, respectively.

For a transverse link $T\in\ST$, the front projection $\pi_F(T)$ is a smooth immersion of $nS^1$ in $\R^2_{xz}$, and it locally looks a transverse intersection near a singular point of $T$ in general.
Hence the front projection $\pi_F(T)$ for any $T\in\ST$ can be realized as a diagram of a singular link in $\SK$ without any changes, and moreover, by perturbing $T$ slightly, we may assume that $\pi_F(T)$ is regular in the sense of the regular diagram for $\SK$. 
However, the positive transversality prohibits the appearances of the front projections as depicted in Figure~\ref{fig:forbiddentransversalfront}.

\begin{figure}[ht]
\[
\xymatrix{
\vcenter{\hbox{\includegraphics{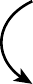}}}&
\vcenter{\hbox{\includegraphics{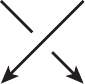}}}&
\vcenter{\hbox{\includegraphics{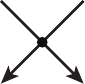}}}&
\vcenter{\hbox{\includegraphics{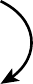}}}
}
\]
\caption{Forbidden front projections for singular transverse links}
\label{fig:forbiddentransversalfront}
\end{figure}

On the other hand, it is not hard to check that any regular diagram $D(K)$ for $K\in\SK$ without projections listed in Figure~\ref{fig:forbiddentransversalfront} can be realized as a front projection of a singular transverse link.
Furthermore, any planar isotopy and (singular) Reidemeister move without making the forbidden front projections induce an equivalence in $\ST$. Hence the equivalence in $\ST$ can be given by a sequence of planar isotopies and singular Reidemeister moves without forbidden front projections. More precisely, one can find the set of singular transverse Reidemeister moves as depicted in Figure~\ref{fig:transverseRM}. Notice that the first Reidemeister move $(RM1)$ always contains a downward cusp and so it is not allowed in $\ST$. This is summarized as follows and the proof is obvious and omitted.

\begin{figure}[ht]
\begin{align*}
\vcenter{\hbox{\includegraphics[scale=0.9]{grid_R.pdf}}}&\xleftrightarrow{(TRM0_1)}
\vcenter{\hbox{\includegraphics[scale=0.9]{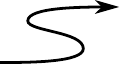}}} &
\vcenter{\hbox{\includegraphics[scale=0.9]{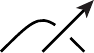}}} &\xleftrightarrow{(TRM0_2)}
\vcenter{\hbox{\includegraphics[scale=0.9]{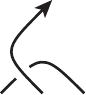}}} &
\vcenter{\hbox{\includegraphics[scale=0.9]{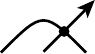}}} &\xleftrightarrow{(TRM0_3)}
\vcenter{\hbox{\includegraphics[scale=0.9]{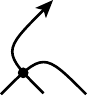}}}\\
\vcenter{\hbox{\includegraphics[scale=0.9]{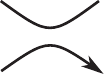}}} &\xleftrightarrow{(TRM2_1)}
\vcenter{\hbox{\includegraphics[scale=0.9]{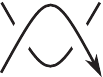}}} &
\vcenter{\hbox{\includegraphics[scale=0.9]{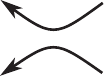}}} &\xleftrightarrow{(TRM2_2)}
\vcenter{\hbox{\includegraphics[scale=0.9]{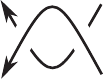}}} &
\vcenter{\hbox{\includegraphics[scale=0.9]{KRM3_1.pdf}}} &\xleftrightarrow{(TRM3)}
\vcenter{\hbox{\includegraphics[scale=0.9]{KRM3_2.pdf}}}\\
\vcenter{\hbox{\includegraphics[scale=0.9]{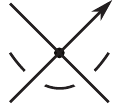}}} &\xleftrightarrow{(TRM4_1)}
\vcenter{\hbox{\includegraphics[scale=0.9]{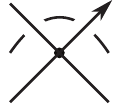}}} &
\vcenter{\hbox{\includegraphics[scale=0.9]{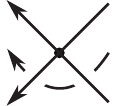}}} &\xleftrightarrow{(TRM4_2)}
\vcenter{\hbox{\includegraphics[scale=0.9]{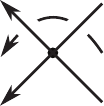}}} &
\vcenter{\hbox{\includegraphics[scale=0.9]{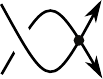}}} &\xleftrightarrow{(TRM5)}
\vcenter{\hbox{\includegraphics[scale=0.9]{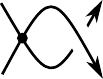}}}
\end{align*}
\caption{Reidemeister moves for $\ST$}
\label{fig:transverseRM}
\end{figure}

\begin{thm}
Let $T_1, T_2\in\ST$ be singular transverse links with regular $\pi_F(T_1)$ and $\pi_F(T_2)$ in the sense of the regular projection for $\SK$.
Then $T_1$ and $T_2$ are equivalent in $\ST$ if and only if $\pi_F(T_1)$ and $\pi_F(T_2)$ are related by a finite sequence of moves $(TRM0_*)\sim (TRM5)$ including their reflections about the $z$-axis, depicted in Figure~\ref{fig:transverseRM}.
\end{thm}
Note that this fact had been mentioned already in \cite{EFM} for nonsingular transverse links $\T$.
%
%
%
Moreover, since regular diagrams in $\ST$ can be regarded as regular diagrams in $\SK$, we can define resolutions on $\ST$ in the same manner as defined on $\SK$.
\begin{defn}[Resolutions on $\ST$]
Let $T\in\ST$ and $p$ be a singular point of $T$. Then for each $\eta\in\{+,-,0\}$, we define the {\em $\eta$-resolution} $\mathcal{R}_\eta(T,p)$ at $p$ by the diagram replacement as follows.
\begin{align*}
\mathcal{R}_+\left(\vcenter{\hbox{\includegraphics[scale=0.7]{K_singularcrossing_RR.pdf}}},p\right)&:=
\vcenter{\hbox{\includegraphics[scale=0.7]{K_singularcrossing_RR_+.pdf}}}&
\mathcal{R}_-\left(\vcenter{\hbox{\includegraphics[scale=0.7]{K_singularcrossing_RR.pdf}}},p\right)&:=
\vcenter{\hbox{\includegraphics[scale=0.7]{K_singularcrossing_RR_-.pdf}}}&
\mathcal{R}_0\left(\vcenter{\hbox{\includegraphics[scale=0.7]{K_singularcrossing_RR.pdf}}},p\right)&:=
\vcenter{\hbox{\includegraphics[scale=0.7]{K_singularcrossing_RR_0.pdf}}}
\end{align*}
\end{defn}

\subsubsection{Positive push-offs of singular Legendrian links}
We define a map $(\cdot)^+:\SL\to\ST$ by a {\em positive push-off}, which is a small perturbation that makes a given Legendrian link {\em positive}. 
Diagrammatically, it is defined as depicted in Figure~\ref{fig:pushoff}.
One can prove the well-definedness of $(\cdot)^+$ by showing that all Legendrian Reidemeister moves $(LRM1)\sim(LRM6)$ define equivalences in $\ST$. Note that this depends on the orientation of a given link.
See \cite[\S2.9]{E} for details of nonsingular cases.

\begin{figure}[ht]
\[
\xymatrix@C=0.5pc@R=1pc{
\SL\ar[d]_{(\cdot)^+} & &
\vcenter{\hbox{\includegraphics{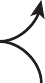}}}\ar@{|->}[d] &
\vcenter{\hbox{\includegraphics{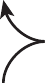}}}\ar@{|->}[d] &
\vcenter{\hbox{\includegraphics{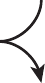}}}\ar@{|->}[d] &
\vcenter{\hbox{\includegraphics{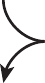}}}\ar@{|->}[d] &
\vcenter{\hbox{\includegraphics{K_crossing.pdf}}}\ar@{|->}[d] 
&\vcenter{\hbox{\includegraphics[scale=0.9]{L_singularcrossing_RL.pdf}}}\ar@{|->}[d] &
\vcenter{\hbox{\includegraphics[scale=0.9]{L_singularcrossing_RR.pdf}}}\ar@{|->}[d] &
\vcenter{\hbox{\includegraphics[scale=0.9]{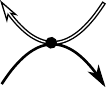}}}\ar@{|->}[d] &
\vcenter{\hbox{\includegraphics[scale=0.9]{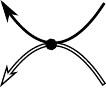}}}\ar@{|->}[d]
\\
\ST & &
\vcenter{\hbox{\includegraphics{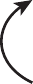}}}&
\vcenter{\hbox{\includegraphics{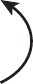}}}&
\vcenter{\hbox{\includegraphics{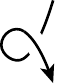}}} &
\vcenter{\hbox{\includegraphics{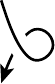}}} &
\vcenter{\hbox{\includegraphics{K_crossing.pdf}}} 
&
\vcenter{\hbox{\includegraphics[scale=0.9]{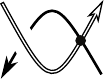}}} &
\vcenter{\hbox{\includegraphics[scale=0.9]{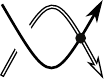}}} &
\vcenter{\hbox{\includegraphics[scale=0.9]{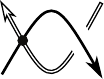}}} &
\vcenter{\hbox{\includegraphics[scale=0.9]{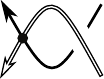}}}
}
\]
\caption{Positive push-off's of Legendrian front diagrams}
\label{fig:pushoff}
\end{figure}

\begin{thm}\label{thm:SLST}
The map $(\cdot)^+$ induces bijections 
\[
(\cdot)^+:\L/\{(LS_-)\} \to\T,\qquad (\cdot)^+:\SL/\{(LS_-), (LF)\} \to\ST.
\]
%
\end{thm}
For nonsingular Legendrian and transverse links, this is a well-known fact. See \cite[Theorem~2.1]{EFM}.
\begin{proof}
Since neither $(LS_-)$ nor $(LF)$ contain any forbidden front projection, they induce equivalences in $\ST$ via the pushoff $(\cdot)^+$. Therefore the maps $(\cdot)^+$ on the quotient spaces $\L/\{(LS_-)\}$ and $\SL/\{(LS_-),(LF)\}$ are well-defined.

For any $T\in\ST$, one can find $L$ such that $L^+=T$ by taking inverses for upward vertical tangencies and crossings in $\pi_F(T)$, and hence $(\cdot)^+$ is surjective.

Finally, we claim that a planar isotopy and (singular) Reidemeister moves without forbidden front projections can be realized by a sequence of (singular) Legendrian Reidemeister moves, $(LS_-)$ and $(LF)$. 
For the moves which do not involve any singular point, the claim is proved by \cite[Theorem~2.1]{EFM}. Therefore it suffices to consider only $(TRM0_3)$, $(TRM4_1)$, $(TRM4_2)$ and $(TRM5)$.

It is not hard to see that the move $(TRM4_*)$ and $(TRM5)$ are essentially the same as $(LRM5)$ and $(LF)$, respectively, via the pushoff $(\cdot)^+$ up to $(LS_-)$.
Finally, we show that the move $(TRM0_3)$ is related with $(LRM4)$ as shown in Figure~\ref{fig:TRM0_3} and therefore we prove the claim.
\end{proof}

\begin{figure}[ht]
\[
\xymatrix@C=3pc@R=1pc{
\vcenter{\hbox{\includegraphics{TRM0_3_2.pdf}}}\ar@{<-|}[r]^-{(\cdot)^+}\ar@{<->}[d]_{(TRM0_3)} &
\vcenter{\hbox{\includegraphics{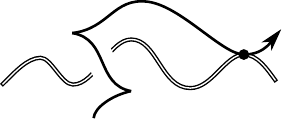}}} \ar[r]^-{/(LS_-)} & \vcenter{\hbox{\includegraphics{LRM4_1.pdf}}}
\ar@{<->}[d]^{(LRM4)}\\
\vcenter{\hbox{\includegraphics{TRM0_3_1.pdf}}}\ar@{<-|}[r]^-{(\cdot)^+} &
\vcenter{\hbox{\includegraphics{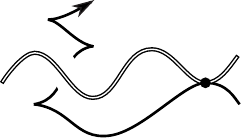}}} \ar[r]^-{/(LS_-)} & \vcenter{\hbox{\includegraphics{LRM4_2.pdf}}}
}
\]
\caption{Transverse Reidemeister move $(TRM0_3)$ and Legendrian Reidemeister move $(LRM4)$}
\label{fig:TRM0_3}
\end{figure}

\subsubsection{Transverse closures of singular braids}
Now we consider a map $\hat{(\cdot)}_\T:\SB\to\ST$, which is a generalization of the transverse closure $\hat{(\cdot)}_\T:\B\to\T$.
For $\beta\in\SB$, a transverse closure $\hat{\beta}_\T$ of $\beta$ is defined by the ordinary closure together with a $(-)$-kink at each strands on the right as follows.
\[
\xymatrix{
\vcenter{\hbox{
\begingroup%
  \makeatletter%
  \providecommand\color[2][]{%
    \errmessage{(Inkscape) Color is used for the text in Inkscape, but the package 'color.sty' is not loaded}%
    \renewcommand\color[2][]{}%
  }%
  \providecommand\transparent[1]{%
    \errmessage{(Inkscape) Transparency is used (non-zero) for the text in Inkscape, but the package 'transparent.sty' is not loaded}%
    \renewcommand\transparent[1]{}%
  }%
  \providecommand\rotatebox[2]{#2}%
  \ifx\svgwidth\undefined%
    \setlength{\unitlength}{63.31499634bp}%
    \ifx\svgscale\undefined%
      \relax%
    \else%
      \setlength{\unitlength}{\unitlength * \real{\svgscale}}%
    \fi%
  \else%
    \setlength{\unitlength}{\svgwidth}%
  \fi%
  \global\let\svgwidth\undefined%
  \global\let\svgscale\undefined%
  \makeatother%
  \begin{picture}(1,0.50673618)%
    \put(0,0){\includegraphics[width=\unitlength,page=1]{B_braid.pdf}}%
    \put(0.39233038,0.17699114){\color[rgb]{0,0,0}\makebox(0,0)[lb]{\smash{$\beta$}}}%
  \end{picture}%
\endgroup%
}}\in\SB\ar@{|->}[r]^-{\hat{(\cdot)}_\T}&
\vcenter{\hbox{\def\svgscale{0.8}
\begingroup%
  \makeatletter%
  \providecommand\color[2][]{%
    \errmessage{(Inkscape) Color is used for the text in Inkscape, but the package 'color.sty' is not loaded}%
    \renewcommand\color[2][]{}%
  }%
  \providecommand\transparent[1]{%
    \errmessage{(Inkscape) Transparency is used (non-zero) for the text in Inkscape, but the package 'transparent.sty' is not loaded}%
    \renewcommand\transparent[1]{}%
  }%
  \providecommand\rotatebox[2]{#2}%
  \ifx\svgwidth\undefined%
    \setlength{\unitlength}{116.19100289bp}%
    \ifx\svgscale\undefined%
      \relax%
    \else%
      \setlength{\unitlength}{\unitlength * \real{\svgscale}}%
    \fi%
  \else%
    \setlength{\unitlength}{\svgwidth}%
  \fi%
  \global\let\svgwidth\undefined%
  \global\let\svgscale\undefined%
  \makeatother%
  \begin{picture}(1,0.5065711)%
    \put(0,0){\includegraphics[width=\unitlength,page=1]{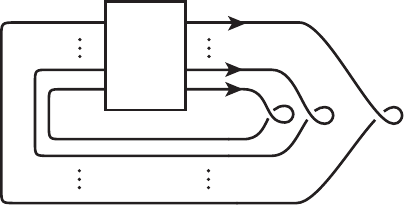}}%
    \put(0.32604906,0.35457836){\color[rgb]{0,0,0}\makebox(0,0)[lb]{\smash{$\beta$}}}%
  \end{picture}%
\endgroup%
}}=
\vcenter{\hbox{\def\svgscale{0.8}
\begingroup%
  \makeatletter%
  \providecommand\color[2][]{%
    \errmessage{(Inkscape) Color is used for the text in Inkscape, but the package 'color.sty' is not loaded}%
    \renewcommand\color[2][]{}%
  }%
  \providecommand\transparent[1]{%
    \errmessage{(Inkscape) Transparency is used (non-zero) for the text in Inkscape, but the package 'transparent.sty' is not loaded}%
    \renewcommand\transparent[1]{}%
  }%
  \providecommand\rotatebox[2]{#2}%
  \ifx\svgwidth\undefined%
    \setlength{\unitlength}{86.4000001bp}%
    \ifx\svgscale\undefined%
      \relax%
    \else%
      \setlength{\unitlength}{\unitlength * \real{\svgscale}}%
    \fi%
  \else%
    \setlength{\unitlength}{\svgwidth}%
  \fi%
  \global\let\svgwidth\undefined%
  \global\let\svgscale\undefined%
  \makeatother%
  \begin{picture}(1,0.62253458)%
    \put(0,0){\includegraphics[width=\unitlength,page=1]{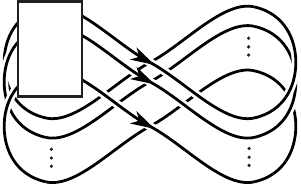}}%
    \put(0.13315604,0.43342761){\color[rgb]{0,0,0}\makebox(0,0)[lb]{\smash{$\beta$}}}%
  \end{picture}%
\endgroup%
}}\in\ST
}
\]
Refer to \cite[\S2.4]{KN} for details of the nonsingular case. Then $\hat{(\cdot)}_\T$ is well-defined since the defining relators for $\mathbf{SB}$ obviously correspond to the singular transverse Reidemeister moves and moreover conjugation $(BrC)$ and exchange move $(BrE)$ induce equivalences in $\ST$ as follows.
\begin{align*}
\hat{(\beta_1\beta_2)}_\T&=\vcenter{\hbox{\scriptsize\def\svgscale{0.8}
\begingroup%
  \makeatletter%
  \providecommand\color[2][]{%
    \errmessage{(Inkscape) Color is used for the text in Inkscape, but the package 'color.sty' is not loaded}%
    \renewcommand\color[2][]{}%
  }%
  \providecommand\transparent[1]{%
    \errmessage{(Inkscape) Transparency is used (non-zero) for the text in Inkscape, but the package 'transparent.sty' is not loaded}%
    \renewcommand\transparent[1]{}%
  }%
  \providecommand\rotatebox[2]{#2}%
  \ifx\svgwidth\undefined%
    \setlength{\unitlength}{86.4000001bp}%
    \ifx\svgscale\undefined%
      \relax%
    \else%
      \setlength{\unitlength}{\unitlength * \real{\svgscale}}%
    \fi%
  \else%
    \setlength{\unitlength}{\svgwidth}%
  \fi%
  \global\let\svgwidth\undefined%
  \global\let\svgscale\undefined%
  \makeatother%
  \begin{picture}(1,0.62253336)%
    \put(0,0){\includegraphics[width=\unitlength,page=1]{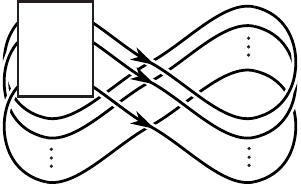}}%
    \put(0.06944444,0.45152638){\color[rgb]{0,0,0}\makebox(0,0)[lb]{\smash{$\beta_1\beta_2$}}}%
  \end{picture}%
\endgroup%
}}
=\vcenter{\hbox{\scriptsize\def\svgscale{0.8}
\begingroup%
  \makeatletter%
  \providecommand\color[2][]{%
    \errmessage{(Inkscape) Color is used for the text in Inkscape, but the package 'color.sty' is not loaded}%
    \renewcommand\color[2][]{}%
  }%
  \providecommand\transparent[1]{%
    \errmessage{(Inkscape) Transparency is used (non-zero) for the text in Inkscape, but the package 'transparent.sty' is not loaded}%
    \renewcommand\transparent[1]{}%
  }%
  \providecommand\rotatebox[2]{#2}%
  \ifx\svgwidth\undefined%
    \setlength{\unitlength}{86.4000001bp}%
    \ifx\svgscale\undefined%
      \relax%
    \else%
      \setlength{\unitlength}{\unitlength * \real{\svgscale}}%
    \fi%
  \else%
    \setlength{\unitlength}{\svgwidth}%
  \fi%
  \global\let\svgwidth\undefined%
  \global\let\svgscale\undefined%
  \makeatother%
  \begin{picture}(1,0.62939053)%
    \put(0,0){\includegraphics[width=\unitlength,page=1]{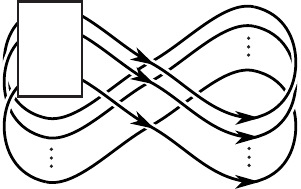}}%
    \put(0.09964077,0.439865){\color[rgb]{0,0,0}\makebox(0,0)[lb]{\smash{$\beta_1$}}}%
    \put(0.65277774,0.57875389){\color[rgb]{0,0,0}\makebox(0,0)[lb]{\smash{}}}%
    \put(0,0){\includegraphics[width=\unitlength,page=2]{T_closure_conjugate_2.pdf}}%
    \put(0.77035007,0.13430947){\color[rgb]{0,0,0}\makebox(0,0)[lb]{\smash{$\beta_2$}}}%
  \end{picture}%
\endgroup%
}}
=\vcenter{\hbox{\scriptsize\def\svgscale{0.8}
\begingroup%
  \makeatletter%
  \providecommand\color[2][]{%
    \errmessage{(Inkscape) Color is used for the text in Inkscape, but the package 'color.sty' is not loaded}%
    \renewcommand\color[2][]{}%
  }%
  \providecommand\transparent[1]{%
    \errmessage{(Inkscape) Transparency is used (non-zero) for the text in Inkscape, but the package 'transparent.sty' is not loaded}%
    \renewcommand\transparent[1]{}%
  }%
  \providecommand\rotatebox[2]{#2}%
  \ifx\svgwidth\undefined%
    \setlength{\unitlength}{86.4000001bp}%
    \ifx\svgscale\undefined%
      \relax%
    \else%
      \setlength{\unitlength}{\unitlength * \real{\svgscale}}%
    \fi%
  \else%
    \setlength{\unitlength}{\svgwidth}%
  \fi%
  \global\let\svgwidth\undefined%
  \global\let\svgscale\undefined%
  \makeatother%
  \begin{picture}(1,0.62253336)%
    \put(0,0){\includegraphics[width=\unitlength,page=1]{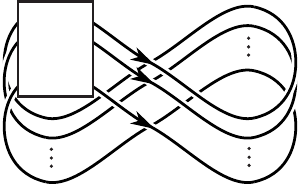}}%
    \put(0.06944444,0.45152638){\color[rgb]{0,0,0}\makebox(0,0)[lb]{\smash{$\beta_2\beta_1$}}}%
  \end{picture}%
\endgroup%
}}=\hat{(\beta_2\beta_1)}_\T,\\
\hat{(\beta_1\sigma_n\beta_2\sigma_n^{-1})}_\T
&=\vcenter{\hbox{\scriptsize\def\svgscale{0.8}
\begingroup%
  \makeatletter%
  \providecommand\color[2][]{%
    \errmessage{(Inkscape) Color is used for the text in Inkscape, but the package 'color.sty' is not loaded}%
    \renewcommand\color[2][]{}%
  }%
  \providecommand\transparent[1]{%
    \errmessage{(Inkscape) Transparency is used (non-zero) for the text in Inkscape, but the package 'transparent.sty' is not loaded}%
    \renewcommand\transparent[1]{}%
  }%
  \providecommand\rotatebox[2]{#2}%
  \ifx\svgwidth\undefined%
    \setlength{\unitlength}{111.49137846bp}%
    \ifx\svgscale\undefined%
      \relax%
    \else%
      \setlength{\unitlength}{\unitlength * \real{\svgscale}}%
    \fi%
  \else%
    \setlength{\unitlength}{\svgwidth}%
  \fi%
  \global\let\svgwidth\undefined%
  \global\let\svgscale\undefined%
  \makeatother%
  \begin{picture}(1,0.50902965)%
    \put(0,0){\includegraphics[width=\unitlength,page=1]{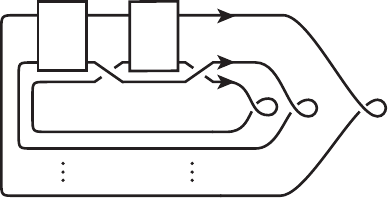}}%
    \put(0.11271659,0.38366199){\color[rgb]{0,0,0}\makebox(0,0)[lb]{\smash{$\beta_1$}}}%
    \put(0.34950619,0.38366199){\color[rgb]{0,0,0}\makebox(0,0)[lb]{\smash{$\beta_2$}}}%
  \end{picture}%
\endgroup%
}}
=\vcenter{\hbox{\scriptsize\def\svgscale{0.8}
\begingroup%
  \makeatletter%
  \providecommand\color[2][]{%
    \errmessage{(Inkscape) Color is used for the text in Inkscape, but the package 'color.sty' is not loaded}%
    \renewcommand\color[2][]{}%
  }%
  \providecommand\transparent[1]{%
    \errmessage{(Inkscape) Transparency is used (non-zero) for the text in Inkscape, but the package 'transparent.sty' is not loaded}%
    \renewcommand\transparent[1]{}%
  }%
  \providecommand\rotatebox[2]{#2}%
  \ifx\svgwidth\undefined%
    \setlength{\unitlength}{111.49137846bp}%
    \ifx\svgscale\undefined%
      \relax%
    \else%
      \setlength{\unitlength}{\unitlength * \real{\svgscale}}%
    \fi%
  \else%
    \setlength{\unitlength}{\svgwidth}%
  \fi%
  \global\let\svgwidth\undefined%
  \global\let\svgscale\undefined%
  \makeatother%
  \begin{picture}(1,0.50902965)%
    \put(0,0){\includegraphics[width=\unitlength,page=1]{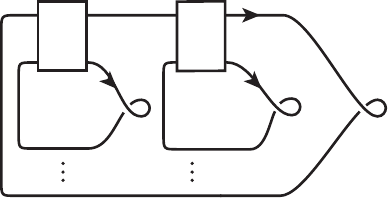}}%
    \put(0.11271659,0.38366199){\color[rgb]{0,0,0}\makebox(0,0)[lb]{\smash{$\beta_1$}}}%
    \put(0.47148874,0.38366199){\color[rgb]{0,0,0}\makebox(0,0)[lb]{\smash{$\beta_2$}}}%
  \end{picture}%
\endgroup%
}}\\
&=\vcenter{\hbox{\scriptsize\def\svgscale{0.8}
\begingroup%
  \makeatletter%
  \providecommand\color[2][]{%
    \errmessage{(Inkscape) Color is used for the text in Inkscape, but the package 'color.sty' is not loaded}%
    \renewcommand\color[2][]{}%
  }%
  \providecommand\transparent[1]{%
    \errmessage{(Inkscape) Transparency is used (non-zero) for the text in Inkscape, but the package 'transparent.sty' is not loaded}%
    \renewcommand\transparent[1]{}%
  }%
  \providecommand\rotatebox[2]{#2}%
  \ifx\svgwidth\undefined%
    \setlength{\unitlength}{111.49137846bp}%
    \ifx\svgscale\undefined%
      \relax%
    \else%
      \setlength{\unitlength}{\unitlength * \real{\svgscale}}%
    \fi%
  \else%
    \setlength{\unitlength}{\svgwidth}%
  \fi%
  \global\let\svgwidth\undefined%
  \global\let\svgscale\undefined%
  \makeatother%
  \begin{picture}(1,0.50902965)%
    \put(0,0){\includegraphics[width=\unitlength,page=1]{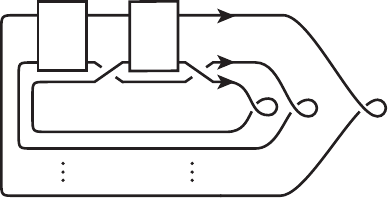}}%
    \put(0.11271659,0.38366199){\color[rgb]{0,0,0}\makebox(0,0)[lb]{\smash{$\beta_1$}}}%
    \put(0.34950619,0.38366199){\color[rgb]{0,0,0}\makebox(0,0)[lb]{\smash{$\beta_2$}}}%
  \end{picture}%
\endgroup%
}}
=\hat{(\beta_1\sigma_n^{-1}\beta_2\sigma_n)}_\T.
\end{align*}

Indeed, any singular transverse link $T\in\ST$ can be represented by a transverse closure $\hat\beta_\T$ of a singular braid $\beta=\beta(T)\in\SB$ as follows.
For given $T\in\ST$, we may assume that its front projection $D=\pi_F(T)$ is regular. Then we perform the following planar isotopies on $D$ and denote the resulting diagram by $D'$.
\[
\xymatrix@C=1pc@R=2pc{
D\ar[d]& & \vcenter{\hbox{\includegraphics{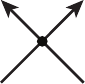}}}\ar@{|->}[d] & \vcenter{\hbox{\includegraphics{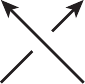}}}\ar@{|->}[d] & 
\vcenter{\hbox{\includegraphics{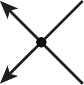}}} \ar@{|->}[d] &
\vcenter{\hbox{\includegraphics{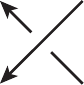}}}\ar@{|->}[d] &
\vcenter{\hbox{\includegraphics{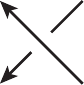}}}\ar@{|->}[d]\\
D' & & \vcenter{\hbox{\includegraphics{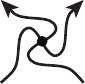}}} & 
\vcenter{\hbox{\includegraphics{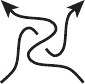}}} & 
\vcenter{\hbox{\includegraphics{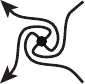}}} &
\vcenter{\hbox{\includegraphics{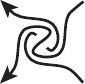}}} &
\vcenter{\hbox{\includegraphics{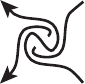}}}
}
\]
Now we split $D'$ by cutting all vertical tangencies, and we call arcs from the right to the left {\em backward arcs}.
Notice that each backward arc looks like $\vcenter{\hbox{\includegraphics{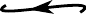}}}$ since there are no upward vertical tangencies. Moreover, after performing the above isotopies, every backward arc is only involving nonsingular crossings and lying {\em below} to all the other arcs.
Especially, all backward arcs are disjoint as shown in Figure~\ref{fig:backwardarcs}, and we can pull both ends of backward arcs to the left and right, respectively, by keeping backward arcs below and disjoint so that their ends are lying sufficiently far from their original positions. See Figure~\ref{fig:backwardarcs_2}.
Now by pulling down all backward arcs in $\mathcal{A}$, we can obtain a transversely closed braid $\hat\beta_\T$ for some $\beta=\beta(D')\in\mathbf{SB}$. See Figure~\ref{fig:backwardarcs_3}.

Remark that $\beta$ is not unique since it depends on the way how to pull end points of backward arcs.
More precisely, if we pull backward arcs to two ways, then the resulting braids $\beta_1$ and $\beta_2$ are different as much as conjugation. However, $\beta(D')$ as an element in $\SB$ is well-defined since $\SB = \mathbf{SB}/\{(BrC), (BrE)\}$.

\begin{figure}[ht]
\subfigure[
\label{fig:backwardarcs}]{
$D'=\vcenter{\hbox{\includegraphics{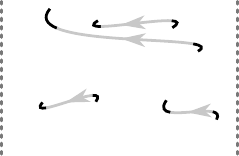}}}$
}{\ \ $=\joinrel=$}
\subfigure[
\label{fig:backwardarcs_2}]{
$\vcenter{\hbox{\includegraphics{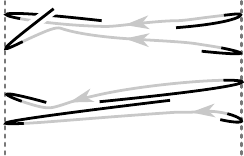}}}$
}{\ \ $=\joinrel=$}
\subfigure[
\label{fig:backwardarcs_3}]{
$\vcenter{\hbox{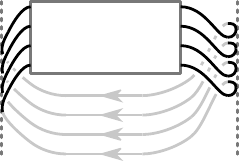}}=\hat\beta_\T$
}
\caption{Backward arcs and a closed braid presentation}
\end{figure}

\begin{thm}\label{thm:SBtoST}
The map $\hat{(\cdot)}_\T$ induces bijections
\[
\hat{(\cdot)}_\T:\B/\{(Br S_+)\}\to \T,\qquad
\hat{(\cdot)}_\T:\SB/\{(BrS_+)\}\to \ST.
\]
%
\end{thm}
Remark that nonsingular cases are already covered by \cite{Ben,OS}\label{thm:SBST}.
\begin{proof}
It is easy to see that $(Br S_+)$ induces an equivalence in $\ST$, and moreover, $\hat{(\cdot)}_\T$ is surjective since for any $T\in\ST$ and its regular front projection $D=\pi_F(T)$, one can construct $\beta\in\mathbf{SB}$ such that $\hat{\beta}_\T = T\in\ST$ as above.

Notice that the map $\ST\to\SB/\{(BrS_+)\}$ defined by $T\mapsto\beta$ is a candidate for the inverse of $\hat{(\cdot)}_\T$. Therefore it suffices to prove the well-definedness of $\beta$ in $\SB$.

Suppose that two regular front projections $D_1$ and $D_2$ for singular transverse links $T_1$ and $T_2$ are different as much as one transverse Reidemeister move. Then we need to prove that corresponding braids $\beta_1$ and $\beta_2$ are the same up to $(+)$-stabilization $(BrS_+)$ in $\SB$.
Recall that the construction of $\beta_i$ depends on where vertical tangencies lie.

As seen in Figure~\ref{fig:transverseRM}, the move $(TRM0_1)$ increases or decreases two vertical tangencies, or equivalently, one backward arc. 
Hence it is easy to check that $(TRM0_1)$ effects on the braid $\beta$ as the positive (de)stabilization $(BrS_+)^{\pm1}$ up to conjugacy.
On the other hand, suppose $D_1$ and $D_2$ differ by $(TRM0_2)$ or $(TRM0_3)$. Then two diagrams $D'_1$ and $D'_2$ are the same or differ by $(TRM0_1)$. Therefore the map $D\mapsto\beta\in\mathcal{SB}$ is well-defined under the planar isotopy.

For the other transverse Reidemeister moves, there are many cases according to the orientations of strands, but it is a routine process to check the well-definedness of $D\mapsto\beta$. 
For example, whatever the orientation on $(TRM2_1)$ is given, there are no backward overcrossings. Hence $D\mapsto\beta$ is well-defined under $(TRM2_1)$. 

Also, we illustrate $(TRM2_2)$ as follows. Suppose that $D_1$ and $D_2$ differ by one $(TRM2_2)$ move. 
Then as shown in Figure~\ref{fig:TRM2_2}, $D_2$ can be obtained from $D_1'$ by the sequence of transverse Reidemeister moves $(TRM0_1)^{\pm1}$ and $(TRM2_1)$, which do not change the braid $\beta$.
Therefore the map $D\mapsto\beta(D)$ is well-defined under the move $(TRM2_2)$ as well.

The rest of the proof can be done by the similar process as above, and we omit the detail.
\end{proof}

\begin{figure}[ht]
\[
\xymatrix@C=3pc{
D_1=\vcenter{\hbox{\includegraphics{TRM2_2_2.pdf}}}\ar@<3ex>@{<->}[d]_{(TRM2_2)}\ar[r]^-{(\cdot)'}_-{(TRM0_*)} & D_1'=\vcenter{\hbox{\includegraphics{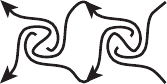}}}
\ar@{=}[r] &\vcenter{\hbox{\includegraphics{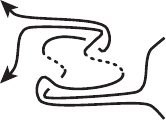}}} \ar@{<->}[d]^{(TRM0_1)^2}\\
D_2=\vcenter{\hbox{\includegraphics{TRM2_2_1.pdf}}} &\vcenter{\hbox{\includegraphics{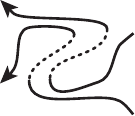}}}\ar@{<->}[l]_-{(TRM0_1)^2} & \vcenter{\hbox{\includegraphics{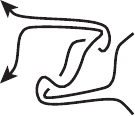}}}\ar@{<->}[l]_{(TRM2_1)}
}
\]
\caption{A decomposition of transverse Reidemeister move $(TRM2_2)$}
\label{fig:TRM2_2}
\end{figure}

\subsubsection{Transverse stabilizations and the map onto $\SK$}
Let us consider a canonical map $\|\cdot\|:\ST\to\SK$ defined by taking the underlying smooth singular links for given transverse links.
Similar to $\SL$, this map is also (infinitely) many-to-one and we consider moves, called {\em transverse (de)stabilization $(TrS)^{\pm1}$}, which change the singular transverse link type but preserve the underlying singular link type.
\begin{defn}[Transverse (de)stabilizations]
For $T\in\ST$, a {\em transverse stabilization} $(TS)$ on $T$ adds a {\em double-kink} at an upward vertical tangency of $T$ as follows, and its inverse operation is called the {\em transverse destabilization} and denoted by $(TS)^{-1}$.
\[
\xymatrix@C=3pc{
\vcenter{\hbox{\includegraphics{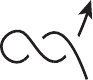}}}\quad\ar@<-.5ex>[r]_-{(TS)^{-1}}&\quad
\vcenter{\hbox{\includegraphics{K_bending_LU.pdf}}}\quad\ar@<-.5ex>[l]_-{(TS)}&
\vcenter{\hbox{\includegraphics{K_bending_RU.pdf}}}\ar@<.5ex>[r]^-{(TS)}\quad&\quad
\vcenter{\hbox{\includegraphics{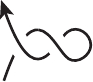}}}\ar@<.5ex>[l]^-{(TS)^{-1}}
}
\]
\end{defn}

Then since $(LS_+)$ makes two down cusps in addition, they correspond to two more kinks via the positive push-off, which is the same as $(TS)$. 
On the other hand, the braid negative stabilization $(BrS_-)$ will also become the transverse stabilization $(TS)$ via $\hat{(\cdot)}_\T$.
\[
\xymatrix@C=3pc{
\vcenter{\hbox{\includegraphics{L_cusp_RU.pdf}}}\ar[r]^{(\cdot)^+}\ar[d]_{(LS_+)} & \vcenter{\hbox{\includegraphics{K_bending_RU.pdf}}}\ar[d]^{(TS)} &
\vcenter{\hbox{\def\svgscale{0.8}
\begingroup%
  \makeatletter%
  \providecommand\color[2][]{%
    \errmessage{(Inkscape) Color is used for the text in Inkscape, but the package 'color.sty' is not loaded}%
    \renewcommand\color[2][]{}%
  }%
  \providecommand\transparent[1]{%
    \errmessage{(Inkscape) Transparency is used (non-zero) for the text in Inkscape, but the package 'transparent.sty' is not loaded}%
    \renewcommand\transparent[1]{}%
  }%
  \providecommand\rotatebox[2]{#2}%
  \ifx\svgwidth\undefined%
    \setlength{\unitlength}{63.31499634bp}%
    \ifx\svgscale\undefined%
      \relax%
    \else%
      \setlength{\unitlength}{\unitlength * \real{\svgscale}}%
    \fi%
  \else%
    \setlength{\unitlength}{\svgwidth}%
  \fi%
  \global\let\svgwidth\undefined%
  \global\let\svgscale\undefined%
  \makeatother%
  \begin{picture}(1,0.50673618)%
    \put(0,0){\includegraphics[width=\unitlength,page=1]{B_braid.pdf}}%
    \put(0.39233038,0.17699114){\color[rgb]{0,0,0}\makebox(0,0)[lb]{\smash{$\beta$}}}%
  \end{picture}%
\endgroup%
}}\ar[r]^{\hat{(\cdot)}_\T}\ar[d]_{(BrS_-)} & 
\vcenter{\hbox{\def\svgscale{0.8}
\begingroup%
  \makeatletter%
  \providecommand\color[2][]{%
    \errmessage{(Inkscape) Color is used for the text in Inkscape, but the package 'color.sty' is not loaded}%
    \renewcommand\color[2][]{}%
  }%
  \providecommand\transparent[1]{%
    \errmessage{(Inkscape) Transparency is used (non-zero) for the text in Inkscape, but the package 'transparent.sty' is not loaded}%
    \renewcommand\transparent[1]{}%
  }%
  \providecommand\rotatebox[2]{#2}%
  \ifx\svgwidth\undefined%
    \setlength{\unitlength}{116.19100289bp}%
    \ifx\svgscale\undefined%
      \relax%
    \else%
      \setlength{\unitlength}{\unitlength * \real{\svgscale}}%
    \fi%
  \else%
    \setlength{\unitlength}{\svgwidth}%
  \fi%
  \global\let\svgwidth\undefined%
  \global\let\svgscale\undefined%
  \makeatother%
  \begin{picture}(1,0.5065711)%
    \put(0,0){\includegraphics[width=\unitlength,page=1]{T_closure.pdf}}%
    \put(0.32604906,0.35457836){\color[rgb]{0,0,0}\makebox(0,0)[lb]{\smash{$\beta$}}}%
  \end{picture}%
\endgroup%
}}\ar[d]^{(TS)}\\
\vcenter{\hbox{\includegraphics{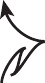}}}\ar[r]^{(\cdot)^+} & \vcenter{\hbox{\includegraphics{T_stabilization_R.pdf}}} &
\vcenter{\hbox{\def\svgscale{0.8}
\begingroup%
  \makeatletter%
  \providecommand\color[2][]{%
    \errmessage{(Inkscape) Color is used for the text in Inkscape, but the package 'color.sty' is not loaded}%
    \renewcommand\color[2][]{}%
  }%
  \providecommand\transparent[1]{%
    \errmessage{(Inkscape) Transparency is used (non-zero) for the text in Inkscape, but the package 'transparent.sty' is not loaded}%
    \renewcommand\transparent[1]{}%
  }%
  \providecommand\rotatebox[2]{#2}%
  \ifx\svgwidth\undefined%
    \setlength{\unitlength}{86.332bp}%
    \ifx\svgscale\undefined%
      \relax%
    \else%
      \setlength{\unitlength}{\unitlength * \real{\svgscale}}%
    \fi%
  \else%
    \setlength{\unitlength}{\svgwidth}%
  \fi%
  \global\let\svgwidth\undefined%
  \global\let\svgscale\undefined%
  \makeatother%
  \begin{picture}(1,0.47310116)%
    \put(0,0){\includegraphics[width=\unitlength,page=1]{B_stabilization_-.pdf}}%
    \put(0.2926837,0.23847346){\color[rgb]{0,0,0}\makebox(0,0)[lb]{\smash{$\beta$}}}%
  \end{picture}%
\endgroup%
}}\ar[r]^{\hat{(\cdot)}_\T} & 
\vcenter{\hbox{\def\svgscale{0.8}
\begingroup%
  \makeatletter%
  \providecommand\color[2][]{%
    \errmessage{(Inkscape) Color is used for the text in Inkscape, but the package 'color.sty' is not loaded}%
    \renewcommand\color[2][]{}%
  }%
  \providecommand\transparent[1]{%
    \errmessage{(Inkscape) Transparency is used (non-zero) for the text in Inkscape, but the package 'transparent.sty' is not loaded}%
    \renewcommand\transparent[1]{}%
  }%
  \providecommand\rotatebox[2]{#2}%
  \ifx\svgwidth\undefined%
    \setlength{\unitlength}{111.49200289bp}%
    \ifx\svgscale\undefined%
      \relax%
    \else%
      \setlength{\unitlength}{\unitlength * \real{\svgscale}}%
    \fi%
  \else%
    \setlength{\unitlength}{\svgwidth}%
  \fi%
  \global\let\svgwidth\undefined%
  \global\let\svgscale\undefined%
  \makeatother%
  \begin{picture}(1,0.52792131)%
    \put(0,0){\includegraphics[width=\unitlength,page=1]{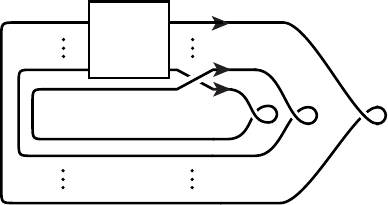}}%
    \put(0.30748497,0.40641491){\color[rgb]{0,0,0}\makebox(0,0)[lb]{\smash{$\beta$}}}%
  \end{picture}%
\endgroup%
}}
}
\]

Therefore $(\cdot)^+$ and $\hat{(\cdot)}_\T$ transform $(LS_+)$ and $(BrS_-)$ into $(TS)$, respectively,
\[
(LS_+)^+ = (TS)=\hat{(BrS_-)}_\T
\]
and we have the following theorem.

\begin{thm}\label{thm:STSK}
The map $\|\cdot\|$ induces bijections
\[
\|\cdot\|:\T/\{(TS)\}\to\K,\qquad \|\cdot\|:\ST/\{(TS)\}\to\SK.
\]
%
\end{thm}
For nonsingular cases, this is already known by \cite[Theorem~5.4]{FT}.
\begin{proof}
Since $(TS)$ is the same as $(LS_+)$ via the pushoff $(\cdot)^+$, we are done by Proposition~\ref{prop:SLSK} and Theorem~\ref{thm:SLST}.

Alternatively, this also follows from Theorem~\ref{thm:SBtoSK} and Theorem~\ref{thm:SBtoST} by using the relation $\hat{(BrS_-)}_\T=(TS)$.
\end{proof}

\section{Singular grid diagrams}\label{sec:extended_grid}
\subsection{Grid diagrams}
We first briefly review the definition of grid diagrams and the maps defined on the set $\G$ of grid diagrams.
Let $\mathbf{T}$ be the set of eight grid tiles depicted in Figure~\ref{fig:local_grid} together with all possible orientations.
Then the number of elements of $\mathbf{T}$ is precisely 17.
We call $\mathbf{T}$ the {\em tileset}.

\begin{defn}[(Pseudo) Grid diagrams]
A {\em pseudo grid diagram} $G$ of size $n$ is an $n\times n$ square grid of tiles in $\mathbf{T}$ satisfying the following.
\begin{enumerate}
\item Arcs in any pair of two contiguous tiles in $G$ must be {\em suitably-connected},
\item each row and column of $G$ has {\em at most} two corners.
\end{enumerate}

Then a {\em grid diagram} $G$ can be defined as a pseudo grid diagram such that arcs in $G$ do not meet the boundary of $G$ and there are no rows and columns without corners.
%
\end{defn}

\begin{rmk}
A pseudo grid diagram is a grid-diagrammatic analog of a tangle in $\K$, and by changing tileset $\mathbf{T}$, we may obtain a new set of grid diagrams.
\end{rmk}

\begin{rmk}\label{rmk:matrix}
Any grid diagram $G\in\G$ is completely determined by its oriented corners. That is, $G$ of size $n$ can be encoded as a square matrix of size $n$ whose entry is either empty or an oriented corner.
\end{rmk}

Let $G\in\G$ be given. Since $G$ itself is an oriented link diagram, it defines a map $\|\cdot\|_\K:\G\to\K$ by $\|G\|_\K=K$ where $K$ is a link given by a link diagram $G$.

A Legendrian link $L=\|G\|_\L$ is given by the front projection $\pi_F(L)$ which is obtained as follows.
We rotate $G$ 45 degrees counterclockwise, smooth up and down corners and turn left and right corners into cusps. The pictorial definition of $\|\cdot\|_{\L}$ is depicted in Figure~\ref{fig:G2L}.

\begin{figure}[ht]
\[
\xymatrix@C=0.5pc@R=1pc{
\G\ar[d]_{\|\cdot\|_\L} &  &
\vcenter{\hbox{\includegraphics[scale=0.8]{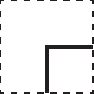}}}\ar@{|->}[d]& 
\vcenter{\hbox{\includegraphics[scale=0.8]{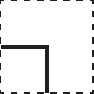}}}\ar@{|->}[d]& 
\vcenter{\hbox{\includegraphics[scale=0.8]{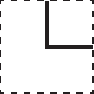}}}\ar@{|->}[d]& 
\vcenter{\hbox{\includegraphics[scale=0.8]{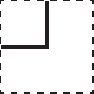}}}\ar@{|->}[d] & 
\vcenter{\hbox{\includegraphics[scale=0.8]{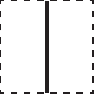}}}\ar@{|->}[d]& 
\vcenter{\hbox{\includegraphics[scale=0.8]{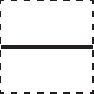}}}\ar@{|->}[d]&
\vcenter{\hbox{\includegraphics[scale=0.8]{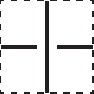}}}\ar@{|->}[d]\\
\L& &
\vcenter{\hbox{\includegraphics[scale=0.8]{L_cusp_L.pdf}}}&
\vcenter{\hbox{\includegraphics[scale=0.8]{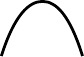}}}& 
\vcenter{\hbox{\includegraphics[scale=0.8]{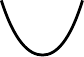}}}& 
\vcenter{\hbox{\includegraphics[scale=0.8]{L_cusp_R.pdf}}}&
\vcenter{\hbox{\includegraphics[scale=0.8]{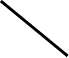}}}& 
\vcenter{\hbox{\includegraphics[scale=0.8]{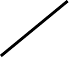}}}&
\vcenter{\hbox{\includegraphics[scale=0.8]{K_crossing.pdf}}}
}
\]
\caption{A pictorial definition of $\|\cdot\|_\L:\G\to\L$}
\label{fig:G2L}
\end{figure}

To define $\|\cdot\|_{\B}$, we will use the set $\bar{\mathbf{B}}$ of non-singular rectilinear braid diagrams.
For given $G\in\G$, let $H(G)=\{h_1,\dots,h_\ell\}$ be the set of all horizontal segments in $G$ whose orientation is {\em reversed}, namely, from right to left. We call these {\em reversed segments of type $\rm(I)$}.
\begin{align*}
\tag{I}\vcenter{\hbox{\includegraphics{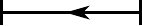}}}\subset G
\end{align*}
Here, small bars at both ends represent corners. Then a {\em flip $G^\vee\in\bar{\mathbf{B}}$} of $G$ is a rectilinear braid diagram defined by {\em flipping} all reversed horizontal segments in $G$ as described in \cite{MM, NT}.
\begin{align*}
\tag{$\rm I^\vee$}
\left(
\vcenter{\hbox{\includegraphics{horizontal_segment_1.pdf}}}
\right)^\vee:=
\vcenter{\hbox{\includegraphics{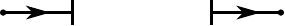}}}\subset G^\vee
\end{align*}
The small dots in flipped segments of $\rm(I^\vee)$ mean the left and right ends of the rectilinear braid diagram $G^\vee$, respectively.

We define $\|\cdot\|_{\B}$ as the composition of the flip map $(\cdot)^\vee$ and the map $sl:\bar{\mathbf{B}}\to\B$, which slants all vertical segments as defined before.
\[
\|\cdot\|_{\B}:=sl\left((\cdot)^\vee\right):\G\stackrel{(\cdot)^\vee}{\longrightarrow}\bar{\mathbf{B}}\stackrel{sl}{\longrightarrow}\B
\]

Finally, for any $G\in\G$, we define $\|\cdot\|_\T:\G\to\T$ as the composition 
\[
\|G\|_\T:=\hat{(\|G\|_\B)}_\T.
\]
Then as mentioned before, it is known by \cite{KN} that 
$\|G\|_\T = (\|G\|_\L)^+$, and so these maps fit into the commutative diagram in Figure~\ref{fig:maps}.

The set $\G$ of grid diagrams admits {\em elementary moves} consisting of {\em translations}, {\em commutations} and four types of {\em (de)stabilizations}. These moves are defined as depicted in Figures~\ref{fig:trans}, \ref{fig:commu} and \ref{fig:stabil}.
Then for each $\mathcal{C}\in\{\B,\L,\T,\K\}$, Proposition~\ref{prop:COSTNT} completely classifies the set of moves on $\G$ which induce equivalences in $\mathcal{C}$ via the map $\|\cdot\|_\mathcal{C}$. 

\subsection{Singular grid diagrams}
We first define singular grid diagrams by introducing a {\em singular tile} as follows.
\begin{defn}[Singular tiles]\label{defn:singulartile}
A {\em singular tile} $t_s$ is a tile of size $1\times 1$ such that
\begin{enumerate}
\item $t_s$ has exactly one end point on each side,
\item $\|t_s\|_\SK$ is given by a singular tangle in $\SK$ consisting of two arcs with exactly one singular point.
\end{enumerate}
\end{defn}

For a singular tile $t_s$, there are always 4 ways of orientations that it has. We usually denote oriented singular tiles by
$t_s^N, t_s^E, t_s^W$, and $t_s^S$.

\begin{rmk}
The characters `$N$', `$E$', `$W$' and `$S$' may {\em not} be related at all with the actual directions where arcs in $t_s$ are pointing.
\end{rmk}

\begin{defn}[Singular grid diagrams]
Let $\mathbf{T}_s:=\mathbf{T}\cup\{t_s^N, t_s^E, t_s^W, t_s^S\}$.
A {\em singular grid diagram extended by $t_s$} is a grid diagram which uses $\mathbf{T}_s$ as a tileset instead of $\mathbf{T}$.
We denote by $\SG_s$ the set of all singular grid diagrams extended by $t_s$.
\end{defn}

Then by using the singular tangle $\|t_s\|_\SK$, a function $\|\cdot\|_\SK:\SG_s\to\SK$ is well-defined.

\begin{rmk}\label{rmk:matrix_SG}
Similar to Remark~\ref{rmk:matrix}, the information of oriented corners and singular tiles in $G\in\SG$ completely determines $G$ itself.
\end{rmk}

\begin{defn}[Resolutions on $\SG_s$]
A singular tile $t_s$ is called {\em resolutive} if it has resolutions, that is, for each $\eta\in\{+,-,0\}$ and $*\in\{N,E,W,S\}$, there exists a pseudo grid diagram $\mathcal{R}_\eta(t_s^*)$ of no longer having any singular point. Then we say that $\SG_s$ {\em admits} resolutions.
\end{defn}

Indeed, the resolution on $\SG_s$ is defined as follows. Suppose that $t_s$ is resolutive.
Let $G\in\SG_s$ and $t_s^*$ be a singular tile contained in $G$ for some $*\in\{N,E,W,S\}$.
Then for each $\eta\in\{+,-,0\}$, the $\eta$-resolution $\mathcal{R}_\eta(G,t_s^*)$ is a singular grid diagram obtained from $G$ by replacing a singular tile $t_s^*$ with the pseudo grid diagram $\mathcal{R}_\eta(t_s^*)$.
We denote this replacement as
\begin{align}
\mathcal{R}_\eta(G,t_s^*):= (G\setminus t_s^*)\sqcup \mathcal{R}_\eta(t_s^*).
\end{align}

\begin{defn}[A unified description]\label{defn:unified}
Let $\SG_s$ be the set of singular grid diagrams which admits resolutions.
We say that $\SG_s$ gives a {\em unified description} if it satisfies the following.
Let $*\in\{N,E,W,S\}$ and $\mathcal{C}\in\{\B,\L,\T\}$.
\begin{enumerate}
\item There exist 
\begin{itemize}
\item [(i)] a singular braid $\|t_s^*\|_\SB$, 
\item [(ii)] a singular Legendrian tangle $\|t_s^*\|_\SL$,
\item [(iii)] a singular transverse tangle $\|t_s^*\|_\ST$.
\end{itemize} 
\item Each $\|t_s^*\|_\mathcal{SC}$ canonically produces a map
\begin{align*}
\|\cdot\|_\mathcal{SC}&:\SG_s\to\mathcal{SC}
\end{align*}
which extends $\|\cdot\|_\mathcal{C}:\G\to\mathcal{C}$. 
\item Resolutions commute with the maps defined on $\SG_s$. i.e., 
\begin{align*}
\mathcal{R}_\eta(\|t_s^*\|_\mathcal{SC})&=\|\mathcal{R}_\eta(t_s^*)\|_\mathcal{C}.
\end{align*}
\item The diagram below is commutative,
\[
\xymatrix@C=5pc{
\SG_s \ar[r]^{\|\cdot\|_{\SL}}\ar[d]_{\|\cdot\|_{\SB}}\ar@/^4.5pc/[ddrr]^{\|\cdot\|_\SK} \ar[rd]^{\|\cdot\|_{\ST}}
& \SL \ar[d]^{(\cdot)^+}\ar[ddr]^{\|\cdot\|}\\
\SB\ar[r]^{\hat{(\cdot)}_\T}\ar[rrd]_{\widehat{(\cdot)}} & \ST\ar[dr]^{\|\cdot\|}\\
& & \SK
}
\]
\end{enumerate}
\end{defn}

As discussed in Section~\ref{sec:singular links}, there are two kinds of local pictures presenting singular points in $\SL$, $\SB$, $\ST$ and $\SK$, which are 
\[
\vcenter{\hbox{\includegraphics{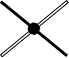}}}\in \SB, \ST, \SK\quad\text{ and }\quad
\vcenter{\hbox{\includegraphics{L_singularcrossing.pdf}}}\in \SL.
\]
Hence there are two natural candidates of singular tiles
\[
t_\times=\vcenter{\hbox{\includegraphics[scale=0.7]{grid_tile_singularcrossing.pdf}}}\quad\text{ and }\quad
t_\bullet=\vcenter{\hbox{\includegraphics[scale=0.7]{grid_tile_Leg_singularcrossing.pdf}}},
\]
whose realizations in $\SK$ are
\begin{align*}
\|t_\times\|_\SK= \vcenter{\hbox{\includegraphics{K_singularcrossing.pdf}}},\qquad
\|t_\bullet\|_\SK=\vcenter{\hbox{\includegraphics{KRM5_1_1.pdf}}}=\vcenter{\hbox{\includegraphics{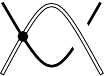}}}.
\end{align*}

Let $\SG_\times$ and $\SG_\bullet$ be two sets of singular grid diagrams extended by $t_\times$ and $t_\bullet$, respectively. Then two maps $\|\cdot\|_\SK:\SG_\times\to\SK$ and $\|\cdot\|_\SK:\SG_\bullet\to\SK$ are induced by $\|t_\times\|_\SK$ and $\|t_\bullet\|_\SK$, respectively.
Figure~\ref{fig:examplesinSG} shows examples of singular grid diagrams in $\SG_\times$ and $\SG_\bullet$.

\begin{figure}[ht]
\[
G_{3_1^{\times}}=\vcenter{\hbox{\includegraphics[scale=0.6]{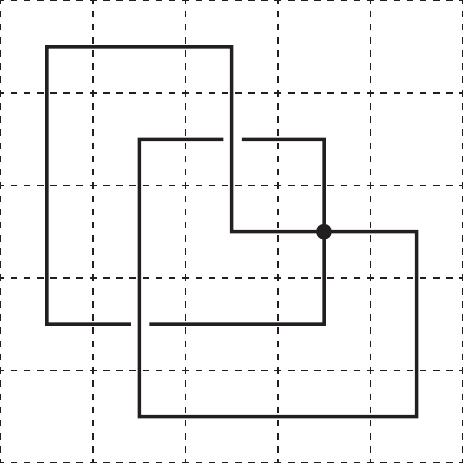}}}\in\SG_{\times}\qquad \qquad
G_{3_1^{\bullet}}=\vcenter{\hbox{\includegraphics[scale=0.6]{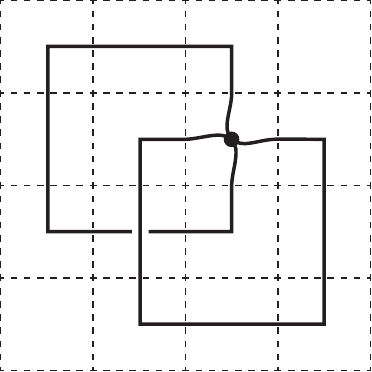}}}\in\SG_{\bullet}
\]
\[
\|G_{3_1^\times}\|_\SK=\|G_{3_1^\bullet}\|_\SK=\vcenter{\hbox{\includegraphics{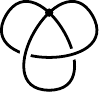}}}\in\SK
\]
\caption{Examples of extended grid diagrams in $\SG_{\times}$ and $\SG_{\bullet}$}
\label{fig:examplesinSG}
\end{figure}

\subsection{Singular grid diagrams extended by \texorpdfstring{$t_\bullet$}{t.}}
We define orientations and resolutions of $t_\bullet^*$ for $*\in\{N,E,W,S\}$ as depicted in Figure~\ref{fig:singularresolutions}.
Note that the naming convention and resolutions of $t_\bullet^*$ essentially come from those of $L_\bullet^*$ as described in Figure~\ref{fig:singularpoints} and Definition~\ref{defn:LegResolutions}.

\begin{figure}[ht]
\[
\begin{array}{l|c|c|c}
& \mathcal{R}_+ & \mathcal{R}_- &\mathcal{R}_0 \\
\hline&&&\\
t_\bullet^N=\vcenter{\hbox{\includegraphics[scale=0.7]{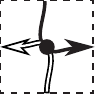}}} & 
\vcenter{\hbox{\includegraphics[scale=0.7]{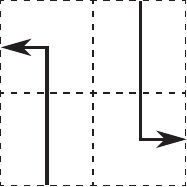}}} &
\vcenter{\hbox{\includegraphics[scale=0.7]{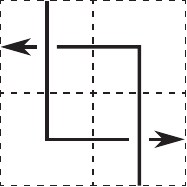}}} & 
\vcenter{\hbox{\includegraphics[scale=0.7]{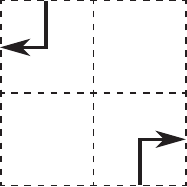}}}\\&&&\\
 \hline&&&\\
t_\bullet^E=\vcenter{\hbox{\includegraphics[scale=0.7]{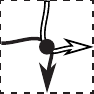}}} & 
\vcenter{\hbox{\includegraphics[scale=0.7]{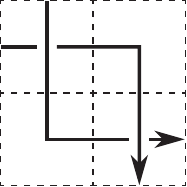}}} & 
\vcenter{\hbox{\includegraphics[scale=0.7]{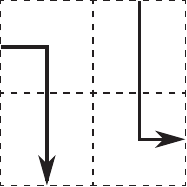}}} & 
\vcenter{\hbox{\includegraphics[scale=0.7]{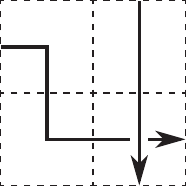}}}\\&&&\\
 \hline&&&\\
t_\bullet^W=\vcenter{\hbox{\includegraphics[scale=0.7]{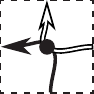}}} & 
\vcenter{\hbox{\includegraphics[scale=0.7]{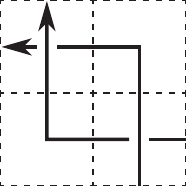}}} &
\vcenter{\hbox{\includegraphics[scale=0.7]{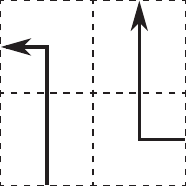}}} & 
\vcenter{\hbox{\includegraphics[scale=0.7]{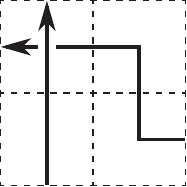}}} \\&&&\\
\hline&&&\\
t_\bullet^S=\vcenter{\hbox{\includegraphics[scale=0.7]{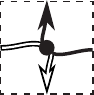}}} & 
\vcenter{\hbox{\includegraphics[scale=0.7]{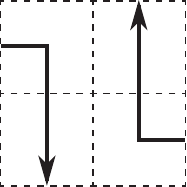}}} &
\vcenter{\hbox{\includegraphics[scale=0.7]{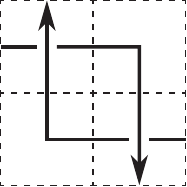}}} & 
\vcenter{\hbox{\includegraphics[scale=0.7]{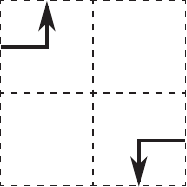}}} \\&&&\\
\hline
\end{array}
\]
\caption{Resolutions of oriented tiles $t_\bullet^*$}
\label{fig:singularresolutions}
\end{figure}

For each $\mathcal{C}\in\{\B,\L,\T\}$, we first define th map from $\SG_\bullet$ to $\SC$ and then prove the following theorem.

\begin{thm}[Theorem~\ref{thm:unified}]\label{thm:extension}
The set $\SG_\bullet$ gives a unified description in the sense of Definition~\ref{defn:unified}.
\end{thm}

\subsubsection{$\|\cdot\|_\SB:\SG_\bullet\to\SB$}
We will use the set $\bar{\mathbf{SB}}$ of singular rectilinear braid diagrams to define $\|\cdot\|_\SB$ as follows.
Let $G\in\SG_\bullet$ be given and $H(G)=\{h_1,\dots,h_\ell\}$ be the set of all horizontal segments in $G$ whose ends are at corners or singular points and orientations are reversed as before.
Then, including segments of type (I) described before, there are four types of reversed horizontal segments in $G$ as follows:
\[
\rm(I)\ \vcenter{\hbox{\includegraphics{horizontal_segment_1.pdf}}}\qquad
(II)\ \vcenter{\hbox{\includegraphics{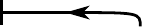}}}\qquad
(III)\ \vcenter{\hbox{\includegraphics{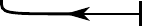}}}\qquad
(IV)\ \vcenter{\hbox{\includegraphics{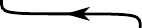}}}
\]
Here, small bars and curves at both ends represent corners and singular points, respectively.
Then a flip $G^\vee\in\bar{\mathbf{SB}}$ is defined as follows:
\begin{itemize}
\item [-] For a reversed horizontal segment of type (I), it is the same as described in $\rm(I^\vee)$.

\item [-] For a reversed horizontal segment $h$ of type (II), there are two cases according to the orientation of the another horizontal segment $h'$ which meets at the singular point in the right.
We define a flip of $h$ as follows.
\begin{align*}
\tag{$\rm II^\vee_1$}\left(
\vcenter{\hbox{
\begingroup%
  \makeatletter%
  \providecommand\color[2][]{%
    \errmessage{(Inkscape) Color is used for the text in Inkscape, but the package 'color.sty' is not loaded}%
    \renewcommand\color[2][]{}%
  }%
  \providecommand\transparent[1]{%
    \errmessage{(Inkscape) Transparency is used (non-zero) for the text in Inkscape, but the package 'transparent.sty' is not loaded}%
    \renewcommand\transparent[1]{}%
  }%
  \providecommand\rotatebox[2]{#2}%
  \ifx\svgwidth\undefined%
    \setlength{\unitlength}{49.81701776bp}%
    \ifx\svgscale\undefined%
      \relax%
    \else%
      \setlength{\unitlength}{\unitlength * \real{\svgscale}}%
    \fi%
  \else%
    \setlength{\unitlength}{\svgwidth}%
  \fi%
  \global\let\svgwidth\undefined%
  \global\let\svgscale\undefined%
  \makeatother%
  \begin{picture}(1,0.2430288)%
    \put(0,0){\includegraphics[width=\unitlength,page=1]{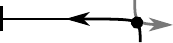}}%
    \put(0.19510979,0.16791582){\color[rgb]{0,0,0}\makebox(0,0)[lb]{\smash{$h$}}}%
    \put(0.80652209,0.15918137){\color[rgb]{0,0,0}\makebox(0,0)[lb]{\smash{$\color{gray}h'$}}}%
  \end{picture}%
\endgroup%
}}
\right)^\vee&=
\vcenter{\hbox{
\begingroup%
  \makeatletter%
  \providecommand\color[2][]{%
    \errmessage{(Inkscape) Color is used for the text in Inkscape, but the package 'color.sty' is not loaded}%
    \renewcommand\color[2][]{}%
  }%
  \providecommand\transparent[1]{%
    \errmessage{(Inkscape) Transparency is used (non-zero) for the text in Inkscape, but the package 'transparent.sty' is not loaded}%
    \renewcommand\transparent[1]{}%
  }%
  \providecommand\rotatebox[2]{#2}%
  \ifx\svgwidth\undefined%
    \setlength{\unitlength}{81.768bp}%
    \ifx\svgscale\undefined%
      \relax%
    \else%
      \setlength{\unitlength}{\unitlength * \real{\svgscale}}%
    \fi%
  \else%
    \setlength{\unitlength}{\svgwidth}%
  \fi%
  \global\let\svgwidth\undefined%
  \global\let\svgscale\undefined%
  \makeatother%
  \begin{picture}(1,0.30172679)%
    \put(0,0){\includegraphics[width=\unitlength,page=1]{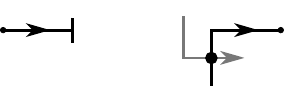}}%
    \put(0.12122312,0.26455989){\color[rgb]{0,0,0}\makebox(0,0)[lb]{\smash{$h^\vee$}}}%
    \put(0.86260827,0.06952328){\color[rgb]{0,0,0}\makebox(0,0)[lb]{\smash{$\color{gray}h'$}}}%
  \end{picture}%
\endgroup%
}}\\
\tag{$\rm II^\vee_2$}\left(
\vcenter{\hbox{
\begingroup%
  \makeatletter%
  \providecommand\color[2][]{%
    \errmessage{(Inkscape) Color is used for the text in Inkscape, but the package 'color.sty' is not loaded}%
    \renewcommand\color[2][]{}%
  }%
  \providecommand\transparent[1]{%
    \errmessage{(Inkscape) Transparency is used (non-zero) for the text in Inkscape, but the package 'transparent.sty' is not loaded}%
    \renewcommand\transparent[1]{}%
  }%
  \providecommand\rotatebox[2]{#2}%
  \ifx\svgwidth\undefined%
    \setlength{\unitlength}{50.60949817bp}%
    \ifx\svgscale\undefined%
      \relax%
    \else%
      \setlength{\unitlength}{\unitlength * \real{\svgscale}}%
    \fi%
  \else%
    \setlength{\unitlength}{\svgwidth}%
  \fi%
  \global\let\svgwidth\undefined%
  \global\let\svgscale\undefined%
  \makeatother%
  \begin{picture}(1,0.32413201)%
    \put(0,0){\includegraphics[width=\unitlength,page=1]{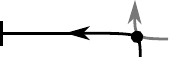}}%
    \put(0.19205462,0.16910767){\color[rgb]{0,0,0}\makebox(0,0)[lb]{\smash{$h$}}}%
    \put(0.83401553,0.1528676){\color[rgb]{0,0,0}\makebox(0,0)[lb]{\smash{$\color{gray}h'$}}}%
  \end{picture}%
\endgroup%
}}
\right)^\vee&=
\vcenter{\hbox{
\begingroup%
  \makeatletter%
  \providecommand\color[2][]{%
    \errmessage{(Inkscape) Color is used for the text in Inkscape, but the package 'color.sty' is not loaded}%
    \renewcommand\color[2][]{}%
  }%
  \providecommand\transparent[1]{%
    \errmessage{(Inkscape) Transparency is used (non-zero) for the text in Inkscape, but the package 'transparent.sty' is not loaded}%
    \renewcommand\transparent[1]{}%
  }%
  \providecommand\rotatebox[2]{#2}%
  \ifx\svgwidth\undefined%
    \setlength{\unitlength}{82.50793115bp}%
    \ifx\svgscale\undefined%
      \relax%
    \else%
      \setlength{\unitlength}{\unitlength * \real{\svgscale}}%
    \fi%
  \else%
    \setlength{\unitlength}{\svgwidth}%
  \fi%
  \global\let\svgwidth\undefined%
  \global\let\svgscale\undefined%
  \makeatother%
  \begin{picture}(1,0.30940069)%
    \put(0,0){\includegraphics[width=\unitlength,page=1]{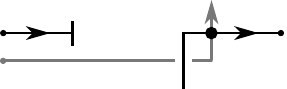}}%
    \put(0.12052384,0.26218732){\color[rgb]{0,0,0}\makebox(0,0)[lb]{\smash{$h^\vee$}}}%
    \put(0.75303037,0.06889979){\color[rgb]{0,0,0}\makebox(0,0)[lb]{\smash{$\color{gray}h'^\vee$}}}%
  \end{picture}%
\endgroup%
}}
\end{align*}
Here, the big dots mean the singular points as before.

\item [-] For reversed horizontal segments of types (III) and (IV), we define flips similarly as follows.
\begin{align*}
\tag{$\rm III^\vee_1$}\left(
\vcenter{\hbox{
\begingroup%
  \makeatletter%
  \providecommand\color[2][]{%
    \errmessage{(Inkscape) Color is used for the text in Inkscape, but the package 'color.sty' is not loaded}%
    \renewcommand\color[2][]{}%
  }%
  \providecommand\transparent[1]{%
    \errmessage{(Inkscape) Transparency is used (non-zero) for the text in Inkscape, but the package 'transparent.sty' is not loaded}%
    \renewcommand\transparent[1]{}%
  }%
  \providecommand\rotatebox[2]{#2}%
  \ifx\svgwidth\undefined%
    \setlength{\unitlength}{51.27183432bp}%
    \ifx\svgscale\undefined%
      \relax%
    \else%
      \setlength{\unitlength}{\unitlength * \real{\svgscale}}%
    \fi%
  \else%
    \setlength{\unitlength}{\svgwidth}%
  \fi%
  \global\let\svgwidth\undefined%
  \global\let\svgscale\undefined%
  \makeatother%
  \begin{picture}(1,0.34823359)%
    \put(0,0){\includegraphics[width=\unitlength,page=1]{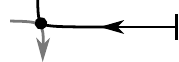}}%
    \put(0.58947379,0.25356386){\color[rgb]{0,0,0}\makebox(0,0)[lb]{\smash{$h$}}}%
    \put(-0.0064759,0.04422712){\color[rgb]{0,0,0}\makebox(0,0)[lb]{\smash{$\color{gray}h'$}}}%
  \end{picture}%
\endgroup%
}}
\right)^\vee&=
\vcenter{\hbox{
\begingroup%
  \makeatletter%
  \providecommand\color[2][]{%
    \errmessage{(Inkscape) Color is used for the text in Inkscape, but the package 'color.sty' is not loaded}%
    \renewcommand\color[2][]{}%
  }%
  \providecommand\transparent[1]{%
    \errmessage{(Inkscape) Transparency is used (non-zero) for the text in Inkscape, but the package 'transparent.sty' is not loaded}%
    \renewcommand\transparent[1]{}%
  }%
  \providecommand\rotatebox[2]{#2}%
  \ifx\svgwidth\undefined%
    \setlength{\unitlength}{81.768bp}%
    \ifx\svgscale\undefined%
      \relax%
    \else%
      \setlength{\unitlength}{\unitlength * \real{\svgscale}}%
    \fi%
  \else%
    \setlength{\unitlength}{\svgwidth}%
  \fi%
  \global\let\svgwidth\undefined%
  \global\let\svgscale\undefined%
  \makeatother%
  \begin{picture}(1,0.33585716)%
    \put(0,0){\includegraphics[width=\unitlength,page=1]{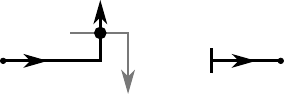}}%
    \put(0.06823403,0.00718974){\color[rgb]{0,0,0}\makebox(0,0)[lb]{\smash{$h^\vee$}}}%
    \put(0.12502738,0.22446439){\color[rgb]{0,0,0}\makebox(0,0)[lb]{\smash{$\color{gray}h'$}}}%
  \end{picture}%
\endgroup%
}}\\
\left(
\tag{$\rm III^\vee_2$}\vcenter{\hbox{
\begingroup%
  \makeatletter%
  \providecommand\color[2][]{%
    \errmessage{(Inkscape) Color is used for the text in Inkscape, but the package 'color.sty' is not loaded}%
    \renewcommand\color[2][]{}%
  }%
  \providecommand\transparent[1]{%
    \errmessage{(Inkscape) Transparency is used (non-zero) for the text in Inkscape, but the package 'transparent.sty' is not loaded}%
    \renewcommand\transparent[1]{}%
  }%
  \providecommand\rotatebox[2]{#2}%
  \ifx\svgwidth\undefined%
    \setlength{\unitlength}{51.27183432bp}%
    \ifx\svgscale\undefined%
      \relax%
    \else%
      \setlength{\unitlength}{\unitlength * \real{\svgscale}}%
    \fi%
  \else%
    \setlength{\unitlength}{\svgwidth}%
  \fi%
  \global\let\svgwidth\undefined%
  \global\let\svgscale\undefined%
  \makeatother%
  \begin{picture}(1,0.28529795)%
    \put(0,0){\includegraphics[width=\unitlength,page=1]{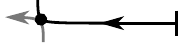}}%
    \put(0.58947379,0.21231627){\color[rgb]{0,0,0}\makebox(0,0)[lb]{\smash{$h$}}}%
    \put(-0.0064759,0.01146615){\color[rgb]{0,0,0}\makebox(0,0)[lb]{\smash{$\color{gray}h'$}}}%
  \end{picture}%
\endgroup%
}}
\right)^\vee&=
\vcenter{\hbox{
\begingroup%
  \makeatletter%
  \providecommand\color[2][]{%
    \errmessage{(Inkscape) Color is used for the text in Inkscape, but the package 'color.sty' is not loaded}%
    \renewcommand\color[2][]{}%
  }%
  \providecommand\transparent[1]{%
    \errmessage{(Inkscape) Transparency is used (non-zero) for the text in Inkscape, but the package 'transparent.sty' is not loaded}%
    \renewcommand\transparent[1]{}%
  }%
  \providecommand\rotatebox[2]{#2}%
  \ifx\svgwidth\undefined%
    \setlength{\unitlength}{81.8bp}%
    \ifx\svgscale\undefined%
      \relax%
    \else%
      \setlength{\unitlength}{\unitlength * \real{\svgscale}}%
    \fi%
  \else%
    \setlength{\unitlength}{\svgwidth}%
  \fi%
  \global\let\svgwidth\undefined%
  \global\let\svgscale\undefined%
  \makeatother%
  \begin{picture}(1,0.34560601)%
    \put(0,0){\includegraphics[width=\unitlength,page=1]{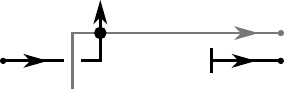}}%
    \put(0.05277102,0.00718693){\color[rgb]{0,0,0}\makebox(0,0)[lb]{\smash{$h^\vee$}}}%
    \put(0.42814771,0.25710276){\color[rgb]{0,0,0}\makebox(0,0)[lb]{\smash{$\color{gray}h'^\vee$}}}%
  \end{picture}%
\endgroup%
}}
\end{align*}

\begin{align*}
\tag{$\rm IV^\vee_1$}\left(\vcenter{\hbox{
\begingroup%
  \makeatletter%
  \providecommand\color[2][]{%
    \errmessage{(Inkscape) Color is used for the text in Inkscape, but the package 'color.sty' is not loaded}%
    \renewcommand\color[2][]{}%
  }%
  \providecommand\transparent[1]{%
    \errmessage{(Inkscape) Transparency is used (non-zero) for the text in Inkscape, but the package 'transparent.sty' is not loaded}%
    \renewcommand\transparent[1]{}%
  }%
  \providecommand\rotatebox[2]{#2}%
  \ifx\svgwidth\undefined%
    \setlength{\unitlength}{60.37341893bp}%
    \ifx\svgscale\undefined%
      \relax%
    \else%
      \setlength{\unitlength}{\unitlength * \real{\svgscale}}%
    \fi%
  \else%
    \setlength{\unitlength}{\svgwidth}%
  \fi%
  \global\let\svgwidth\undefined%
  \global\let\svgscale\undefined%
  \makeatother%
  \begin{picture}(1,0.2873341)%
    \put(0,0){\includegraphics[width=\unitlength,page=1]{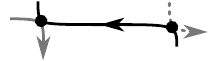}}%
    \put(0.42883969,0.20162402){\color[rgb]{0,0,0}\makebox(0,0)[lb]{\smash{$h$}}}%
    \put(-0.00549963,0.23699635){\color[rgb]{0,0,0}\makebox(0,0)[lb]{\smash{$\color{gray}h'$}}}%
    \put(0.84261363,0.20409185){\color[rgb]{0,0,0}\makebox(0,0)[lb]{\smash{$\color{gray}h''$}}}%
  \end{picture}%
\endgroup%
}}\right)^\vee=
\vcenter{\hbox{
\begingroup%
  \makeatletter%
  \providecommand\color[2][]{%
    \errmessage{(Inkscape) Color is used for the text in Inkscape, but the package 'color.sty' is not loaded}%
    \renewcommand\color[2][]{}%
  }%
  \providecommand\transparent[1]{%
    \errmessage{(Inkscape) Transparency is used (non-zero) for the text in Inkscape, but the package 'transparent.sty' is not loaded}%
    \renewcommand\transparent[1]{}%
  }%
  \providecommand\rotatebox[2]{#2}%
  \ifx\svgwidth\undefined%
    \setlength{\unitlength}{81.79999998bp}%
    \ifx\svgscale\undefined%
      \relax%
    \else%
      \setlength{\unitlength}{\unitlength * \real{\svgscale}}%
    \fi%
  \else%
    \setlength{\unitlength}{\svgwidth}%
  \fi%
  \global\let\svgwidth\undefined%
  \global\let\svgscale\undefined%
  \makeatother%
  \begin{picture}(1,0.42855812)%
    \put(0,0){\includegraphics[width=\unitlength]{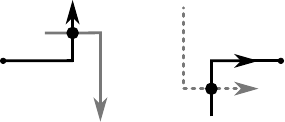}}%
    \put(0.03080135,0.0713535){\color[rgb]{0,0,0}\makebox(0,0)[lb]{\smash{$h^\vee$}}}%
    \put(0.66835096,0.32640882){\color[rgb]{0,0,0}\makebox(0,0)[lb]{\smash{$\color{gray}h''$}}}%
    \put(0.32457264,0.34180759){\color[rgb]{0,0,0}\makebox(0,0)[lb]{\smash{$\color{gray}h'$}}}%
  \end{picture}%
\endgroup%
}}\\
\tag{$\rm IV^\vee_2$}\left(\vcenter{\hbox{
\begingroup%
  \makeatletter%
  \providecommand\color[2][]{%
    \errmessage{(Inkscape) Color is used for the text in Inkscape, but the package 'color.sty' is not loaded}%
    \renewcommand\color[2][]{}%
  }%
  \providecommand\transparent[1]{%
    \errmessage{(Inkscape) Transparency is used (non-zero) for the text in Inkscape, but the package 'transparent.sty' is not loaded}%
    \renewcommand\transparent[1]{}%
  }%
  \providecommand\rotatebox[2]{#2}%
  \ifx\svgwidth\undefined%
    \setlength{\unitlength}{60.37341893bp}%
    \ifx\svgscale\undefined%
      \relax%
    \else%
      \setlength{\unitlength}{\unitlength * \real{\svgscale}}%
    \fi%
  \else%
    \setlength{\unitlength}{\svgwidth}%
  \fi%
  \global\let\svgwidth\undefined%
  \global\let\svgscale\undefined%
  \makeatother%
  \begin{picture}(1,0.26512219)%
    \put(0,0){\includegraphics[width=\unitlength,page=1]{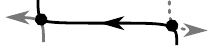}}%
    \put(0.42883969,0.18652999){\color[rgb]{0,0,0}\makebox(0,0)[lb]{\smash{$h$}}}%
    \put(-0.00549963,0.01378142){\color[rgb]{0,0,0}\makebox(0,0)[lb]{\smash{$\color{gray}h'$}}}%
    \put(0.84261363,0.18899783){\color[rgb]{0,0,0}\makebox(0,0)[lb]{\smash{$\color{gray}h''$}}}%
  \end{picture}%
\endgroup%
}}\right)^\vee=
\vcenter{\hbox{
\begingroup%
  \makeatletter%
  \providecommand\color[2][]{%
    \errmessage{(Inkscape) Color is used for the text in Inkscape, but the package 'color.sty' is not loaded}%
    \renewcommand\color[2][]{}%
  }%
  \providecommand\transparent[1]{%
    \errmessage{(Inkscape) Transparency is used (non-zero) for the text in Inkscape, but the package 'transparent.sty' is not loaded}%
    \renewcommand\transparent[1]{}%
  }%
  \providecommand\rotatebox[2]{#2}%
  \ifx\svgwidth\undefined%
    \setlength{\unitlength}{81.8bp}%
    \ifx\svgscale\undefined%
      \relax%
    \else%
      \setlength{\unitlength}{\unitlength * \real{\svgscale}}%
    \fi%
  \else%
    \setlength{\unitlength}{\svgwidth}%
  \fi%
  \global\let\svgwidth\undefined%
  \global\let\svgscale\undefined%
  \makeatother%
  \begin{picture}(1,0.40987818)%
    \put(0,0){\includegraphics[width=\unitlength]{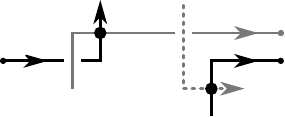}}%
    \put(0.05277102,0.07145876){\color[rgb]{0,0,0}\makebox(0,0)[lb]{\smash{$h^\vee$}}}%
    \put(0.42814771,0.32137459){\color[rgb]{0,0,0}\makebox(0,0)[lb]{\smash{$\color{gray}h'^\vee$}}}%
    \put(0,0){\includegraphics[width=\unitlength]{horizontal_segment_4_2_flip.pdf}}%
    \put(0.88122249,0.07823978){\color[rgb]{0,0,0}\makebox(0,0)[lb]{\smash{$\color{gray}h''$}}}%
  \end{picture}%
\endgroup%
}}
\end{align*}
\begin{align*}
\tag{$\rm IV^\vee_3$}\left(\vcenter{\hbox{
\begingroup%
  \makeatletter%
  \providecommand\color[2][]{%
    \errmessage{(Inkscape) Color is used for the text in Inkscape, but the package 'color.sty' is not loaded}%
    \renewcommand\color[2][]{}%
  }%
  \providecommand\transparent[1]{%
    \errmessage{(Inkscape) Transparency is used (non-zero) for the text in Inkscape, but the package 'transparent.sty' is not loaded}%
    \renewcommand\transparent[1]{}%
  }%
  \providecommand\rotatebox[2]{#2}%
  \ifx\svgwidth\undefined%
    \setlength{\unitlength}{60.37341893bp}%
    \ifx\svgscale\undefined%
      \relax%
    \else%
      \setlength{\unitlength}{\unitlength * \real{\svgscale}}%
    \fi%
  \else%
    \setlength{\unitlength}{\svgwidth}%
  \fi%
  \global\let\svgwidth\undefined%
  \global\let\svgscale\undefined%
  \makeatother%
  \begin{picture}(1,0.35943536)%
    \put(0,0){\includegraphics[width=\unitlength,page=1]{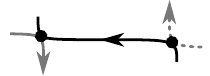}}%
    \put(0.42883969,0.20162402){\color[rgb]{0,0,0}\makebox(0,0)[lb]{\smash{$h$}}}%
    \put(-0.00549963,0.23699635){\color[rgb]{0,0,0}\makebox(0,0)[lb]{\smash{$\color{gray}h'$}}}%
    \put(0.84261363,0.20409185){\color[rgb]{0,0,0}\makebox(0,0)[lb]{\smash{$\color{gray}h''$}}}%
  \end{picture}%
\endgroup%
}}\right)^\vee=
\vcenter{\hbox{
\begingroup%
  \makeatletter%
  \providecommand\color[2][]{%
    \errmessage{(Inkscape) Color is used for the text in Inkscape, but the package 'color.sty' is not loaded}%
    \renewcommand\color[2][]{}%
  }%
  \providecommand\transparent[1]{%
    \errmessage{(Inkscape) Transparency is used (non-zero) for the text in Inkscape, but the package 'transparent.sty' is not loaded}%
    \renewcommand\transparent[1]{}%
  }%
  \providecommand\rotatebox[2]{#2}%
  \ifx\svgwidth\undefined%
    \setlength{\unitlength}{85.03048743bp}%
    \ifx\svgscale\undefined%
      \relax%
    \else%
      \setlength{\unitlength}{\unitlength * \real{\svgscale}}%
    \fi%
  \else%
    \setlength{\unitlength}{\svgwidth}%
  \fi%
  \global\let\svgwidth\undefined%
  \global\let\svgscale\undefined%
  \makeatother%
  \begin{picture}(1,0.42880959)%
    \put(0,0){\includegraphics[width=\unitlength]{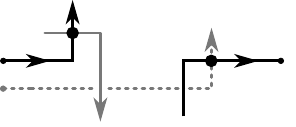}}%
    \put(0.77331201,0.2918852){\color[rgb]{0,0,0}\makebox(0,0)[lb]{\smash{$h^\vee$}}}%
    \put(0.00916696,0.00978509){\color[rgb]{0,0,0}\makebox(0,0)[lb]{\smash{$\color{gray}{h''}^\vee$}}}%
    \put(-0.00138755,0.30674545){\color[rgb]{0,0,0}\makebox(0,0)[lb]{\smash{$\color{gray}h'$}}}%
  \end{picture}%
\endgroup%
}}\\
\tag{$\rm IV^\vee_4$}\left(\vcenter{\hbox{
\begingroup%
  \makeatletter%
  \providecommand\color[2][]{%
    \errmessage{(Inkscape) Color is used for the text in Inkscape, but the package 'color.sty' is not loaded}%
    \renewcommand\color[2][]{}%
  }%
  \providecommand\transparent[1]{%
    \errmessage{(Inkscape) Transparency is used (non-zero) for the text in Inkscape, but the package 'transparent.sty' is not loaded}%
    \renewcommand\transparent[1]{}%
  }%
  \providecommand\rotatebox[2]{#2}%
  \ifx\svgwidth\undefined%
    \setlength{\unitlength}{61.5482079bp}%
    \ifx\svgscale\undefined%
      \relax%
    \else%
      \setlength{\unitlength}{\unitlength * \real{\svgscale}}%
    \fi%
  \else%
    \setlength{\unitlength}{\svgwidth}%
  \fi%
  \global\let\svgwidth\undefined%
  \global\let\svgscale\undefined%
  \makeatother%
  \begin{picture}(1,0.33938262)%
    \put(0,0){\includegraphics[width=\unitlength,page=1]{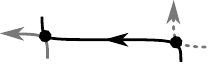}}%
    \put(0.4397416,0.18458347){\color[rgb]{0,0,0}\makebox(0,0)[lb]{\smash{$h$}}}%
    \put(0.01369264,0.01351837){\color[rgb]{0,0,0}\makebox(0,0)[lb]{\smash{$\color{gray}h'$}}}%
    \put(0.84561771,0.1870042){\color[rgb]{0,0,0}\makebox(0,0)[lb]{\smash{$\color{gray}h''$}}}%
  \end{picture}%
\endgroup%
}}\right)^\vee=
\vcenter{\hbox{
\begingroup%
  \makeatletter%
  \providecommand\color[2][]{%
    \errmessage{(Inkscape) Color is used for the text in Inkscape, but the package 'color.sty' is not loaded}%
    \renewcommand\color[2][]{}%
  }%
  \providecommand\transparent[1]{%
    \errmessage{(Inkscape) Transparency is used (non-zero) for the text in Inkscape, but the package 'transparent.sty' is not loaded}%
    \renewcommand\transparent[1]{}%
  }%
  \providecommand\rotatebox[2]{#2}%
  \ifx\svgwidth\undefined%
    \setlength{\unitlength}{85.03048975bp}%
    \ifx\svgscale\undefined%
      \relax%
    \else%
      \setlength{\unitlength}{\unitlength * \real{\svgscale}}%
    \fi%
  \else%
    \setlength{\unitlength}{\svgwidth}%
  \fi%
  \global\let\svgwidth\undefined%
  \global\let\svgscale\undefined%
  \makeatother%
  \begin{picture}(1,0.4288099)%
    \put(0,0){\includegraphics[width=\unitlength]{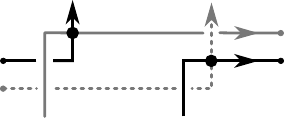}}%
    \put(0.77331201,0.08458133){\color[rgb]{0,0,0}\makebox(0,0)[lb]{\smash{$h^\vee$}}}%
    \put(0.19508718,0.00978509){\color[rgb]{0,0,0}\makebox(0,0)[lb]{\smash{$\color{gray}{h''}^\vee$}}}%
    \put(0.357979,0.33941582){\color[rgb]{0,0,0}\makebox(0,0)[lb]{\smash{$\color{gray}h'$}}}%
  \end{picture}%
\endgroup%
}}
\end{align*}
\end{itemize}

Especially, $(\cdot)^\vee$ maps the singular tile $t_\bullet$ locally as depicted in Figure~\ref{fig:flipsingulartile} according to its orientation.

\begin{figure}[ht]
\[
\xymatrix@C=0.5pc@R=1pc{
\SG_\bullet\ar[d]_{(\cdot)^\vee} & &
t_\bullet^N=\vcenter{\hbox{\includegraphics[scale=0.7]{grid_tile_Leg_singularcrossing_NE.pdf}}}\ar@{|->}[d] &
t_\bullet^E=\vcenter{\hbox{\includegraphics[scale=0.7]{grid_tile_Leg_singularcrossing_SE.pdf}}}\ar@{|->}[d] &
t_\bullet^W=\vcenter{\hbox{\includegraphics[scale=0.7]{grid_tile_Leg_singularcrossing_NW.pdf}}}\ar@{|->}[d] &
t_\bullet^S=\vcenter{\hbox{\includegraphics[scale=0.7]{grid_tile_Leg_singularcrossing_SW.pdf}}}\ar@{|->}[d] \\
\bar{\mathbf{SB}}& &
\vcenter{\hbox{\includegraphics{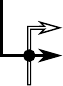}}}&
\vcenter{\hbox{\includegraphics{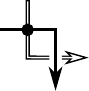}}}&
\vcenter{\hbox{\includegraphics{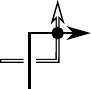}}}&
\vcenter{\hbox{\includegraphics{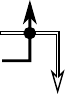}}}
}
\]
\caption{Images of the singular tile $t_\bullet$ under the flip $(\cdot)^\vee$}
\label{fig:flipsingulartile}
\end{figure}

We define $\|\cdot\|_{\SB}$ as the composition 
\[
\|\cdot\|_{\SB}:\SG_\bullet\stackrel{(\cdot)^\vee}{\longrightarrow}\bar{\mathbf{SB}}\stackrel{sl}{\longrightarrow}\SB
\]
of the flip $(\cdot)^\vee$ and the slanting map $sl$.

\begin{lem}\label{lem:map_sb}
The map $\|\cdot\|_{\SB}:\SG_\bullet\to\SB$ extends $\|\cdot\|_\B:\G\to\B$ and commutes with resolutions.
\end{lem}
\begin{proof}
Since all reversed horizontal segments in $\G$ are of type (I), $(\cdot)^\vee$ for any segment of type (I) on $\SG_\bullet$ is the same as the ordinary flip map $(\cdot)^\vee$ on $\G$ defined in \cite{MM, NT}.
Therefore $(\cdot)^\vee$ extends the flip map for $\G$, and so $\|\cdot\|_{\SB}$ extends $\|\cdot\|_\B$ as well.

We use the diagrams depicted in Figure~\ref{fig:singularresolutions} and Figure~\ref{fig:flipsingulartile} to prove the commutativity with resolutions.
The proof is straightforward and so we show here the proof only for $\|\mathcal{R}_+(t_\bullet^{NE})\|_\SB=\mathcal{R}_+(\|t_\bullet^{NE}\|_\SB)$ as follows.
\[
\xymatrix@C=3pc@R=1pc{
& \vcenter{\hbox{\includegraphics[scale=0.7]{grid_tile_Leg_singularcrossing_NE.pdf}}}\ar[ld]_{\mathcal{R}_+}\ar[dr]^{\|\cdot\|_{\SB}} \\
\vcenter{\hbox{\includegraphics[scale=0.7]{grid_R+_NE.pdf}}}\ar[r]^{\|\cdot\|_{\SB}} \ar[r] & \vcenter{\hbox{\includegraphics{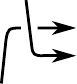}}} = \vcenter{\hbox{\includegraphics{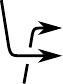}}}
&
\vcenter{\hbox{\includegraphics{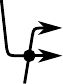}}}\ar[l]_-{\mathcal{R}_+}
}
\]
\end{proof}

\subsubsection{$\|\cdot\|_{\SL}:\SG_\bullet\to\SL$}
For $G\in\SG_\bullet$, a singular Legendrian links $L=\|G\|_{\SL}$ is given by the front projection $\pi_F(L)$ which is obtained by exactly the same way of $\|\cdot\|_\L$. More precisely, 
the pictorial definition of $\|\cdot\|_\SL$ is given by Figure~\ref{fig:G2L} together with 
\[
\left\|\vcenter{\hbox{\includegraphics[scale=0.7]{grid_tile_Leg_singularcrossing.pdf}}}\right\|_\SL:=
\vcenter{\hbox{\includegraphics{L_singularcrossing.pdf}}}.
\]

Then we have the lemma below, where the proof is straightforward and so we omit the proof.
\begin{lem}\label{lem:map_sl}
The map $\|\cdot\|_{\SL}:\SG_\bullet\to\SL$ extends $\|\cdot\|_\L:\G\to\L$ and commutes with all resolutions.
\end{lem}

\subsubsection{$\|\cdot\|_\ST:\SG_\bullet\to\ST$}
We define $\|\cdot\|_{\ST}$ as the composition of $\|\cdot\|_\SB$ and the transverse closure $\hat{(\cdot)}_\T$
\[
\|\cdot\|_\ST = \hat{(\|\cdot\|_{\SB})}_\T:\SG\to\SB\to\ST.
\]
Then we have the lemma as follows.
\begin{lem}\label{lem:map_st}
The map $\|\cdot\|_{\ST}$ extends $\|\cdot\|_\T:\G\to\T$ and commutes with resolutions.

Moreover, for any $G\in\SG_\bullet$, 
\[
\|G\|_\ST=\left(\|G\|_\SL\right)^+.
\]
\end{lem}
\begin{proof}
The first statement is obvious from the definition since both $\|\cdot\|_\SB$ and $\hat{(\cdot)}_\T$ on $\SB$ extends $\|\cdot\|_\B$ and $\hat{(\cdot)}_\T$ on $\B$, respectively, and they commute with resolutions.

Recall the definition of the closure $\hat{(\cdot)}_\T:\SB\to\ST$. Then the composition $\|\cdot\|_{\ST} = \hat{(\|\cdot\|_{\SB})}_\T$ is schematically as follows. 
\[
\xymatrix@R=0pc{
\SG_\bullet\ar[r]^{\|\cdot\|_{\SB}} & \SB \ar[r]^{\hat{(\cdot)}_\T} & \ST\\
\vcenter{\hbox{\includegraphics{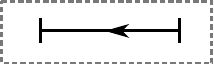}}}\ar@{|->}[r] &
\vcenter{\hbox{\includegraphics{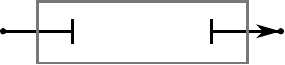}}}\ar@{|->}[r]&
\vcenter{\hbox{\includegraphics{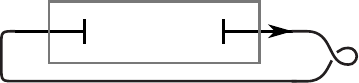}}}\ar@{=}[dd]\\ \\
&&\vcenter{\hbox{\includegraphics{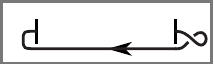}}}\quad\ 
}
\]
The last picture is obtained by moving the reversed horizontal segment under all other segments to the near place where the original segment was. Then it is very similar to the original grid diagram except for corners where the reversed horizontal segment ends.
Indeed, each corner is mapped to a part of a transverse link as follows.
\begin{align*}
\left\|\vcenter{\hbox{\includegraphics[scale=0.7]{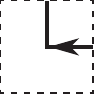}}}\right\|_{\ST}=\vcenter{\hbox{\includegraphics{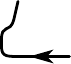}}}=\left(\left\| \vcenter{\hbox{\includegraphics[scale=0.7]{grid_tile_LU.pdf}}} \right\|_{\SL}\right)^+\\
\left\|\vcenter{\hbox{\includegraphics[scale=0.7]{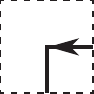}}}\right\|_{\ST}=\vcenter{\hbox{\includegraphics{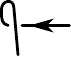}}}=\left(\left\| \vcenter{\hbox{\includegraphics[scale=0.7]{grid_tile_LD.pdf}}} \right\|_{\SL}\right)^+\\
\left\|\vcenter{\hbox{\includegraphics[scale=0.7]{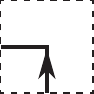}}}\right\|_{\ST}=\vcenter{\hbox{\includegraphics{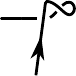}}}=\left(\left\| \vcenter{\hbox{\includegraphics[scale=0.7]{grid_tile_UL.pdf}}} \right\|_{\SL}\right)^+\\
\left\|\vcenter{\hbox{\includegraphics[scale=0.7]{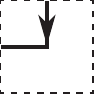}}}\right\|_{\ST}=\vcenter{\hbox{\includegraphics{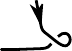}}}=\left(\left\| \vcenter{\hbox{\includegraphics[scale=0.7]{grid_tile_DL.pdf}}} \right\|_{\SL}\right)^+
\end{align*}
Notice that the vertical segments are slightly slanted by the map $sl:\bar{\mathbf{SB}}\to\SB$, and so there is no downward vertical tangencies, and the second equality on each row is obtained by the 45 degree rotation.

Also we can consider the exactly same diagram as above for the reversed horizontal segments of the other types, and then we have a front projection of a transverse link which is very similar to the original grid diagram except for singular tiles.
We use the diagrams in Figure~\ref{fig:flipsingulartile} to find out how singular tiles are mapped to pieces of transverse links.
\[
\begin{array}{c}
\|t_\bullet^E\|_{\ST}=
\vcenter{\hbox{\includegraphics{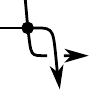}}}=
\left(\|t_\bullet^E\|_{\SL}\right)^+\\
\|t_\bullet^S\|_{\ST}= 
\vcenter{\hbox{\includegraphics{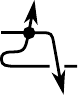}}}=
\left(\|t_\bullet^S\|_{\SL}\right)^+\\
\|t_\bullet^W\|_{\ST}= 
\vcenter{\hbox{\includegraphics{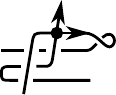}}}=\vcenter{\hbox{\includegraphics{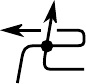}}}=
\left(\|t_\bullet^W\|_{\SL}\right)^+\\
\|t_\bullet^N\|_{\ST}=
\vcenter{\hbox{\includegraphics{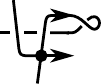}}} = \vcenter{\hbox{\includegraphics{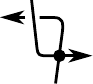}}} =
\left(\|t_\bullet^N\|_{\SL}\right)^+
\end{array}
\]

As seen above, $\|\cdot\|_{\ST}=(\|\cdot\|_{\SL})^+$ for every corner and singular tile and therefore the lemma is proved.
\end{proof}

\begin{proof}[Proof of Theorem~\ref{thm:extension}]
By Lemmas~\ref{lem:map_sb}, \ref{lem:map_sl}  and \ref{lem:map_st}, all maps $\|\cdot\|_\SC$ for $\mathcal{C}\in\{\B,\L,\T\}$ satisfy the conditions (1), (2) and (3) of Definition~\ref{defn:unified}.

Furthermore by definition of $\|\cdot\|_\SK:\SG_\bullet\to\SK$, it is obvious that for any $G\in\SG_\bullet$,
\[
\|G\|_\SK = \hat{\|G\|_\SB} = \left\| \|G\|_\SL \right\| = \left\| \|G\|_\ST\right\|,
\]
and therefore $\|\cdot\|_\SK$ commutes with resolutions.
Together with the second statement of Lemma~\ref{lem:map_st}, the diagram in Definition~\ref{defn:unified} is commutative.
Therefore the condition (4) is satisfied and this completes the proof of Theorem~\ref{thm:extension}.
\end{proof}

%
%
%
%
%
%

\subsection{Singular grid diagrams extended by \texorpdfstring{$t_\times$}{tx}}
Recall the set $\SG_\times$ of singular grid diagrams extended by $t_\times$.
\[
t_\times=\vcenter{\hbox{\includegraphics[scale=0.7]{grid_tile_singularcrossing.pdf}}}\in\SG_\times,\qquad
\|t_\times\|_\SK=\vcenter{\hbox{\includegraphics{K_singularcrossing.pdf}}}\in\SK.
\]

It looks more natural than $t_\bullet$, but surprisingly $\SG_\times$ never give a unified description whatever resolutions and maps onto $\SB,\SL,\ST$ and $\SK$ are defined.

\begin{thm}[Theorem~\ref{thm:notunified}]\label{thm:G_times}
The set of singular grid diagrams $\SG_\times$ does not give a unified description.
\end{thm}
\begin{proof}
Suppose that $\SG_\times$ gives a unified description.
Let 
\[
t_\times^E:=\vcenter{\hbox{\includegraphics[scale=0.7]{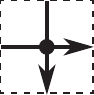}}}\in\SG_\times,\qquad
L_\times:=\left\|t_\times^E\right\|_{\SL}\in\SL.
\]
Then the Legendrian closure $\hat{L_\times}$ has exactly 2 singular points, and so we get a nonsingular Legendrian link $\mathcal{R}_-^2\left(\hat{L_\times}\right)\in\K$ whose knot type is precisely the left-handed trefoil as follows:
\begin{align*}
\left\|\mathcal{R}_-^2\left(\hat{L_\times}\right)\right\|
&=\mathcal{R}_-^2\left( \left\| \vcenter{\hbox{\footnotesize\def\svgscale{0.7}
\begingroup%
  \makeatletter%
  \providecommand\color[2][]{%
    \errmessage{(Inkscape) Color is used for the text in Inkscape, but the package 'color.sty' is not loaded}%
    \renewcommand\color[2][]{}%
  }%
  \providecommand\transparent[1]{%
    \errmessage{(Inkscape) Transparency is used (non-zero) for the text in Inkscape, but the package 'transparent.sty' is not loaded}%
    \renewcommand\transparent[1]{}%
  }%
  \providecommand\rotatebox[2]{#2}%
  \ifx\svgwidth\undefined%
    \setlength{\unitlength}{96.73697328bp}%
    \ifx\svgscale\undefined%
      \relax%
    \else%
      \setlength{\unitlength}{\unitlength * \real{\svgscale}}%
    \fi%
  \else%
    \setlength{\unitlength}{\svgwidth}%
  \fi%
  \global\let\svgwidth\undefined%
  \global\let\svgscale\undefined%
  \makeatother%
  \begin{picture}(1,0.53716093)%
    \put(0,0){\includegraphics[width=\unitlength,page=1]{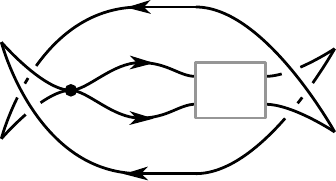}}%
    \put(0.60763264,0.22412966){\color[rgb]{0,0,0}\makebox(0,0)[lb]{\smash{$L_\times$}}}%
  \end{picture}%
\endgroup%
}}  
\right\| \right)
= \mathcal{R}_-^2\left(
\vcenter{\hbox{\footnotesize\def\svgscale{0.7}
\begingroup%
  \makeatletter%
  \providecommand\color[2][]{%
    \errmessage{(Inkscape) Color is used for the text in Inkscape, but the package 'color.sty' is not loaded}%
    \renewcommand\color[2][]{}%
  }%
  \providecommand\transparent[1]{%
    \errmessage{(Inkscape) Transparency is used (non-zero) for the text in Inkscape, but the package 'transparent.sty' is not loaded}%
    \renewcommand\transparent[1]{}%
  }%
  \providecommand\rotatebox[2]{#2}%
  \ifx\svgwidth\undefined%
    \setlength{\unitlength}{128.65548822bp}%
    \ifx\svgscale\undefined%
      \relax%
    \else%
      \setlength{\unitlength}{\unitlength * \real{\svgscale}}%
    \fi%
  \else%
    \setlength{\unitlength}{\svgwidth}%
  \fi%
  \global\let\svgwidth\undefined%
  \global\let\svgscale\undefined%
  \makeatother%
  \begin{picture}(1,0.47972197)%
    \put(0,0){\includegraphics[width=\unitlength,page=1]{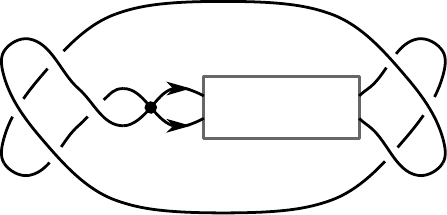}}%
    \put(0.52393454,0.2092987){\color[rgb]{0,0,0}\makebox(0,0)[lb]{\smash{$\|L_\times\|$}}}%
  \end{picture}%
\endgroup%
}}  
\right)\\
&= \mathcal{R}_-^2\left(\vcenter{\hbox{\includegraphics[scale=0.7]{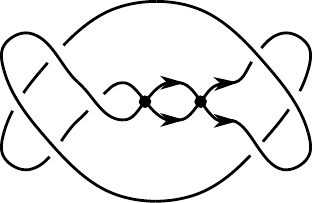}}}  \right) = \vcenter{\hbox{\includegraphics[scale=0.7]{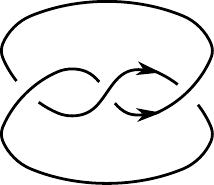}}}\in\K
\end{align*}
However, the maximal Thurston-Bennequin number is known to be $-6$ (see \cite{CN}) and therefore by Lemma~\ref{lem:resolutionandtb},
\begin{equation}\label{eq:tbLx}
tb(\hat{L_\times}) = tb\left(\mathcal{R}_-^2\left(\hat{L_\times}\right)\right) +2 \le -4.
\end{equation}

Recall the singular Legendrian 2-component link $L_{\bigcirc\!\!\!\bigcirc}$ in Example~\ref{ex:L00}, and let $G_{\bigcirc\!\!\!\bigcirc}$ be a singular grid diagram in $\SG_\times$ such that
\[
\|G_{\bigcirc\!\!\!\bigcirc}\|_{\SL} = L_{\bigcirc\!\!\!\bigcirc}.
\]

Since $L_{\bigcirc\!\!\!\bigcirc}$ contains two singular points, $G_{\bigcirc\!\!\!\bigcirc}$ contains exactly two singular tiles $t_\times^1$ and $t_\times^2$.
Then by changing orientations if necessary, we may assume that 
\[
t_\times^1=t_\times^E=\vcenter{\hbox{\includegraphics[scale=0.7]{grid_tile_singularcrossing_arrow.pdf}}}\subset
G_{\bigcirc\!\!\!\bigcirc}.
\]
Recall that the singular Legendrian link type of $L_{\bigcirc\!\!\!\bigcirc}$ is independent of the choice of the orientations as mentioned in Example~\ref{ex:L00}.
Then by taking $\|\cdot\|_\SL$, we have
\[
\|t_\times^1\|_\SL = L_\times \subset L_{\bigcirc\!\!\!\bigcirc},
\]
and so $\hat{L_\times}=L_{\bigcirc\!\!\!\bigcirc}$ by Corollary~\ref{cor:complement} and Theorem~\ref{thm:deg2unit}.
This is contradiction since
\[
-2=tb(L_{\bigcirc\!\!\!\bigcirc})=tb(\hat{L_\times})\le -4
\]
by equation (\ref{eq:tbLx}), and we conclude that $\G_\times$ does not give a unified description in the sense of Definition~\ref{defn:unified}.
\end{proof}

\section{Proof of main theorem}\label{sec:proof_mainthm}
From now on, we denote $\SG_\bullet$ by $\SG$. In this section, we define moves on $\SG$ and observe the invariances of maps defined on $\SG$ under the moves, and finally prove the main theorem of this paper.

\subsection{Moves on $\SG$}
We consider the three types of moves on $\SG$ as follows: (i) translations and commutations, (ii) (de)stabilizations, (iii) rotations, swirl and flype.

\subsubsection{Translations and commutations}
Recall that translations in $\G$ are cyclic permutations of horizontal and vertical edges as shown in Figure~\ref{fig:trans}.
However, they may not be well-defined in $\SG$.
For example, as shown in Figure~\ref{fig:translation}, the cyclic permutations on $G_{3_1^{\bullet}}$ which move the bottm row or the leftmost column to the top or the right are possible, but the cyclic permutations on the other ways are impossible. The solid and dashed lines in black indicate the segments where the cyclic permutations can and can not be performed, respectively.
\begin{figure}[ht]
\[
\xymatrix@C=5pc{
& \vcenter{\hbox{\includegraphics[scale=0.7]{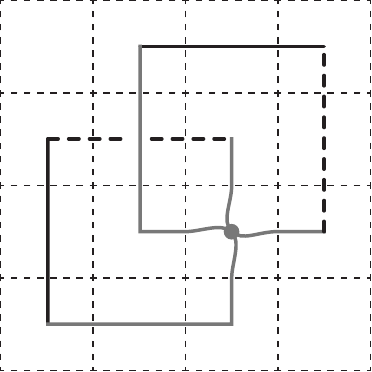}}} &\\
\text{Undefined} &
\vcenter{\hbox{\includegraphics[scale=0.7]{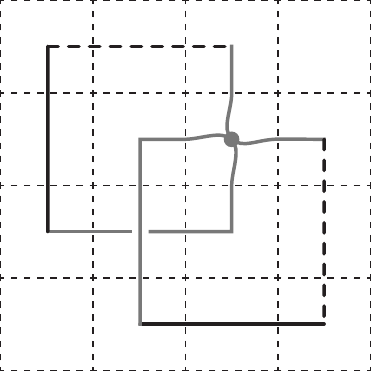}}}\ar[l]_-{Right-to-Left}\ar[r]^-{Left-to-Right}\ar[u]^-{Bottom-to-Top}\ar[d]_-{Top-to-Bottom}&\vcenter{\hbox{\includegraphics[scale=0.7]{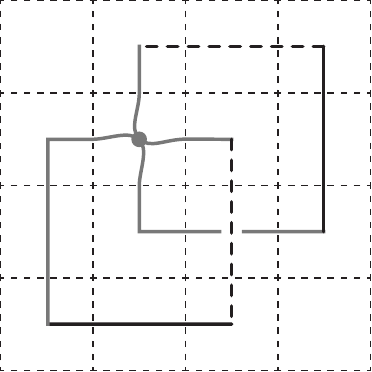}}} \\
&\text{Undefined} &
}
\]
\caption{Cyclic permutations on $G_{3_1^{\bullet}}\in\SG$}
\label{fig:translation}
\end{figure}

The important observation we can get from Figure~\ref{fig:translation} is as follows.
A horizontal cyclic permutation on the leftmost or rightmost vertical segment of given singular grid diagram is possible if and only if none of horizontal segments adjacent to the vertical segment ends at a singular point.
For a vertical cyclic permutation, the same criterion holds.

Now we consider a generalization of translations for singular grid diagrams as follows.
For given $G\in\SG$, we denote {\em horizontal} and {\em vertical decompositions} for $G$ into two nonempty pseudo singular diagrams by $(H_1^L|H_2^R)$ and $\displaystyle{\left(\frac{V_1^U}{V_2^D}\right)}$, respectively.
A horizontal or vertical decomposition, say $(H_1^L|H_2^R)$ or $\displaystyle{\left(\frac{V_1^U}{V_2^D}\right)}$, for $G$ is {\em admissible} if all horizontal or vertical segments joining $H_1^L$ and $H_2^R$ or $V_1^U$ and $V_2^D$ end only at {\em corners}, respectively. See Figure~\ref{fig:admissibledecompositions} for examples of admissible and non-admissible decompositions.

\begin{figure}[ht]
\subfigure[Admissible decompositions]{\makebox[.5\textwidth]{\includegraphics[scale=0.7]{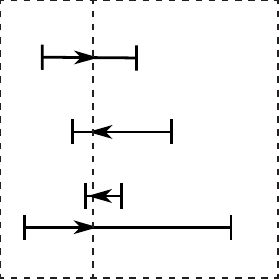} \qquad 
\includegraphics[scale=0.7]{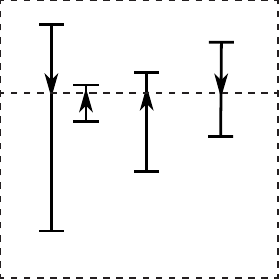}}}
\subfigure[Non-admissible decomposition]{\makebox[.4\textwidth]{\includegraphics[scale=0.7]{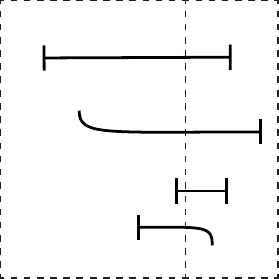}}}
\caption{Admissible and non-admissible decompositions}
\label{fig:admissibledecompositions}
\end{figure}

Suppose that $(H_1^L|H_2^R)$ is an admissible decomposition and let $\{h_1,\dots,h_m\}$ be the set of horizontal edges of $H_1^L$ whose right ends are on the boundary. Then we can flip all $h_i$'s to the left to obtain a pseudo singular grid diagram $H_1^R$. We denote this process by $f_{L\to R}$.

Moreover, one can define the similar processes $f_{R\to L}$, $f_{U\to D}$ and $f_{D\to U}$ by flipping all segments having one end at boundary to the opposite direction. Then they satisfy that
\[
\xymatrix@C=2.5pc@R=0pc{
H_1^L, H_2^L\ar@<.5ex>[r]^{f_{L\to R}} & H_1^R, H_2^R\ar@<.5ex>[l]^{f_{R\to L}} &
V_1^U, V_2^U\ar@<.5ex>[r]^{f_{U\to D}} & V_1^D, V_2^D\ar@<.5ex>[l]^{f_{D\to U}}\\
\vcenter{\hbox{\includegraphics[scale=0.7]{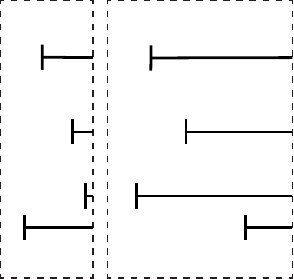}}}\ar@<.5ex>[r]^{f_{L\to R}} & 
\vcenter{\hbox{\includegraphics[scale=0.7]{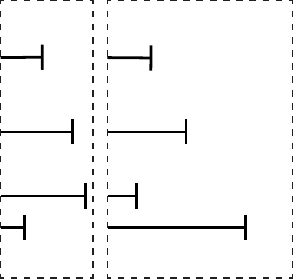}}}\ar@<.5ex>[l]^{f_{R\to L}} &
\vcenter{\hbox{\includegraphics[scale=0.7]{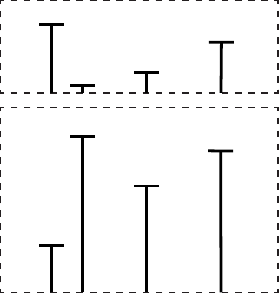}}}\ar@<.5ex>[r]^{f_{U\to D}} & 
\vcenter{\hbox{\includegraphics[scale=0.7]{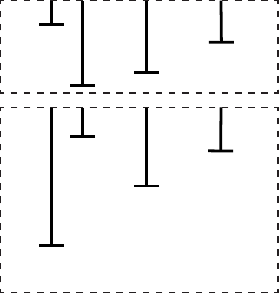}}}\ar@<.5ex>[l]^{f_{D\to U}}
}
\]
and one can glue $H_2^L$ and $H_1^R$ or $V_2^U$ and $V_1^D$ to obtain a new singular grid diagram $(H_2^L|H_1^R)$ or $\displaystyle{\left(\frac{V_2^U}{V_1^D}\right)}$, respectively.

\begin{defn}[Translations]
Let $(H_1^L|H_2^R)$ and $\displaystyle{\left(\frac{V_1^U}{V_2^D}\right)}$ be admissible horizontal and vertical decompositions for $G\in\SG$, respectively. Then {\em horizontal translation $(Tr_H)$} and {\em vertical translation $(Tr_V)$} on these decompositions are defined as follows.
\[
\xymatrix@C=2.5pc@R=0pc{
(H_1^L|H_2^R)\ar@{<->}[r]^-{(Tr_H)}&(H_2^L|H_1^R)&
\displaystyle{\left(\frac{V_1^U}{V_2^D}\right)}\ar@{<->}[r]^-{(Tr_V)}&\displaystyle{\left(\frac{V_2^U}{V_1^D}\right)}\\
\vcenter{\hbox{\includegraphics[scale=0.7]{Translation_H_1.pdf}}}\ar@{<->}[r]^-{(Tr_H)}&
\vcenter{\hbox{\includegraphics[scale=0.7]{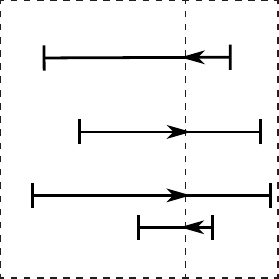}}}&
\vcenter{\hbox{\includegraphics[scale=0.7]{Translation_V_1.pdf}}}\ar@{<->}[r]^-{(Tr_V)}&
\vcenter{\hbox{\includegraphics[scale=0.7]{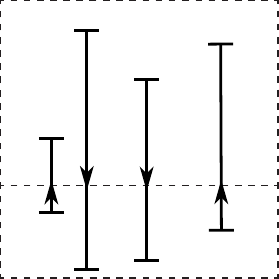}}}
}
\]
\end{defn}

On the other hand, for each column $c$ in $G$, corners and singular points form a disjoint union $I(c)$ of open segments and can be identified with intervals in the vertical axes $\mathbb{R}_y$ so that $I(c)\subset \mathbb{R}_y$.
Similarly, for a row $r$ in $G$, we may define $I(r)\subset\mathbb{R}_x$ consisting of horizontal segments without end points and singular points.
For the diagrammatic illustration of the open intervals, see Figures~\ref{fig:noninterleaving} and \ref{fig:interleaving}.

We say that two contiguous columns $c$ and $c'$ in $G$ are {\em non-interleaving} if $I(c)$ and $I(c')$ are non-interleaving as collections of open intervals in $\mathbb{R}_y$, otherwise $c$ and $c'$ are {\em interleaving}.
The non-interleavingness for two contiguous rows $r$ and $r'$ in $G$ is defined by the similar manner. 
Figure~\ref{fig:noninterleaving} and Figure~\ref{fig:interleaving} show non-interleaving and interleaving examples, respectively.

\begin{figure}[ht]
\[
\vcenter{\hbox{\includegraphics[scale=0.7]{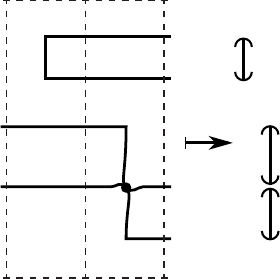}}}\qquad\qquad
\vcenter{\hbox{\includegraphics[scale=0.7]{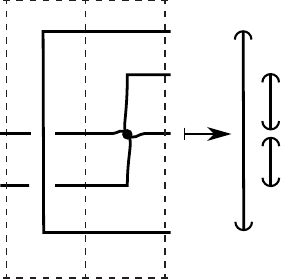}}}\qquad\qquad
\vcenter{\hbox{\includegraphics[scale=0.7]{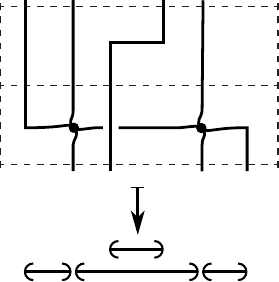}}}
\]
\caption{Non-interleaving contiguous columns and rows}
\label{fig:noninterleaving}
\end{figure}

\begin{figure}[ht]
\[
\vcenter{\hbox{\includegraphics[scale=0.7]{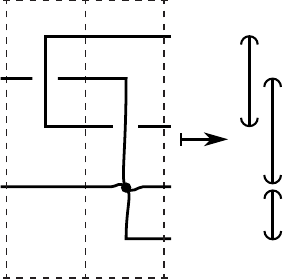}}}\qquad\qquad
\vcenter{\hbox{\includegraphics[scale=0.7]{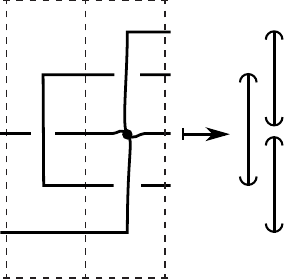}}}\qquad\qquad
\vcenter{\hbox{\includegraphics[scale=0.7]{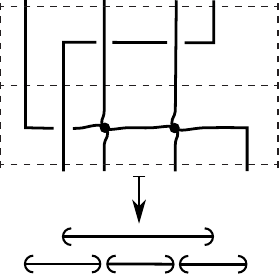}}}
\]
\caption{Interleaving contiguous columns and rows}
\label{fig:interleaving}
\end{figure}

\begin{defn}[Commutations]
Let $G=(\cdots\ c\ c'\cdots)=(\cdots\ r\ r'\cdots)^t\in\SG$ be given by a square matrix of oriented corners and singular tiles as mentioned in Remark~\ref{rmk:matrix_SG}. 
Suppose that $c$ and $c'$ or $r$ and $r'$ are non-interleaving columns or rows, respectively.
Then {\em horizontal commutation $(Cm_H)$} and {\em vertical commutation $(Cm_V)$} on these columns and rows are defined as follows. 
\[
\xymatrix@C=3pc@R=0pc{
(\cdots\ c\ c'\ \cdots)\ar@{<->}[r]^-{(Cm_H)}&(\cdots\ c'\ c\ \cdots)&
\mbox{$
\left(\begin{matrix}
\vdots\\
r\\
r'\\
\vdots
\end{matrix}\right)$
}
\ar@{<->}[r]^-{(Cm_V)}&
\mbox{$
\left(\begin{matrix}
\vdots\\
r'\\
r\\
\vdots
\end{matrix}\right)$
}\\
\vcenter{\hbox{\includegraphics[scale=0.8]{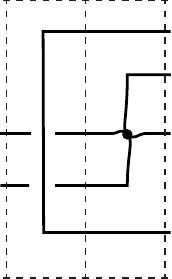}}}\ar@{<->}[r]^-{(Cm_H)}&
\vcenter{\hbox{\includegraphics[scale=0.8]{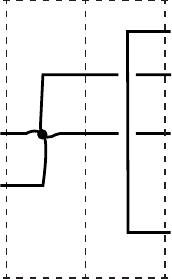}}}&
\vcenter{\hbox{\includegraphics[scale=0.8]{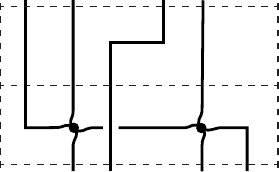}}}\ar@{<->}[r]^-{(Cm_V)}&
\vcenter{\hbox{\includegraphics[scale=0.8]{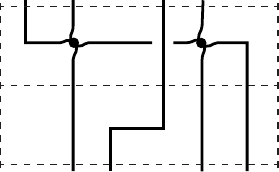}}}
}
\]
\end{defn}

\subsubsection{(De)Stabilizations}
There are several definitions of stabilizations in $\G$ which are equivalent (possibly up to translations and commutations), but the essential part is to make two more corners which are connected by a short horizontal or vertical segment. Then by applying translations or commutations, one can move the stabilized segment freely on the component where it belongs to. By this reason, one can reduce (de)stabilizations as only four cases depicted in Figure~\ref{fig:stabil}.

However, this is not always possible in $\SG$ since a stabilized segment can not move further when it meets a singular point.
Therefore we need more moves to be able to handle all possible cases as follows.

\begin{defn}[(De)Stabilizations] 
There are four types of (de)stabilizations as depicted in Figure~\ref{fig:stabilizations}.
\end{defn}

\begin{figure}[ht]

\begin{align*}
\xymatrix@C=2.5pc@R=0pc{
\vcenter{\hbox{\includegraphics[scale=0.7]{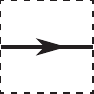}}}\quad\ar@<.5ex>[r]^-{(SE)}&
\quad\vcenter{\hbox{\includegraphics[scale=0.7]{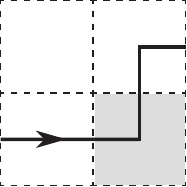}}}\ar@<.5ex>[l]^-{(SE)^{-1}} \qquad \qquad
\vcenter{\hbox{\includegraphics[scale=0.7]{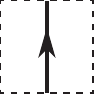}}}\quad\ar@<.5ex>[r]^-{(SE)}&
\quad\vcenter{\hbox{\includegraphics[scale=0.7]{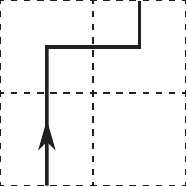}}}\ar@<.5ex>[l]^-{(SE)^{-1}} \\
\vcenter{\hbox{\includegraphics[scale=0.7]{grid_tile_DL.pdf}}}\quad\ar@<.5ex>[r]^-{(SE)}&
\quad\vcenter{\hbox{\includegraphics[scale=0.7]{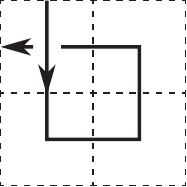}}}\ar@<.5ex>[l]^-{(SE)^{-1}} \qquad \qquad
\vcenter{\hbox{\includegraphics[scale=0.7]{grid_tile_LD.pdf}}}\quad\ar@<.5ex>[r]^-{(SE)}&
\quad\vcenter{\hbox{\includegraphics[scale=0.7]{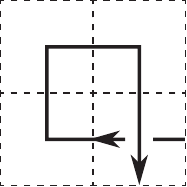}}}\ar@<.5ex>[l]^-{(SE)^{-1}} 
}
\end{align*}
\begin{align*}
\xymatrix@C=2.5pc@R=0pc{
\vcenter{\hbox{\includegraphics[scale=0.7]{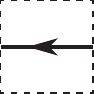}}}\quad\ar@<.5ex>[r]^-{(SW)}&
\quad\vcenter{\hbox{\includegraphics[scale=0.7]{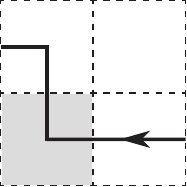}}}\ar@<.5ex>[l]^-{(SW)^{-1}} \qquad \qquad
\vcenter{\hbox{\includegraphics[scale=0.7]{grid_tile_U.pdf}}}\quad\ar@<.5ex>[r]^-{(SW)}&
\quad\vcenter{\hbox{\includegraphics[scale=0.7]{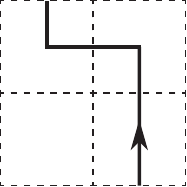}}}\ar@<.5ex>[l]^-{(SW)^{-1}} \\
\vcenter{\hbox{\includegraphics[scale=0.7]{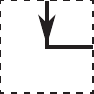}}}\quad\ar@<.5ex>[r]^-{(SW)}&
\quad\vcenter{\hbox{\includegraphics[scale=0.7]{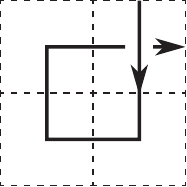}}}\ar@<.5ex>[l]^-{(SW)^{-1}} \qquad \qquad
\vcenter{\hbox{\includegraphics[scale=0.7]{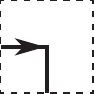}}}\quad\ar@<.5ex>[r]^-{(SW)}&
\quad\vcenter{\hbox{\includegraphics[scale=0.7]{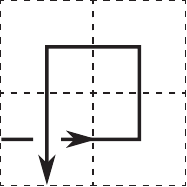}}}\ar@<.5ex>[l]^-{(SW)^{-1}} 
}
\end{align*}
\begin{align*}
\xymatrix@C=2.5pc@R=0pc{
\vcenter{\hbox{\includegraphics[scale=0.7]{grid_tile_L.pdf}}}\quad\ar@<.5ex>[r]^-{(NW)}&
\quad\vcenter{\hbox{\includegraphics[scale=0.7]{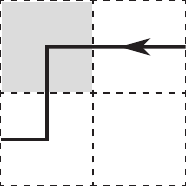}}}\ar@<.5ex>[l]^-{(NW)^{-1}} \qquad \qquad
\vcenter{\hbox{\includegraphics[scale=0.7]{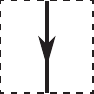}}}\quad\ar@<.5ex>[r]^-{(NW)}&
\quad\vcenter{\hbox{\includegraphics[scale=0.7]{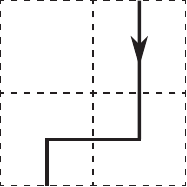}}}\ar@<.5ex>[l]^-{(NW)^{-1}} \\
\vcenter{\hbox{\includegraphics[scale=0.7]{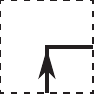}}}\quad\ar@<.5ex>[r]^-{(NW)}&
\quad\vcenter{\hbox{\includegraphics[scale=0.7]{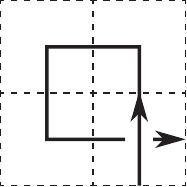}}}\ar@<.5ex>[l]^-{(NW)^{-1}} \qquad \qquad
\vcenter{\hbox{\includegraphics[scale=0.7]{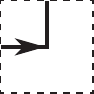}}}\quad\ar@<.5ex>[r]^-{(NW)}&
\quad\vcenter{\hbox{\includegraphics[scale=0.7]{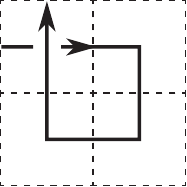}}}\ar@<.5ex>[l]^-{(NW)^{-1}} 
}
\end{align*}
\begin{align*}
\xymatrix@C=2.5pc@R=0pc{
\vcenter{\hbox{\includegraphics[scale=0.7]{grid_tile_R.pdf}}}\quad\ar@<.5ex>[r]^-{(NE)}&
\quad\vcenter{\hbox{\includegraphics[scale=0.7]{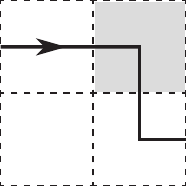}}}\ar@<.5ex>[l]^-{(NE)^{-1}} \qquad \qquad
\vcenter{\hbox{\includegraphics[scale=0.7]{grid_tile_D.pdf}}}\quad\ar@<.5ex>[r]^-{(NE)}&
\quad\vcenter{\hbox{\includegraphics[scale=0.7]{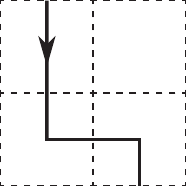}}}\ar@<.5ex>[l]^-{(NE)^{-1}} \\
\vcenter{\hbox{\includegraphics[scale=0.7]{grid_tile_UL.pdf}}}\quad\ar@<.5ex>[r]^-{(NE)}&
\quad\vcenter{\hbox{\includegraphics[scale=0.7]{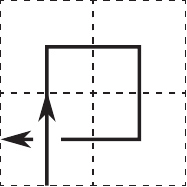}}}\ar@<.5ex>[l]^-{(NE)^{-1}} \qquad \qquad
\vcenter{\hbox{\includegraphics[scale=0.7]{grid_tile_LU.pdf}}}\quad\ar@<.5ex>[r]^-{(NE)}&
\quad\vcenter{\hbox{\includegraphics[scale=0.7]{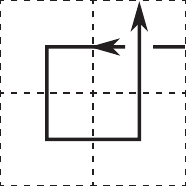}}}\ar@<.5ex>[l]^-{(NE)^{-1}} 
}
\end{align*}
\caption{The four types of (de)stabilizations on horizontal, vertical segments and corners}
\label{fig:stabilizations}
\end{figure}

Note that all stabilizations in Figure~\ref{fig:stabilizations} increase the size of the grid diagram by 1. 
As mentioned before, stabilizations on horizontal and vertical segments are redundant for nonsingular cases since every horizontal and vertical segments end at corners.
The (de)stabilization naming convention follows from the first stabilized corner for each horizontal segment, which is indicated as a shaded region and is the same as the convention that Ozsv\'ath, Szab\'o and Thurston used in \cite{OST}.

\begin{prop}\cite[Proposition~3]{NT}\label{prop:stabilizations}
We have the following correspondences.
\[
\begin{array}{cccl}
(SW)&\longleftrightarrow&(BrS_+)&\text{ Braid $(+)$-stabilization } \\
(NW)&\longleftrightarrow&(BrS_-)&\text{ Braid $(-)$-stabilization } \\
 (NW)&\longleftrightarrow&(LS_+)&\text{ Legendrian $(+)$-stabilization }\\
(SE)&\longleftrightarrow &(LS_-)&\text{ Legendrian $(-)$-stabilization }\\
(NW)&\longleftrightarrow&(TS) &\text{ Transverse stabilization }
\end{array}
\]
\end{prop}
\begin{proof}
The proof is straightforward from the definition.
\end{proof}

\subsubsection{Rotations, swirl and flype}
We introduce three more moves which change the direction of the singular tile $t_\bullet$ as follows.
\begin{defn}
The {\em positive and negative rotations} $(Rot_\pm)$, {\em swirl} $(Swirl)$ and {\em flype} $(Flype)$ are defined as depicted in Figure~\ref{fig:rot_sw_fl}.
\end{defn}

\begin{figure}[ht]
\begin{align*}
\xymatrix{
\vcenter{\hbox{\includegraphics[scale=0.6]{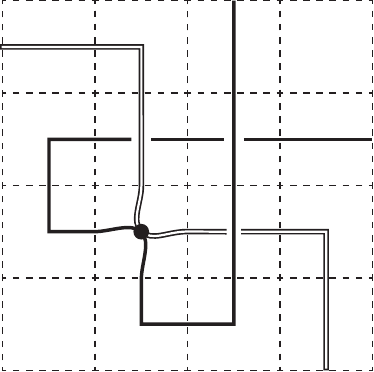}}}\quad&
\quad t_{\bullet}=\vcenter{\hbox{\includegraphics[scale=0.7]{grid_tile_Leg_singularcrossing.pdf}}}\quad\ar[r]^-{(Rot_-)}\ar[l]_-{(Rot_+)}&
\quad\vcenter{\hbox{\includegraphics[scale=0.6]{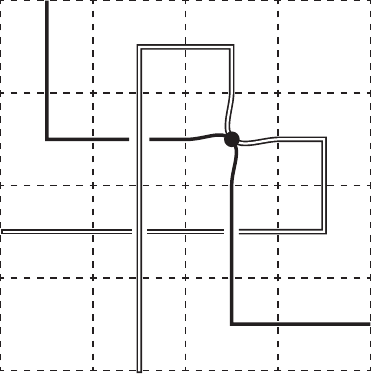}}}
}
\end{align*}
\begin{align*}
\xymatrix{
t_{\bullet}^W=\vcenter{\hbox{\includegraphics[scale=0.7]{grid_tile_Leg_singularcrossing_NW.pdf}}}\quad\ar[r]^-{(Swirl)}&
\quad\vcenter{\hbox{\includegraphics[scale=0.6]{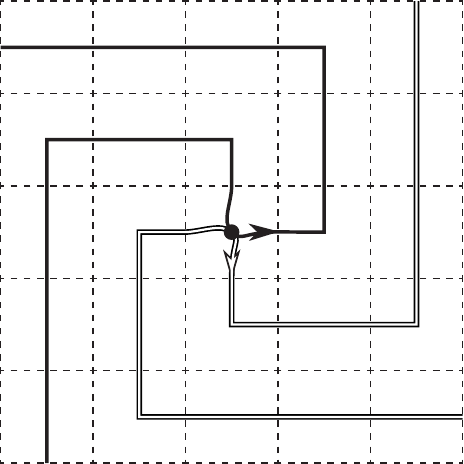}}}
}
\end{align*}
\begin{align*}
\xymatrix{
\vcenter{\hbox{\includegraphics[scale=0.7]{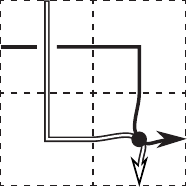}}}\quad\ar@{<->}[r]^{(Flype)}&
\quad\vcenter{\hbox{\includegraphics[scale=0.7]{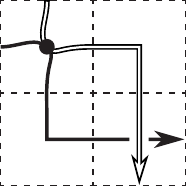}}}
}
\end{align*}
\caption{Rotations, swirl and flype}
\label{fig:rot_sw_fl}
\end{figure}

Two rotations $(Rot_+)$ and $(Rot_-)$ change the direction of the singular tile.
It can be easily verified by replacing $t_\bullet$ in Figure~\ref{fig:rot_sw_fl} with $t_\bullet^*$ for $*\in\{N,E,W,S\}$.
Especially, $(+)$ and $(-)$-rotations from $t_\bullet^S$ and $t_\bullet^N$ to $t_\bullet^E$ are denoted by $(Rot_+^*)$ and $(Rot_-^*)$, respectively.

Although $(Rot_\pm)$ do not depend on orientations given on $t_\bullet$, both $(Swirl)$ and $(Flype)$ depend on the orientation. That is, if orientations are not the same as given in Figure~\ref{fig:rot_sw_fl}, then neither $(Swirl)$ nor $(Flype)$ is applicable.
The diagram in Figure~\ref{fig:directions_of_t} summarizes the above discussion.

\begin{figure}[ht]
\[
\xymatrix@C=10pc@R=4pc{
t_\bullet^W\subset G_W\ar@<-.5ex>[r]_{(Rot_-)}\ar@<2ex>[d]_{(Rot_+)}\ar[rd]^*[@]=0+!D{\labelstyle(Swirl)} & 
G_N\supset t_\bullet^N\ar@<-3ex>[d]_{(Rot_-^*)} \ar@<-.5ex>[l]_{(Rot_+)}\\
t_\bullet^S\subset G_S\ar@<-3ex>[u]_{(Rot_-)}\ar@<-.5ex>[r]_{(Rot_+^*)} & 
G_E\supset t_\bullet^E\ar@<-.5ex>[l]_{(Rot_-)} \ar@<2ex>[u]_{(Rot_+)}\ar@<3ex>@(ld,rd)_{(Flype)}
}
\]
\caption{A way how moves $(Rot_\pm), (Swirl)$ and $(Flype)$ change the directions of $t_\bullet^*$}
\label{fig:directions_of_t}
\end{figure}
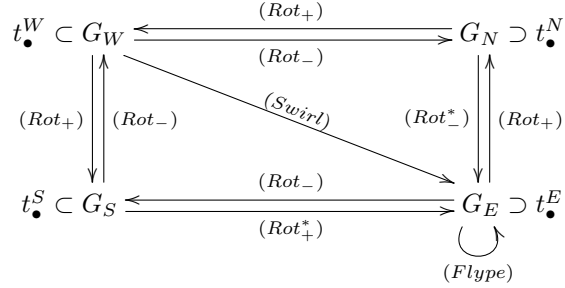

\begin{rmk}\label{rmk:rotations}
It is important to note that two rotations $(Rot_+)$ and $(Rot_-)$ are not inverse to each other.
Furthermore three moves $(Rot_+)^2$, $(Rot_-)^2$ and $(Swirl)$ are independent in the sense that any one of three moves can not be generated by others unless (de)stabilizations are allowed.
For example, a braid coming from $(Rot_+)^2(t_\bullet^W)$ or $(Rot_-)^2(t_\bullet^W)$ via $\|\cdot\|_\SB$ can be obtained from $(Swirl)(t_\bullet^W)$ with one positive braid stabilization $(BrS_+)$. 

However, all three coincide in $\ST$ since $(BrS_+)$ induces an equivalence in $\ST$.
\end{rmk}

\subsection{Invariances}
Now we consider the invariances of the maps $\|\cdot\|_\mathcal{SC}:\SG\to\mathcal{SC}$ for $\mathcal{C}\in\{\B,\L,\T,\K\}$ under (i) translations and commutations, (ii) (de)stabilizations, and (iii) rotations, swirl and flype.

\begin{prop}[Invariance under translations and commutations]\label{prop:basicmoves}
For any $\mathcal{C}\in\{\B,\L,\T,\K\}$, the map $\|\cdot\|_\mathcal{SC}$ is invariant under both translations and commutations.
\end{prop}
\begin{proof}
Since $\|\cdot\|_\ST$ and $\|\cdot\|_\SK$ factor through both $\|\cdot\|_\SB$ and $\|\cdot\|_\SL$, it suffices to prove the invariances for $\|\cdot\|_\SB$ and $\|\cdot\|_\SL$ under translations and commutations.

Notice that under the map $\|\cdot\|_\SL$, both commutations $(Cm_H)$ and $(Cm_V)$ are one of Legendrian Reidemeister moves $(LRM2)$, $(LRM3)$ or $(LRM5)$, and both translations $(Tr_H)$ and $(Tr_V)$ can be expressed as compositions of Legendrian translations $(LTr_H)$ and $(LTr_V)$ defined in Corollary~\ref{cor:LegTr} up to planar isotopy.
Therefore $\|\cdot\|_\SL$ is invariant under translations and commutations.

For $\SB$, by regarding a grid diagram as drawn on a torus, it is straightforward that if $G'\in\SG$ is obtained from $G\in\SG$ by $(Tr_H)$ and $(Tr_V)$, then $\|G\|_{\SB}$ and $\|G'\|_{\SB}$ are the same up to conjugacy $(BrC)$. 
See Figure~\ref{fig:HTransBraid} for $(Tr_H)$ and Figure~\ref{fig:VTransBraid} for $(Tr_V)$.
Moreover, horizontal commutation $(Cm_H)$ yields braid isotopy and it defines the same braids in $\mathbf{SB}$. Therefore the only case we concern is vertical commutation $(Cm_V)$.

\begin{figure}[ht]
\[
\xymatrix@R=2pc{
\vcenter{\hbox{\includegraphics[scale=0.7]{Translation_H_1.pdf}}}\ar@{<->}[d]_{(Tr_H)}\quad\ar@{|->}[r]^-{(\cdot)^\vee}& \quad\vcenter{\hbox{\includegraphics[scale=0.7]{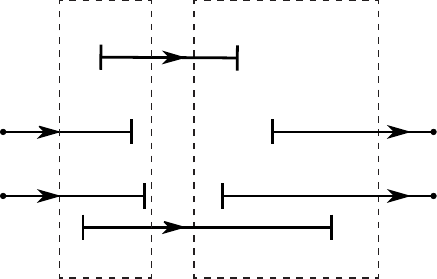}}} \ar@{|->}[r]^-{sl}& \beta_1\beta_2\ar@{<->}[d]^{(BrC)}\\
\vcenter{\hbox{\includegraphics[scale=0.7]{Translation_H_2.pdf}}}\quad\ar@{|->}[r]^-{(\cdot)^\vee}& \quad\vcenter{\hbox{\includegraphics[scale=0.7]{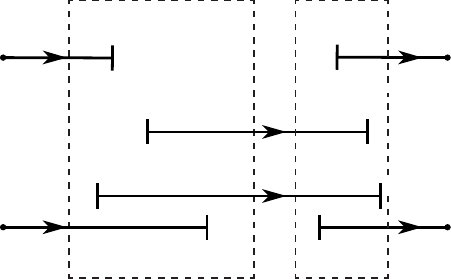}}} \ar@{|->}[r]^-{sl}&\beta_2\beta_1
}
\]
\caption{A horizontal translation $(Tr_H)$ under $\|\cdot\|_\SB=sl\circ (\cdot)^\vee$}
\label{fig:HTransBraid}
\end{figure}
\begin{figure}[ht]
\[
\xymatrix@C=2pc@R=1pc{
\vcenter{\hbox{\includegraphics[scale=0.7]{Translation_V_1.pdf}}}\ar@{<->}[dd]_{(Tr_V)}\quad\ar@{|->}[r]^-{(\cdot)^\vee}&\quad\vcenter{\hbox{\includegraphics[scale=0.7]{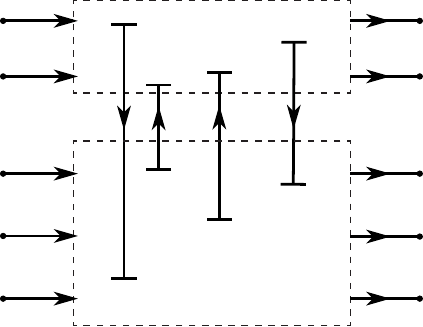}}}\quad\ar@{|->}[rr]^-{sl} & &\beta\ar@{<->}[d]^{(BrC)}\\
& \hspace{-1cm}\qquad\qquad\vcenter{\hbox{\includegraphics[scale=0.7]{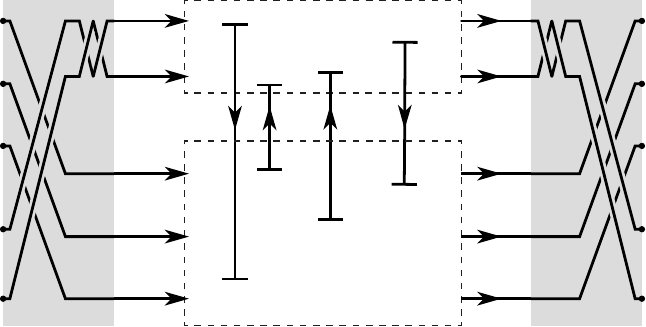}}}\quad\ar@{=}[d]\ar@{|->}[rr]^-{sl}& &
\gamma^{-1}\beta\gamma\ar@{=}[d]\\
\vcenter{\hbox{\includegraphics[scale=0.7]{Translation_V_2.pdf}}}\quad\ar@{|->}[r]^-{(\cdot)^\vee}& \quad\vcenter{\hbox{\includegraphics[scale=0.7]{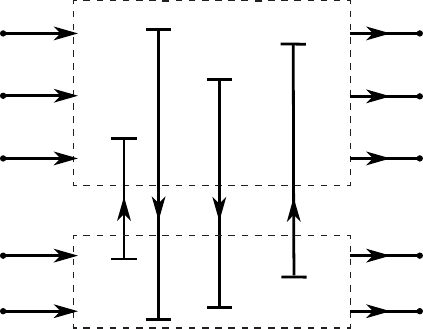}}}\quad\ar@{|->}[rr]^-{sl}&  &\beta'
}
\]
\caption{A vertical translation $(Tr_V)$ under $\|\cdot\|_\SB=sl\circ (\cdot)^\vee$}
\label{fig:VTransBraid}
\end{figure}

Let $G\in\SG$ and let $r$ and $r'$ be the non-interleaving contiguous rows of $G$. Then by commutation of $r$ and $r'$, we obtain $G'\in \SG$.
If none of $r$ and $r'$ contain backward horizontal segments, then the vertical commutation defines the same braids in $\mathbf{SB}$ and so $\|G\|_\SB=\|G'\|_\SB\in\SB$.

On the other hand, if one of $r$ and $r'$ contains backward horizontal segments, this is the case when the exchange move $(BrE)$ occurs. This is essentially the same as described in \cite[Proposition~7]{NT} but possibly the sequence of exchange moves and braid isotopies are needed when $r$ or $r'$ contains singular points as Figure~\ref{fig:Cm_V}.

The other cases are essentially same as above pictures and the detail will be omitted.
\end{proof}

\begin{figure}[ht]
\[
\xymatrix{
\vcenter{\hbox{\includegraphics[scale=0.8]{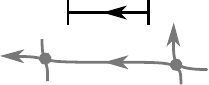}}}\ar@{<->}[dd]_{(Cm_V)}\ar@{|->}[r]^-{(\cdot)^\vee} &
\vcenter{\hbox{\includegraphics[scale=0.8]{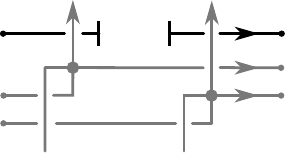}}}\ar@{<->}[r]^{(BrF)}&
\vcenter{\hbox{\includegraphics[scale=0.8]{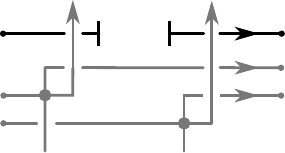}}}\ar@{<->}[r]^{(BrE)}&
\vcenter{\hbox{\includegraphics[scale=0.8]{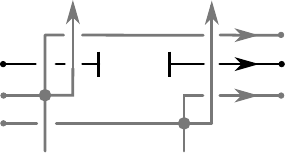}}}\ar@{<->}[d]^{(BrF)}_{(Br2)}\\
& & & \vcenter{\hbox{\includegraphics[scale=0.8]{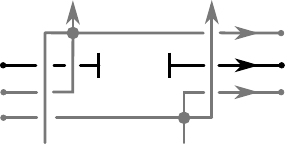}}}\ar@{<->}[d]^{(BrE)}\\
\vcenter{\hbox{\includegraphics[scale=0.8]{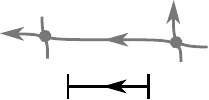}}}\ar@{|->}[r]^-{(\cdot)^\vee} & 
\vcenter{\hbox{\includegraphics[scale=0.8]{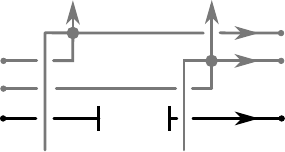}}}\ar@{<->}[r]^{(BrE)}&
\vcenter{\hbox{\includegraphics[scale=0.8]{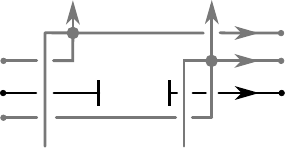}}}\ar@{<->}[r]^{(Br2)}_{(BrF)}&
\vcenter{\hbox{\includegraphics[scale=0.8]{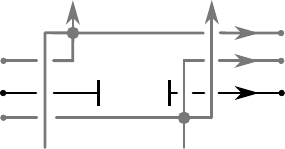}}}\\
}
\]
\caption{A vertical commutation $(Cm_V)$ under $\|\cdot\|_\SB$}
\label{fig:Cm_V}
\end{figure}

\begin{prop}[Invariance under (de)stabilizations]\label{prop:invariantunderstabilizations}
The following holds.
\begin{enumerate}
\item $\|\cdot\|_\SB$ is invariant under $\{(NE), (SE)\}$.
\item $\|\cdot\|_\SL$ is invariant under $\{(NE), (SW)\}$.
\item $\|\cdot\|_\ST$ is invariant under $\{(NE), (SE), (SW)\}$.
\item $\|\cdot\|_\SK$ is invariant under $\{(NE), (NW), (SE), (SW)\}$.
\end{enumerate}
\end{prop}
This is exactly the same as nonsingular cases considered in \cite{C, NT, OST} since these stabilizations do not involve any singular tile, and so we omit the proof.


\begin{prop}[Invariance under rotations, swirl and flype]\label{prop:rotations}The following holds.
\begin{enumerate}
\item $\|\cdot\|_\SB$ is invariant under $(Flype)$, $(Swirl)$ and $(Rot_\pm^*)$.
\item $\|\cdot\|_\SL$ is invariant under $(Rot_\pm)$.
\item $\|\cdot\|_\ST$ and $\|\cdot\|_\SK$ are invariant under $(Flype)$, $(Swirl)$ and $(Rot_\pm)$.
\end{enumerate}
\end{prop}

As mentioned in Remark~\ref{rmk:rotations}, the move $(Swirl)$ is the same as $(Rot_\pm)^2$ in $\ST$ and therefore in $\SK$ as well
\begin{proof}
\noindent (1) At first, we need to check that the flips $\left((Rot_+^*)(t_\bullet^S)\right)^\vee$, $\left((Rot_-^*)(t_\bullet^N)\right)^\vee$ and $\left((Swirl)(t_\bullet^W)\right)^\vee$.
\[
\left((Rot_+^*)(t_\bullet^S)\right)^\vee=\vcenter{\hbox{\includegraphics{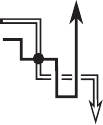}}}=\vcenter{\hbox{\includegraphics{flip_SW.pdf}}}=(t_\bullet^S)^\vee\]
\[
\left((Rot_-^*)(t_\bullet^N)\right)^\vee=\vcenter{\hbox{\includegraphics{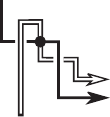}}}=\vcenter{\hbox{\includegraphics{flip_NE.pdf}}}=(t_\bullet^N)^\vee\]
\[
\left((Swirl)(t_\bullet^W)\right)^\vee=\vcenter{\hbox{\includegraphics{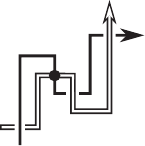}}}=\vcenter{\hbox{\includegraphics{flip_NW.pdf}}}=(t_\bullet^W)^\vee
\]

Since it is obvious that $\|\cdot\|_\SB$ is invariant under $(Flype)$, and so we are done.

\noindent (2) It is not hard to see that $(Rot_+)$ and $(Rot_-)$ are inverses to each other in $\SL$ and correspond to the singular Legendrian Reidemeister move $(LRM6)$.

\noindent (3) Since $\hat{(\|\cdot\|_\SB)}_\T=(\|\cdot\|_\SL)^+$, the map $\|\cdot\|_\ST$ is invariant under $(Flype), (Swirl)$ and $(Rot_\pm)$. 

For $\|\cdot\|_\SK$, this is obvious since $\|\cdot\|_\SK = \big\|\|\cdot\|_\ST\big\|$.
\end{proof}

\subsection{Proof of main theorem}

From now on, we denote $\SG$ modulo translations and commutations by $\tilde\SG$
\[
\tilde\SG=\SG/\{(Tr_H), (Tr_V),(Cm_H),(Cm_V)\},
\]
and for each $\mathcal{C}\in\{\B, \L, \T, \K\}$, we regard the map $\|\cdot\|_{\SC}$ as defined on $\tilde\SG$.

\begin{prop}\label{prop:SGSB}
The map $\|\cdot\|_\SB$ induces the bijection
\[
\|\cdot\|_\SB:\tilde\SG/\{(NE), (SE), (Flype), (Swirl), (Rot_\pm^*)\}\to\SB.
\]
\end{prop}
\begin{proof}
Let $\SG^E$ be the subset of $\SG$ consisting of singular grid diagrams extended with only the singular tile $t_\bullet^E$,
and let
\[
\tilde{\SG^E}=\SG^E/\{(Tr_H), (Tr_V),(Cm_H),(Cm_V)\}\subset\tilde\SG.
\]
Then the quotient map by $(Swirl)$ and $(Rot_\pm^*)$ induces the maps $r:\SG\to\SG^E$
and $\tilde r:\tilde\SG\to\tilde{\SG^E}$ which are retractions.
Therefore it suffices to show that $\|\cdot\|_\SB$ induces the bijection
\[
\tilde{\SG^E}/\{(NE),(SE),(Flype)\}\to\SB,
\]
which is well-defined by Proposition~\ref{prop:basicmoves}, Proposition~\ref{prop:invariantunderstabilizations}~(1) and Proposition~\ref{prop:rotations}~(1).

The rest of the proof is essentially the same as the proof of \cite[Proposition~7~(3)]{NT}, and we will construct an inverse map as follows. Let $\beta\in\SB$ be given by a word on $\{\sigma_i,\sigma_i^{-1},\xi_i|1\le i<n\}$.
Then we have a rectilinear diagram $\bar\beta\in\bar{\mathbf{SB}}$ obtained from $\beta$ as described in Section~\ref{sec:rectilinear}.

By perturbing horizontal segments of $\bar\beta$ slightly, we obtain another rectilinear diagram $\bar\beta'$ so that no horizontal and vertical segments are colinear.
We close $\bar\beta'$ in the {\em rectilinear way} as shown in Figure~\ref{fig:rectilinear_closure} to obtain a singular grid diagram $G(\beta)\in \SG$. Then it is straightforward that $\|G(\beta)\|_\SB = \beta$. Therefore $\|\cdot\|_\SB$ is surjective.

\begin{figure}[ht]
\[
\xymatrix{
\vcenter{\hbox{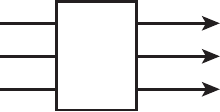}}\ar@{=>}[r]^-{Perturb}\quad&\quad
\vcenter{\hbox{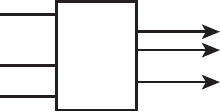}}\ar@{=>}[r]^-{Close}\quad&\quad
\vcenter{\hbox{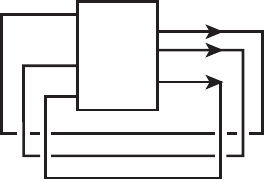}}\in \SG
}
\]
\caption{A rectilinear closure $G(\beta)$ of $\bar\beta'$}
\label{fig:rectilinear_closure}
\end{figure}

Moreover, all equivalence relations of $\mathbf{SB}$ together with conjugations and exchange moves are coming from translations and commutations up to $(NE), (SE)$ and $(Flype)$.
\begin{align*}
(Br0) &\equiv (Cm_V)&
(Br1) &\equiv (Cm_H)&
(Br2) &\equiv (Cm_H)\text{ or }(CM_V)\\
(BrF) &\equiv (Flype) &
(BrC) &\equiv (Tr_H) &
(BrE) &\equiv (Tr_V)
\end{align*}

Therefore if $\beta_1=\beta_2\in\SB$, then $G(\beta_1)$ and $G(\beta_2)$ are equivalent in $\tilde{SG}$ up to $(NE), (SE)$ and $(Flype)$. This implies the injectivity of $\|\cdot\|_\SB$ and completes the proof.
\end{proof}

\begin{prop}\label{prop:SGSL}
The map $\|\cdot\|_\SL$ induces the bijection
\[
\|\cdot\|_\SL:\tilde\SG/\{(NE), (SW), (Rot_\pm)\}\to\SL.
\]
\end{prop}
\begin{proof}
It is obvious that $\|\cdot\|_\SL$ is surjective. Therefore for the injectivity, it suffices to show that each of the moves $(LRM1)\sim(LRM6)$ comes from (a sequence of) moves $(NE)$, $(SW)$ and $(Rot_\pm)$.
Notice that the moves $(LRM1)\sim(LRM6)$ can be viewed as compositions of moves on a grid diagram and the map $\|\cdot\|_\SL$ as follows.
\begin{align*}
(LRM1)=\vcenter{\hbox{\includegraphics{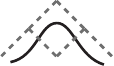}}}
&\longleftrightarrow\vcenter{\hbox{\includegraphics{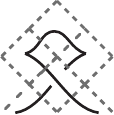}}}
=\|(NE)\|_\SL \text{ or }\|(SW)\|_\SL\\
(LRM2)=\vcenter{\hbox{\includegraphics{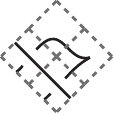}}}
&\longleftrightarrow\vcenter{\hbox{\includegraphics{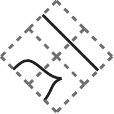}}}
=\|(Cm_*)\|_\SL\\
(LRM3)=\vcenter{\hbox{\includegraphics{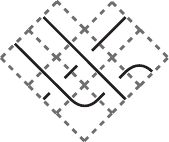}}}
&\longleftrightarrow\vcenter{\hbox{\includegraphics{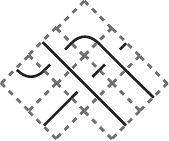}}}
=\|(Cm_*)\|_\SL\\
(LRM4)=\vcenter{\hbox{\includegraphics{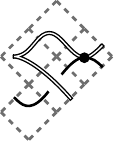}}}
&\longleftrightarrow\vcenter{\hbox{\includegraphics{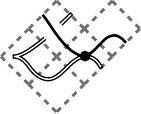}}}
=\|(Rot_\pm)\|_\SL\\
(LRM5)=\vcenter{\hbox{\includegraphics{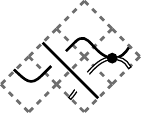}}}
&\longleftrightarrow\vcenter{\hbox{\includegraphics{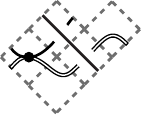}}}
=\|(Cm_*)\|_\SL\\
(LRM6)=\vcenter{\hbox{\includegraphics{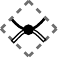}}}
&\longleftrightarrow\vcenter{\hbox{\includegraphics{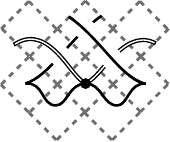}}}
=\|(Rot_\pm)\|_\SL
\end{align*}

Therefore $L=\|G\|_\SL$ and $L'=\|G'\|_\SL$ in $\SL$ are equivalent if and only if $G$ and $G'$ in $\tilde\SG$ are the same up to $(NE)$, $(SW)$ and $(Rot_\pm)$ as desired.
\end{proof}

We give a proof of the main theorem as follows.

\begin{proof}[Proof of Theorem~\ref{thm:maintheorem}]
The first and second statements of the theorem about $\SB$ and $\SL$ are already proved in Proposition~\ref{prop:SGSB} and in Proposition~\ref{prop:SGSL}, respectively.

Moreover, by Proposition~\ref{prop:SLSK}, Theorem~\ref{thm:SLST} and Theorem~\ref{thm:STSK}, we have the following commutative diagram.
\[
\xymatrix@C=6pc@R=4pc{
\tilde\SG\ar[r]^-{\|\cdot\|_\SL}\ar[dr]_-{\|\cdot\|_\ST}&\SL \ar[dr]^{/\{(LS_\pm),(LF)\}}\ar[d]|{/\{(LS_-),(LF)\}}\\
& \ST\ar[r]_-{/\{(TS)\}} & \SK.
}
\]

Since $(LS_+)$, $(LS_-)$, $(TS)$ and $(LF)$ correspond to $(NW)$, $(SE)$, $(NW)$ and $(Flype)$ by Proposition~\ref{prop:stabilizations}, respectively, we have 
\begin{align*}
\ST &\stackrel{(\cdot)^+}{=\joinrel=\joinrel=\joinrel=\joinrel=\joinrel=} \SL /\{(LS_-),(LF)\}\\
&\stackrel{(\|\cdot\|_\SL)^+}{=\joinrel=\joinrel=\joinrel=\joinrel=\joinrel=}\left(\tilde\SG/\{ (NE), (SW), (Rot_\pm)\}\right)\big/\{(SE), (Flype)\}
\\
&\stackrel{\|\cdot\|_\ST}{=\joinrel=\joinrel=\joinrel=\joinrel=\joinrel=} \tilde\SG/\{(NE), (SE), (SW), (Flype), (Rot_\pm)\}
\end{align*}
and
\begin{align*}
\SK &\stackrel{\|\cdot\|_\SK}{=\joinrel=\joinrel=\joinrel=\joinrel=\joinrel=} \SL /\{(LS_\pm),(LF)\}\\
&\stackrel{\big\|\|\cdot\|_\SL\big\|_\SK}{=\joinrel=\joinrel=\joinrel=\joinrel=\joinrel=} \left(\tilde\SG/\{ (NE), (SW), (Rot_\pm)\}\right)\big/\{(NW), (SE), (Flype)\}\\
&\stackrel{\|\cdot\|_\SK}{=\joinrel=\joinrel=\joinrel=\joinrel=\joinrel=} \tilde\SG/\{(NE), (NW), (SE), (SW), (Flype), (Rot_\pm)\}
\end{align*}
as desired.
\end{proof}

\end{document}